\documentclass[a4paper,11pt]{amsart}
\usepackage{amssymb,amscd,amsxtra,graphicx,psfrag,xypic}
\usepackage{array}

\usepackage[pdftex, %
bookmarks=true, %
bookmarksnumbered=true, %
pdftitle={}, %
pdfauthor={Takuzo Okada}, %
colorlinks=true]{hyperref}

\theoremstyle{plain}
\newtheorem{Thm}{Theorem}[section]
\newtheorem{Lem}[Thm]{Lemma}
\newtheorem{Cor}[Thm]{Corollary}
\newtheorem{Prop}[Thm]{Proposition}

\theoremstyle{definition}
\newtheorem{Def}[Thm]{Definition}
\newtheorem{Def-Lem}[Thm]{Definition-Lemma}
\newtheorem{Cond}[Thm]{Condition}
\newtheorem{Rem}[Thm]{Remark}

\newtheorem*{Ack}{Acknowledgments}
\newtheorem{Ex}[Thm]{Example}
\newtheorem{Assumption}[Thm]{Assumption}

\newcommand{\prt}{\partial}

\newcommand{\Sing}{\operatorname{Sing}}
\newcommand{\Spec}{\operatorname{Spec}}

\newcommand{\rank}{\operatorname{rank}}

\newcommand{\Div}{\operatorname{Div}}

\newcommand{\Int}{\operatorname{Int}}
\newcommand{\bNE}{\operatorname{\overline{NE}}}
\newcommand{\Bs}{\operatorname{Bs}}

\newcommand{\Exc}{\operatorname{Exc}}
\newcommand{\mult}{\operatorname{mult}}
\newcommand{\dP}{\operatorname{dP}}

\newcommand{\lcm}{\operatorname{lcm}}

\newcommand{\wt}{\operatorname{wt}}

\newcommand{\Wbl}{\operatorname{Wbl}}

\newcommand{\Cox}{\operatorname{Cox}}
\newcommand{\reduced}{\operatorname{red}}

\newcommand{\bcap}{\bigcap\nolimits}

\newcommand{\mbA}{\mathbb{A}}

\newcommand{\mbC}{\mathbb{C}}

\newcommand{\mbP}{\mathbb{P}}
\newcommand{\mbQ}{\mathbb{Q}}
\newcommand{\mbR}{\mathbb{R}}

\newcommand{\mbZ}{\mathbb{Z}}

\newcommand{\mcG}{\mathcal{G}}
\newcommand{\mcH}{\mathcal{H}}
\newcommand{\mcI}{\mathcal{I}}

\newcommand{\mcL}{\mathcal{L}}

\newcommand{\mcN}{\mathcal{N}}
\newcommand{\mcO}{\mathcal{O}}

\newcommand{\mcS}{\mathcal{S}}

\newcommand{\msp}{\mathsf{p}}
\newcommand{\msq}{\mathsf{q}}

\newcommand{\ratmap}{\dashrightarrow}

\newcommand{\bmu}{\boldsymbol{\mu}}

\title[$\mbQ$-Fano threefold weighted complete intersections]{Birational Mori fiber structures of  $\mbQ$-Fano $3$-fold weighted complete intersections}
\author[Takuzo Okada]{Takuzo Okada}
\address{Department of Mathematics, Faculty of Science and Engineering\endgraf
Saga University, Saga 840-8502 Japan}
\email{okada@cc.saga-u.ac.jp}
\date{}

\begin{document}

\begin{abstract}
In this paper we consider a $\mbQ$-Fano $3$-fold weighted complete intersection of codimension $2$ in the $85$ families listed in the Iano-Fletcher's list and determine which cycle is a maximal center or not.
For each maximal center, we construct either a birational involution which untwists the maximal singularity or a Sarkisov link centered at the cycle to an another explicitly described Mori fiber space.
As a consequence, $19$ families are proved to be birationally rigid and the remaining $66$ families are proved to be birationally nonrigid.
\end{abstract}

\maketitle

\section{Introduction}

Iskovskikh and Manin studied birational maps between smooth quartic threefolds and proved their nonrationality in \cite{IM}.
Since then the method introduced in the paper, which is nowadays called the method of maximal singularities, has been developed and applied to many $\mbQ$-Fano varieties or more generally Mori fiber spaces, and they are proved to be birationally rigid.

\begin{Def}
Let $X$ be a $\mbQ$-Fano variety with at most terminal singularities and with Picard number one.
We say that $X$ is {\it birationally rigid} if for every Mori fiber space $Y \to S$, the existence of a birational equivalence $\varphi \colon X \ratmap Y$ implies that $Y \cong X$.
\end{Def}

On the other hand, when one considers the rationality problem of algebraic varieties, birational rigidity seems to be too strong to conclude nonrationality.
Corti-Mella \cite{CM} shows that a general weighted complete intersection $X_{3,4}$ of weighted hypersurfaces of degree $3$ and $4$ in $\mbP (1,1,1,1,2,2)$ is not birationally rigid but birationally birigid, that is, $X_{3,4}$ has exactly two structures of Mori fiber spaces in its birational equivalence class, which is still enough to conclude nonrationality.
This tells us an importance of an intermediate notion, finiteness of birational Mori fiber structures.
Here, a {\it birational Mori fiber structure} of an algebraic variety $X$ is a Mori fiber space birational to $X$.


Corti-Pukhlikov-Reid \cite{CPR} proved birational rigidity of anticanonically embedded $\mbQ$-Fano $3$-fold weighted hypersurfaces.
Recently, Cheltsov-Park \cite{CP} extended the above result for every quasismooth members.
Brown-Zucconi \cite{BZ} and Ahmadinezhad-Zucconi \cite{AZ} proved birational non-rigidity of suitable codimension three $\mbQ$-Fano threefolds in weighted projective spaces.
We investigate in this paper the birational Mori fiber structures of anticanonically embedded $\mbQ$-Fano $3$-fold weighted complete intersections of codimension two.
There is a list \cite[16.7]{IF}, which we refer to as the Fletcher's list, of such $\mbQ$-Fano WCIs and it is proved to be the complete one in \cite{CCC00}.
They consist of $85$ families.
Among them, $3$ families have been studied in this direction.
Iskovskikh-Pukhlikov \cite{IP} proved birational rigidity of a general complete intersection $X_{2,3} \subset \mbP^5$ of a quadric and a cubic and, as mentioned above, Corti-Mella \cite{CM} proved birational birigidity of a general weighted complete intersection $X_{3,4} \subset \mbP (1,1,1,1,2,2)$ of a cubic and a quartic.
Grinenko \cite{Grinenko} studied a general weighted complete intersection $X_{3,3} \subset \mbP (1,1,1,1,1,2)$ of cubics which has a Sarkisov link to a fibration into cubic surfaces over $\mbP^1$.

In this paper, we treat the remaining families.
We set $I := \{1,2,\dots,85\}$ which is the set of numbers of $85$ families, where the numbering of the families is the one given in \cite{IF}.
We divide $I$ into the disjoint union of the following three subsets
\[
\begin{split}
I_{br} &:= \{ 1,8,14,20,24,31,37,45,47,51,59,60,64,71,75,76,78,84,85 \}, \\
I_{dP} &:= \{ 2,4,5,11,12,13 \}, \\
I_F &:= I \setminus (I_{br} \cup I_{dP}),
\end{split}
\]
where $|I_{br}| = 19$, $|I_{dP}| = 6$ and $|I_F| = 60$.
As explained above, the families No.~1, 2 and 3 have been studied by \cite{IP}, \cite{CM} and \cite{Grinenko}, respectively, and we do not treat them in this paper.
We set $I^* := I \setminus \{1,2,3\}$, $I^*_{br} := I_{br} \setminus \{1\}$, $I^*_{dP} := I_{dP} \setminus \{2\}$ and $I^*_F := I_F \setminus \{3\}$. 

\begin{Thm} \label{mainthmbr}
A general member of the family No.~$i$ is birationally rigid if and only if $i \in I_{br}$.
\end{Thm}

As a direct consequence of Theorem \ref{mainthmbr}, we have the following result.

\begin{Cor}
A general member of a family No.~$i$ with $i \in I_{br}$ is nonrational. 
\end{Cor}

\begin{Thm} \label{mainthmF}
Let $X$ be a general member of the family No.~$i$ with $i \in I^*_F$ $($resp.\ $i \in I^*_{dP}$$)$.
Then there is a $\mbQ$-Fano $3$-fold weighted hypersurface $X'$ $($resp.\ a del Pezzo fiber space $X'/\mbP^1$ over $\mbP^1$$)$ which is birational $($but not isomorphic$)$ to $X$.
Moreover, no curve and no nonsingular point on $X$ is a maximal center, and for each singular point $\msp$ of $X$, exactly one of the following holds.
\begin{enumerate}
\item $\msp$ is not a maximal center.
\item There is a birational involution $\iota_{\msp} \colon X \ratmap X$ which is a Sarkisov link centered at $\msp$.
\item There is a Sarkisov link $\sigma_{\msp} \colon X \ratmap X'$ centered at $\msp$.
\end{enumerate}
\end{Thm}

The generality assumptions in Theorems \ref{mainthmbr} and \ref{mainthmF} will be made explicit in Section \ref{secprelim}.
We emphasize here that we have not determined the birational Mori fiber structures of members of the family No.~$i$ for $i \in I^*_F \cup I^*_{dP}$ because $X'$ in Theorem \ref{mainthmF} may admit Sarkisov links to another Mori fiber spaces.
Nevertheless Theorem \ref{mainthmF} is an important step toward the determination of the birational Mori fiber structures of those $\mbQ$-Fano varieties.
To complete this, we need to work with $X'$ and obtain a similar result for $X'$.

\begin{Ack}
The author is partially supported by JSPS KAKENHI Grant Number 24840034.
The author would like to thank the referee for pointing out many errors and useful suggestions.
\end{Ack}


\section{Notation and conventions}

Throughout the paper, we work over the field $\mbC$ of complex numbers.
A normal projective variety $X$ is said to be a $\mbQ$-{\it Fano variety} if it is $\mbQ$-factorial, has only terminal singularities and its Picard number is one.
We say that an algebraic fiber space $X \to S$ is a {\it Mori fiber space} if $X$ is a normal projective $\mbQ$-factorial variety with at most terminal singularities, $\dim S < \dim X$, the anticanonical divisor of $X$ is relatively ample over $S$ and the relative Picard number is one.

For positive integers $a_0,\dots,a_n$ and $m_1,\dots,m_n$, we define
\[
\mbP (a_0^{m_0},\dots,a_n^{m_n}) 
:= \mbP (\overbrace{a_0,\dots,a_0}^{m_0},\dots,\overbrace{a_n,\dots,a_n}^{m_n}).
\]
We refer the readers to \cite{IF} for the definition and basic properties of weighted projective spaces and weighted projective complete intersections.

By a $\mbQ$-{\it Fano weighted complete intersection} ($\mbQ$-{\it Fano WCI}, for short), we mean a quasismooth anticanonically embedded $\mbQ$-Fano weighted complete intersection of dimension $3$ and codimension $2$ with only terminal singularities.
There are $85$ families of $\mbQ$-Fano WCIs in the Fletcher's list and each family is determined by the weights $a_0,a_1,\dots,a_5$ of the ambient weighted projective space and degrees $d_1$, $d_2$ of the defining polynomials.

Let $X$ be a $\mbQ$-Fano WCI in $\mbP := \mbP (a_0,\dots,a_5)$.
We usually write $x_0,x_1,\dots,x_5$ (resp.\ $x,y,z,\dots$) for the homogeneous coordinates when we treat several families at a time (resp.\ a specific family).
For example, the ambient weighted projective space of the family No.\ $6$ is $\mbP (1^3,2^2,3)$ and we use homogeneous coordinates $x_0,x_1,x_2,y_0,y_1$ and $z$.
We denote by $\msp_j$ the vertex $(0:\cdots:1:\dots:0)$, where the $1$ is in the $(j+1)$-th position.
Throughout the paper, we set $A := -K_X$, which is a generator of the Weil divisor class group of $X$.
The degree $\deg \Gamma$ of a curve $\Gamma \subset X$ is the degree with respect to $A$, that is, $\deg \Gamma = (A \cdot \Gamma)$.
We define $\pi_j \colon X \ratmap \mbP (a_0,\dots,\hat{a}_j,\dots,a_5)$ to be the projection from the point $\msp_j$, which is defined possibly outside $\msp_j$.
More generally, if $\msp \in \mbP$ is a point which can be mapped to a vertex $\msp_j$ by an automorphism of $\mbP$, then we can define the projection $\pi_{\msp}$ from the point $\msp$ in the same way.
It is generically finite and we define $\Exc (\pi_j) \subset X$ (resp.\ $\Exc (\pi_{\msp})$) to be the locus contracted by $\pi_j$ (resp.\ $\pi_{\msp}$).
For homogeneous polynomials $f_1,\dots,f_m$ in the variables $x_0,\dots,x_5$ (or $x,y,z,\dots$), we denote by $(f_1 = \cdots = f_m = 0)$ the closed subscheme of $\mbP$ defined by the homogeneous ideal $(f_1,\dots,f_m)$ and denote by $(f_1 = \cdots = f_m = 0)_X := (f_1 = \cdots = f_m = 0) \cap X$ the scheme-theoretic intersection.
For a polynomial $f = f (x_0,\dots,x_n)$ and a monomial $x_0^{c_0} x_1^{c_1} \cdots x_n^{c_n}$ of degree $\deg f$, we write $x_0^{c_0} x_1^{c_1} \cdots x_n^{c_n} \in f$ (resp.\ $x_0^{c_0} x_1^{c_1} \cdots x_n^{c_n} \notin f$) if the coefficient of the monomial in $f$ is nonzero (resp.\ zero).

Throughout this paper, the table in Section \ref{sec:table} is referred to as {\it the big table}.

\begin{Def}
Let $X$ be a member of the family No.~$i$ with $i \in I^*_F$ (resp.\ $i \in I^*_{dP}$).
A singular point of $X$ is said to be a {\it distinguished singular point} if it is a point marked $X'_d \subset \mbP (b_0,\dots,b_4)$ for some $b_0,\dots,b_4$ and $d$ (resp.\ $\operatorname{dP}_k$ for some $k$) in the third column of the big table.
\end{Def} 

\section{Preliminaries} \label{secprelim}

\subsection{Generality conditions}

We introduce and explain generality conditions for $\mbQ$-Fano WCIs which we assume throughout the paper.
A {\it weighted complete intersection curve} ({\it WCI curve}, for short) of type $(c_1,c_2,c_3,c_4)$ in $\mbP (a_0,\dots,a_5)$ is an irreducible and reduced curve defined by four homogeneous polynomials of degree respectively $c_1, c_2, c_3$ and $c_4$.

\begin{Cond} \label{bcd}
For a member $X = X_{d_1,d_2} \subset \mbP (a_0,\dots,a_5)$, $a_0 \le a_1 \le \cdots \le a_5$, of the family No.\ $i$ with $i \in I^*$, we introduce the following conditions.
\begin{enumerate}
\item[($\mathrm{C}_0$)] $X$ is quasismooth.
\item[($\mathrm{C}_1$)] The monomial in Table \ref{table:monomial} appear in one of the defining polynomials of $X$ with non-zero coefficient.
\item[($\mathrm{C}_2$)] $X$ does not contain any WCI curve listed in Table \ref{table:WCIcurve}.
\item[($\mathrm{C}_3$)] If $i = 21$ or $51$, then $(x_0 = x_1 = 0)_X$ is an irreducible curve and if  $i = 20$, then the base locus of the pencil $|\mcI_{\msp,X} (2 A)|$ is an irreducible curve for each singular point $\msp$ of type $\frac{1}{2} (1,1,1)$.
\item[($\mathrm{C}_4$)] If $i = 4$ (resp.\ $11$), then $X$ satisfies Lemmas \ref{lemexclcdeg1no4} and \ref{lemexclnspexc1} (resp.\ Lemma \ref{lemexclnspexc1}).
If $i = 6$ (resp.\ $9$), then $X$ satisfies Lemma \ref{lemexclnsptno6} (resp.\ Lemma \ref{lem:S_No9})
\end{enumerate} 
\end{Cond}

\begin{table}[h]
\begin{center}
\caption{Monomials}
\label{table:monomial}
\begin{tabular}{cc|cc|cc}
No. & Monomial & No. & Monomial & No. & Monomial \\
\hline
9 & $y z$ & 32 & $y^2 s$ & 48 & $z s$ \\
22 & $z s$ & 33 & $z s$ &57 & $s t$ \\
25 & $y z^2$ & 37 & $s^2$ & 70 & $y^3 t$ \\ 
28 & $z s$ & 39 & $z^2 t$ & & \\
\end{tabular}
\end{center}
\end{table}

\begin{table}[h]
\begin{center}
\caption{Type of WCI curves}
\label{table:WCIcurve}
\begin{tabular}{cc|cc}
No. & Type & No. & Type \\
\hline
6 & (1,1,1,2) & 37 & (1,2,6,8) \\
10 & (1,1,2,3) & 38 & (1,3,4,7) \\
15 & (1,1,2,3) & 39 & (1,3,4,5) \\
16 & (1,1,2,3), (1,1,3,4), (1,1,3,6) & 44 & (1,2,3,5) \\
18 & (1,1,2,3) & 46 & (1,3,4,5) \\
19 & (1,1,4,6) & 47 & (1,4,5,6) \\
22 & (1,1,4,5) & 52 & (1,3,5,8) \\
26 & (1,1,3,4) & 53 & (1,4,5,6) \\
27 & (1,2,3,4), (1,3,3,4), (1,3,4,5) & 57 & (1,2,7,9) \\
31 & (1,2,4,5) & 59 & (1,4,6,7) \\
33 & (1,1,5,6) & 63 & (1,3,5,8) \\
34 & (1,2,3,5) & 71 & (1,7,8,10) \\
35 & (1,3,4,5) & &
\end{tabular}
\end{center}
\end{table}

We explain that the above conditions are generality conditions, that is, $\mbQ$-Fano WCIs in the family No.\ $i$ satisfying the above conditions form a nonempty open subset.
It is clear that ($\mathrm{C}_0$) and ($\mathrm{C}_1$) are generality conditions.

Generality of ($\mathrm{C}_2$) can be checked by counting dimensions.
For example, let $X = X_{4,5} \subset \mbP (1^3,2^2,3)$ be a member of the family No.~$6$.
Then, WCI curves of type $(1,1,1,2)$ in $\mbP (1^3,2^2,3)$ form a $1$-dimensional family and members of the family No.~$6$ containing a given WCI curve of type $(1,1,1,2)$ form a $2$-codimensional subfamily.
It follows that $\mbQ$-Fano WCIs containing a WCI curve of type $(1,1,1,2)$ form a proper closed subset in the family No.~$6$.

The condition $(\mathrm{C}_3)$ will be made explicit in Examples \ref{ex:C3no20}, \ref{ex:C3no21} and \ref{ex:C3no51}.
Finally, the condition $(\mathrm{C}_4)$ is not explicit and its generality will be proved in the corresponding lemma.


\begin{Def}
For $i \in I^*$, we define $\mcG_i$ to be the subfamily of the family No.~$i$ of $\mbQ$-Fano WCIs satisfying ($\mathrm{C}_0$)-($\mathrm{C}_4$) in Condition \ref{bcd}.
\end{Def}

\subsection{Basic definitions and the test class method}

Throughout this subsection, let $X$ be a $\mbQ$-Fano variety whose Weil divisor class group $\Div (X)$ is isomorphic to $\mbZ$ and  we assume that $A := -K_X$ is the positive generator of $\Div (X)$.
Let $\mcH \subset |n A|$ be a movable linear system on $X$, that is, it is a linear system without base divisor.

\begin{Def}
We define the {\it canonical threshold} $c (X, \mcH)$ of $(X, \mcH)$ to be
\[
c (X, \mcH) := \max \{ \lambda \mid K_X + \lambda \mcH \text{ is canonical} \}.
\]
\end{Def}

\begin{Def}
A {\it maximal singularity} of $\mcH$ is an extremal extraction $Y \to X$ in the Mori category having exceptional divisor $E$ with $1/n > c (X,\mcH) = a_E (K_X) / m_E (\mcH)$.
The center of the extraction $Y \to X$ is called a {\it maximal center}.
\end{Def}

The test class method, which we will explain below, will be used to exclude curves, nonsingular points and some singular points as maximal center.

\begin{Lem}[{\cite[Lemma 5.2.1]{CPR}}] \label{CPRlem5.2.1}
Let $\Gamma \subset X$ be a maximal center and $\varphi \colon (E \subset Y) \to (\Gamma \subset X)$ a maximal contraction.
Then the $1$-cycle $(-K_Y)^2 \in N_1 (Y)$ lies in the interior of the Mori cone of $Y$:
\[
(-K_Y)^2 \in \Int \bNE (Y).
\]
\end{Lem}

\begin{Def}
Let $\varphi \colon (E \subset Y) \to (\Gamma \subset X)$ be an extremal divisorial contraction in the Mori category.
A {\it test class} is a nonzero nef class $M \in N^1 (Y)$.
\end{Def}

As a corollary to Lemma \ref{CPRlem5.2.1}, we obtain the following result.

\begin{Cor}[{\cite[Corollary 5.2.2]{CPR}}] \label{criexclmstc}
Let $\varphi \colon (E \subset Y) \to (\Gamma \subset X)$ be an extremal divisorial extraction in the Mori category.
If $(M \cdot (-K_Y)^2) \le 0$ for some test class $M$, then $E$ is not a maximal singularity.
\end{Cor}

\begin{Def}[{\cite[Definition 5.2.4]{CPR}}]
Let $L$ be a Weil divisor class on $X$ and $\Gamma \subset X$ an irreducible subvariety of codimension $\ge 2$.
For an integer $s > 0$, consider the linear system
\[
\mcL_{\Gamma}^s := | \mcI_{\Gamma}^s (s L) |,
\]
where $\mcI_{\Gamma}$ is the ideal sheaf of $\Gamma$.
We say that the class $L$ {\it isolates} $\Gamma$, or is $\Gamma$-{\it isolating}, if
\begin{enumerate}
\item $\Gamma \subset \Bs \mcL_{\Gamma}^s$ is an isolated component for some positive integer $s$; in other words, in a neighborhood of $\Gamma$, the base locus of $\mcL_{\Gamma}^s$ is contained in $\Gamma$ (as a set).
\item In the case $\Gamma$ is a curve, the generic point of $\Gamma$ appears in $\Bs \mcL_{\Gamma}^s$ with multiplicity $1$.
\end{enumerate}
\end{Def}

\begin{Lem}[{\cite[Lemma 5.2.5]{CPR}}]
Suppose that $L$ isolates $\Gamma \subset X$ and let $s$ be as above.
Then, for any extremal divisorial contraction $\varphi \colon (E \subset Y) \to (\Gamma \subset X)$, the birational transform $M =f_*^{-1} \mcL_{\Gamma}^s$ is a test class on $Y$.
\end{Lem}

\subsection{Terminal quotient singularities and the Kawamata blowup}

\begin{Def}
Let $(X, x)$ be a germ of a $3$-fold cyclic quotient singularity.
We say that it is {\it of type} $\frac{1}{r} (a_1,a_2,a_3)$ if $(X,x)$ is analytically isomorphic to the germ $(\mbA^3/\bmu_r, o)$, where $\bmu_r$ is the cyclic group of $r$-th roots of unity with a primitive $r$-th root of unity $\varepsilon$, $\mbA^3/\bmu_r$ is the geometric quotient of $\mbA^3$ by the $\bmu_r$-action given by $(x_1, x_2, x_3) \mapsto (\varepsilon^{a_1} x_1, \varepsilon^{a_2} x_2, \varepsilon^{a_3} x_3)$ and $o$ is the image of the origin of $\mbA^3$.
\end{Def}

The following is the characterization of $3$-dimensional terminal cyclic quotient singularities.
For a proof, we refer the readers to \cite[Section 5]{Reid}.

\begin{Thm}
Let $(X,x)$ be a germ of a $3$-dimensional cyclic quotient singularity.
Then $(X,x)$ is terminal if and only if it is of type $\frac{1}{r} (1, a, r-a)$ with $a$ coprime to $r$.
\end{Thm}

\begin{Thm}[\cite{Kawamata}] \label{divcontquot}
Let $(X,x)$ be the germ of terminal cyclic quotient singularity of type $\frac{1}{r} (1,a,r-a)$ with $r \ge 2$, $0 < a < r$ and $a$ coprime to $r$, and $\varphi \colon (E \subset Y) \to (\Gamma \subset X)$ a divisorial contraction with $x \in \Gamma$.
Then $\Gamma = x$ and $\varphi$ is the weighted blowup with weight $\frac{1}{r} (1,a,r-a)$.
\end{Thm}

The unique divisorial contraction in Theorem \ref{divcontquot} is called the {\it Kawamata blowup} of $X$ at $x$.

\subsection{Inequalities on multiplicities}

Let $\msp \in X$ be a point of a $\mbQ$-Fano variety with $\Div (X) \cong \mbZ$ and assume that $A = -K_X$ is the positive generator of $\Div (X)$.
If $\msp$ is a maximal center for some movable linear system $\mcH \subset |n A|$, then $K_X + \frac{1}{n} \mcH$ is not canonical at $\msp$.
We have the following inequalities on multiplicities.

\begin{Thm}[{\cite[Theorem 5.3.2]{CPR}, \cite[Corollary 3.4]{Corti2}}] \label{multineq1}
Let $\msp \in X$ be a germ of a nonsingular threefold and $\mcH$ a movable linear system on $X$.
Assume that $K_X + \frac{1}{n} \mcH$ is not canonical at $\msp$.
\begin{enumerate}
\item If $\msp \in S \subset X$ is a surface and $\mcL = \mcH |_S$, then $K_S + \frac{1}{n} \mcL$ is not log canonical.
\item If $Z = H_1 \cap H_2$ is the intersection of two general members $H_1, H_2$ of $\mcH$, then $\mult_{\msp} Z > 4 n^2$.
\end{enumerate}
\end{Thm}

\begin{Thm}[{\cite[Theorem 3.1]{Corti2}}] \label{multineq2}
Let $\msp \in \Delta_1 + \Delta_2 \subset S$ be an analytic germ of a normal crossing curve on a nonsingular surface $($so that $\Delta_1 + \Delta_2 \subset S$ is isomorphic to the coordinate axes $(x y = 0) \subset \mbC^2$$)$.
Let $\mcL$ be a movable linear system on $S$, and write $(\mcL^2)_{\msp}$ for the local intersection multiplicity $(L_1 \cdot L_2)_{\msp}$ of two general members $L_1, L_2 \in \mcL$ at $\msp$.
Fix rational numbers $a_1, a_2 \ge 0$, and assume that
\[
K_S + (1-a_1) \Delta_1 + (1-a_2) \Delta_2 + \frac{1}{n} \mcL
\]
is not log canonical at $\msp$.
Then the following assertions hold.
\begin{enumerate}
\item If either $a_1 \le 1$ or $a_2 \le 1$, then $(\mcL^2)_{\msp} > 4 a_1 a_2 n^2$.
\item If both $a_i > 1$, then $(\mcL^2)_{\msp} > 4 (a_1 + a_2 -1) n^2$.
\end{enumerate}
\end{Thm}

\subsection{Sarkisov links}

\begin{Def}
A {\it Sarkisov link} between two $\mbQ$-Fano $3$-folds $X$ and $X'$ is a birational map $\sigma \colon X \ratmap X'$ which factorizes as
\[
\xymatrix{
Y \ar@{-->}[r] \ar[d] & Y' \ar[d] \\
X \ar@{-->}[r]^{\sigma \hspace{1mm}} & X',}
\]
where $Y \to X$ and $Y' \to X'$ are extremal divisorial contractions, and $Y \ratmap Y'$ is a composite of inverse flips, flops and flips (in that order).
This is called a {\it Sarkisov link of type II}.
A {\it Sarkisov link} from a $\mbQ$-Fano $3$-fold $X$ to a strict Mori fiber space $V \to T$ (i.e.\ either a del Pezzo fiber space or a conic bundle) is a birational map $\sigma \colon X \ratmap V$ which factorizes as
\[
\xymatrix{
Y \ar@{-->}[rd] \ar[d] & \\
X \ar@{-->}[r]_{\sigma \hspace{3mm}} & V/T,}
\]
where $Y \to X$ is an extremal divisorial contraction and $Y \ratmap V$ is a composite of inverse flips, flops and flips (in that order).
This is called a {\it Sarkisov link of type I}.

The center of the contraction $Y \to X$ is called the {\it center of a Sarkisov link}.
\end{Def}

Let $X$ be a $\mbQ$-Fano threefold and assume that there is a birational map $f \colon X \ratmap V$ to a Mori fiber space.
We fix a sufficiently ample complete linear system $\mcH_V$ on $V$ and let $\mcH$ be its birational transform on $X$ via $f$.

\begin{Def}
Notation as above.
The {\it degree} of $f$ (relative to $\mcH_V$) is the quasieffective threshold $c (X,\mcH)$.
We say that a Sarkisov link $\sigma \colon X \ratmap X'$ of type II {\it untwists} $f$ if $f' = f \circ \sigma^{-1} \colon X' \ratmap V$ has degree smaller than $f$.
\end{Def}

\begin{Lem} \label{untwisting}
Notation as above.
If $\varphi \colon (E \subset Y) \to X$ is a maximal singularity of $\mcH$, then any Sarkisov link $\sigma \colon X \ratmap X'$ of type II starting with the contraction $\varphi$ untwists $f \colon X \ratmap V$.
\end{Lem}

\begin{proof}
See \cite[Lemma 4.2]{CPR} or the proof of \cite[Theorem 5.4]{Corti}.
\end{proof}

\subsection{Structure of the proof}

In the remainder of the paper, we prove the following theorem which implies Theorems \ref{mainthmbr} and \ref{mainthmF}.

\begin{Thm} \label{thm}
Let $X$ be a member of the family $\mcG_i$ with $i \in I^*$.
Then the following assertions hold.
\begin{enumerate}
\item[(a)] No curve on $X$ is a maximal center.
\item[(b)] No nonsingular point of $X$ is a maximal center.
\item[(c)] For each singular point $\msp$ of $X$ which is not a distinguished singular point, one of the following holds.
\begin{enumerate}
\item[(c-1)] $\msp$ is not a maximal center.
\item[(c-2)] There is a birational involution $\iota_{\msp} \colon X \ratmap X$ which is a Sarkisov link of type II centered at $\msp$.
\end{enumerate}
\item[(d)] For each distinguished singular point $\msp$ of $X$, there is a Sarkisov link $\sigma_{\msp} \colon X \ratmap X'$ to a $\mbQ$-Fano variety $X'$ or to a del Pezzo fiber space $X'/\mbP^1$ over $\mbP^1$.
\end{enumerate}
The Mori fiber space $X'$ in \emph{(d)} does not depend on the choice of distinguished points of $X$ and is uniquely determined by $X$.
\end{Thm}

If $i \in I^*_F \cup I^*_{dP}$, then Theorem \ref{thm} coincides with Theorem \ref{mainthmF}.
Assume that $i \in I^*_{br}$.
We explain that Theorem \ref{thm} implies Theorem \ref{mainthmbr}.
Note that there is no distinguished singular point on $X$ and thus (d) does not occur.
Suppose that we are given a birational map $f \colon X \ratmap V$ to a Mori fiber space.
Pick a sufficiently ample complete linear system $\mcH_V$ on $V$ and let  $\mcH \subset |n A|$ be the birational transform of $\mcH_V$ on $X$ via $f$.
If $f$ is not an isomorphism, then Noether-Fano-Iskovskikh inequality (see \cite[Theorem 2.4]{Corti2}) implies that $K_X+\frac{1}{n} \mcH$ is not canonical.
Then the argument of Corti \cite{Corti} implies that there is a maximal singularity of $\mcH$.
By (a) and (b), the center of the maximal singularity is a singular point $\msp$ of $X$.
By (c), there is a birational involution $\iota_{\msp} \colon X \ratmap X$ which is a Sarkisov link of type II centered at $\msp$.
Lemma \ref{untwisting} shows that $\iota_{\msp}$ untwists $f$.
If $f \circ \iota_{\msp}$ is an isomorphism then $V \cong X$, otherwise there is a maximal center and we repeat the untwisting process.
This process stops in finitely many steps.
It follows that $V \cong X$ and the given birational map $f$ is the composite of finitely many birational involutions constructed in (c--2).
This shows that $X$ is birationally rigid. 

The proofs of (a), (b), (c--1), (c--2) and (d) will be done in Sections \ref{sec:curve}, \ref{sec:nspt}, \ref{sec:singpt}, \ref{sec:birinv} and \ref{sec:Slink}, respectively, and Theorem \ref{thm} follows from Theorems \ref{exclcurve}, \ref{exclnspt}, \ref{exclsingpts}, \ref{birinvQI}, \ref{birinvEI}, and \ref{linkFano}.
This paper is based on the paper \cite{CPR} for the idea and technique necessary to prove the main theorem. 
Constructions of birational involutions and Sarkisov links are essentially the same as those given in \cite{CPR} and \cite{CM}.
Moreover there is a general framework of constructions of such birational maps (see \cite{Ahmadinezhad2}, \cite{ABR}, \cite{BZ}, \cite{Pap}, \cite{PR}, \cite{Reid2}) and there is a database of $\mbQ$-Fano $3$-folds in \cite{Brown} where one can in particular find centers of those birational maps.
The main part of this paper is to exclude the remaining centers as maximal center.

\section{Sarkisov link to another Mori fiber space} \label{sec:Slink}

\subsection{Weighted blowup}

In this subsection, we fix notation on weighted blowups of cyclic quotients of affine spaces which will be used in what follows.

Let $\mbA := \mbA^n$ be an affine $n$-space with affine coordinates $x_1,\dots,x_n$ and let $b_1,\dots,b_n$ be positive integers.
Let $\varphi \colon \mbA \ratmap \mbP := \mbP (b_1,\dots,b_n)$ be the rational map defined by sending $(\alpha_1,\dots,\alpha_n) \in \mbA$ to $(\alpha_1:\cdots:\alpha_n) \in \mbP$ and $\Wbl (\mbA) \subset \mbA \times \mbP$ the graph of $\varphi$.
We call the projection $\Wbl (\mbA) \to \mbA$ the {\it weighted blowup} of $\mbA$ at the origin with $\wt (x_1,\dots,x_n) = (b_1,\dots,b_n)$.
Suppose that we are given a $\bmu_r$-action on $\mbA$ of type $\frac{1}{r} (a_1,\dots,a_n)$ for some positive integer $r$.
Let $V$ be the quotient of $\mbA$ by the $\bmu_r$-action and $X$ a subvariety of $V$ through the origin.
We assume that $b_i \equiv a_i \pmod{r}$ for every $i$.
Then the rational map $\varphi$ descends to a rational map $\varphi_V \colon V \ratmap \mbP$. 
We define $\Wbl (V)$ (resp.\ $\Wbl (X)$) to be the graph of $\varphi_V$ (resp.\ $\varphi_V|_X$) and call the natural projection $\Wbl (V) \to V$ (resp.\ $\Wbl (X) \to X$) the {\it weighted blowup} of $V$ (resp.\ $X$) at the origin with $\wt (x_1,\dots,x_n) = \frac{1}{r} (b_1,\dots,b_n)$.
We consider the $\bmu_r$-action on $\mbA \times \mbP$ which acts on $\mbA$ as above and on $\mbP$ trivially.
This induces the $\bmu_r$-action on $\Wbl (\mbA)$.
We have a natural morphism $\Wbl (\mbA) \to \Wbl (V)$ and from this we can see $\Wbl (V)$ as the quotient of $\Wbl (\mbA)$ by the $\bmu_r$-action. 

We describe $\Wbl (\mbA)$ and $\Wbl (V)$ locally.
Let $X_1,\dots,X_n$ be the homogeneous coordinates of $\mbP$.
For $i = 1,\dots,n$, we define $\Wbl_i (\mbA)$ (resp.\ $\Wbl_i (V)$) to be the open subset of $\Wbl (\mbA)$ (resp.\ $\Wbl_i (V)$) which is the intersection of $\Wbl (\mbA)$ (resp.\ $\Wbl (V)$) and the open subset $\mbA \times (X_i \ne 0)$ (resp.\ $V \times (X_i \ne 0)$), where $(X_i \ne 0)$ is the open subset of $\mbP$.
Let $U_i (\mbA)$ be an affine $n$-space with affine coordinates $\tilde{x}_1,\dots,\tilde{x}_{i-1}, x'_i, \tilde{x}_{i+1},\dots,\tilde{x}_n$ and define a morphism $U_i (\mbA) \to \mbA$ by the identification $x_i = {x'}_i^{b_i}$ and $x_j = \tilde{x}_j {x'}_i^{b_j}$ for $j \ne i$. 
We consider the action of $\bmu_{b_i}$ on $U_i (\mbA)$ of type $\frac{1}{b_i} (b_1,\dots,b_{i-1},-1,b_{i+1},\dots,b_n)$.
We see that the quotient of $U_i (\mbA)$ by the above action is naturally isomorphic to $\Wbl_i (\mbA)$ and the section $x'_i$ cuts out the open subset  $(X_i \ne 0)$ of the exceptional divisor $\mbP$.
Note that $\Wbl_i (V)$ is the quotient of $\Wbl_i (\mbA)$ by the action of $\bmu_r$.
Using orbifold coordinates, the $\bmu_r$-action on $\Wbl_i (\mbA)$ is the one given by $x'_i \mapsto \zeta_r x'_i$ and $\tilde{x}_j \mapsto \tilde{x}_j$ for $j \ne i$, where $\zeta_r$ is a primitive $r$-th root of unity.
Let $U_i (V)$ be the affine $n$-space with affine coordinates $\tilde{x}_1,\dots,\tilde{x}_n$.
The identification $\tilde{x}_i = {x'}^r_i$ defines a morphism $U_i (\mbA) \to U_i (V)$ and we see that $\Wbl_i (V)$ is the quotient of $U_i (V)$ by the action of $\bmu_{b_i}$ on $U_i (V)$ of type $\frac{1}{b_i} (b_1,\dots,b_{i-1}, -r,b_{i+1},\dots,b_n)$.
We call $U_i (V)$ the {\it orbifold chart} of $\Wbl_i (V)$ and call $\tilde{x}_1, \dots, \tilde{x}_n$ {\it orbifold coordinates} of $\Wbl_i (V)$.

\subsection{Sarkisov link to a $\mbQ$-Fano variety}
Throughout this subsection, let $X = X_{d_1,d_2} \subset \mbP (a_0,\dots,a_5)$ be a member of $\mcG_i$ with $i \in I_F$ and $\msp$ a distinguished singular point of $X$, that is, it is a singular point marked $X'_d \subset \mbP (1,b_1,b_2,b_3,b_4)$ in the third column of the big table.
We observe that $d_1 \ne d_2$ and we assume that $d_1 < d_2$.

The construction of Sarkisov links in this and the next subsection can be done by the technique of \cite{BZ} (see also \cite{Ahmadinezhad2}) and we can moreover say that the existence of such links is an implicit result of the above mentioned paper and \cite{CM}.
We give a completely explicit treatment of these constructions to make the target Mori fiber spaces as exlicit as possible.  

\begin{Lem} \label{Slinkcoord}
After re-ordering weights $a_0,\dots,a_5$ and re-choosing homogeneous coordinates, the following assertions hold.
\begin{enumerate}
\item $\msp = \msp_5$. 
\item $a_3 < a_4$.
\item Defining polynomials $F_1$ and $F_2$ of $X$ can be written as
\[
F_1 = x_5 x_3 + G_1 \text{ and } F_2 = x_5 x_4 + G_2,
\]
where $G_i$ does not involve the variable $x_5$.
\end{enumerate}
\end{Lem}

\begin{proof}
We can assume that $\msp$ is a vertex $\msp = \msp_k$ for some $k$ after re-choosing homogeneous coordinates.
It is enough to show that $x_k x_i \in F_1$ and $x_k x_j \in F_2$ for some $i \ne k$ and $j \ne k$.
Indeed if this is the case, then after re-choosing homogeneous coordinates we can assume $F_1 = x_k x_i + G_1$ and $F_2 = x_k x_j + G_2$, where $G_1$ and $G_2$ do not involve $x_k$.
Here the change of coordinates is immediate if $2 a_k > d_2$ but less clear if $2 a_k \le d_2$.
In the latter case, we observe that $2 a_j > d_2$, so that both $F_1$ and $F_2$ are linear with respect to $x_j$, which enables us to get the desired defining polynomials (see Example \ref{exslNo57} below).
Finally the assertions follows after re-ordering coordinates.

In the remainder of the proof, we assume that $a_0 \le a_1 \le \cdots \le a_5$.
Assume first that $\msp$ is a point with empty forth column in the big table.
In this case we have $2 a_k \ge d_2 > d_1$ except for the case where $\msp = \msp_4$ and $X$ is a member of the family $\mcG_i$ for $i \in \{21,36,38,52,63\}$.
If $2 a_k \ge d_2 > d_1$, then the assertion follows immediately from quasismoothness of $X$.
In the above exceptional case, we have $2 a_4 > d_1$ and $a_4 + a_5 = d_2$.
It follows that $x_4 x_i \in F_1$ for some $i \le 3$ and $x_4 x_5 \in F_2$, where the latter follows from quasismoothness of $X$ at $\msp_5$.

Assume that $\msp$ is a point marked $(\mathrm{C}_1)$ in the forth column of the big table.
In this case we have $a_3 + a_4 = d_1$, $a_3 + a_5 = d_2$ and we may assume that $\msp = \msp_3$.
By quasismoothness of $X$ at $\msp_5 \in X$, we have $x_3 x_5 \in F_2$, and by $(\mathrm{C}_1)$, we have $x_3 x_4 \in F_1$.

Assume that $\msp$ is a point marked $(\mathrm{C}_2)$ in the forth column of the big table.
In this case we have $a_3 = a_4$, $ 2 a_4 = d_1$, $a_3 + a_5 = d_2$, and after re-choosing homogeneous coordinates, we may assume that $\msp = \msp_4$.
By quasismoothness of $X$ we have $x_4 x_3 \in F_1$.
If $x_4 x_5 \notin F_2$, then $X$ contains the WCI curve $(x_0 = x_1 = x_2 = x_3 = 0)$ of type $(a_0,a_1,a_2,a_3)$.
By $(\mathrm{C}_2)$ this cannot happen and we have $x_4 x_5 \in F_2$.
This completes the proof.
\end{proof}

In the following, we keep weights and homogeneous coordinates as in Lemma \ref{Slinkcoord} unless otherwise stated.

\begin{Rem} \label{rem:SLCox}
We explain briefly the construction of a Sarkisov link in terms of toric modifications of ambient spaces following \cite{BZ}.
Let $\varphi \colon Y \to X$ be the Kawamata blowup of $X$ at $\msp$ with exceptional divisor $E$.
The sections $x_0,x_1,x_2$ can be chosen as local coordinates of $X$ at $\msp$ and they vanish along $E$ to order $a_0/a_5$, $a_1/a_5$, $a_2/a_5$.
By the equations $F_1 = 0$ and $F_2 = 0$, the sections $x_3$ and $x_4$ vanish along $E$ to order $(a_3 + a_5)/a_5 = d_1/a_5$ and $(a_4 + a_5)/a_5 = d_2/a_5$.
Let $Q \to \mbP := \mbP (a_0,\dots,a_5)$ be the weighted blowup at $\msp$ with weight $\wt (x_0,x_1,x_2,x_3,x_4) = \frac{1}{a_5} (a_0,a_1,a_2,d_1,d_2)$.
Then the restriction of $Q \to \mbP$ to the strict transform of $X$ in $Q$ is the Kawamata blowup of $X$ at $\msp$, so we can identify $Y$ with a divisor in $Q$.
By the argument in \cite{BZ}, the $2$-ray game for toric varieties starting with $Q \to \mbP$ restricts to the $2$-ray game starting with $Y \to X$ and it ends with a $\mbQ$-Fano threefold.
\end{Rem}

\begin{Def}
We define weighted hypersurfaces
\[
Z := (x_4 G_1 - x_3 G_2 = 0) \subset \mbP (a_0,a_1,a_2,a_3,a_4),
\]
and
\[
X' := (w G'_1 - G'_2 = 0) \subset \mbP (a_0,a_1,a_2,a_3,b),
\]
where $w$ is the homogeneous coordinate with $\deg w = b := a_4 -a_3$ and $G'_i = G'_i (x_0,x_1,x_2,x_3,w) := G_i (x_0,x_1,x_2,x_3,x_3 w)$.
Let $\pi \colon X \ratmap Z$ be the projection from $\msp$ and $\pi' \colon X' \ratmap Z$ the rational map determined by the correspondence $x_4 = x_3 w$.
Let $\varphi \colon Y \to X$ be the Kawamata blowup of $X$ at $\msp$ and $\varphi' \colon Y' \to X'$ the weighted blowup of $X'$ at the point $\msp' := (0:0:0:0:1) \in X'$ with $\wt (x_0,x_1,x_2,x_3) = \frac{1}{b} (a_0,a_1,a_2,a_4)$.
\end{Def}

We see that both $\pi$ and $\pi'$ are birational maps.
We shall show that $X'$ is a $\mbQ$-Fano threefold and the birational map ${\pi'}^{-1} \circ \pi \colon X \ratmap X'$ is a Sarkisov link.

\begin{Lem} \label{SLQFanoflopping1}
The Kawamata blowup $\varphi \colon Y \to X$ resolves the indeterminacy of $\pi$ and the induced birational morphism $\psi \colon Y \to Z$ is a flopping contraction.
\end{Lem}

\begin{proof}
It is easy to see that sections $x_0$, $x_1$, $x_2$, $x_3$ and $x_4$ lift to plurianticanonical sections of $Y$ (see Example \ref{exslNo57} below) and define the morphism $\psi \colon Y \to Z$, which contracts finitely many curves $(x_3 = x_4 = G_1 = G_2 = 0)$.
It follows that $Z$ is the anticanonical model of $Y$ and $\psi \colon Y \to Z$ is a flopping contraction.
\end{proof}

\begin{Ex} \label{exslNo57}
Let $X = X_{12,14} \subset \mbP (1,2,3,5,7,9)$ be a member of $\mcG_{57}$ with defining polynomials $F_1, F_2$ of degree respectively $12, 14$ and $\msp = \msp_3$ the point of type $\frac{1}{5} (1,2,3)$.
Since $s t \in F_1$ by $(\mathrm{C}_1)$ and $s u \in F_2$ by quasismoothness, we can write $F_1 = s^2 f_2 + s t + u f_3 + t f_5 + f_{12}$ and $F_2 = s^2 g_2 + s u + u g_5 + t^2 + t g_7 + g_{14}$ for some $f_i, g_i \in \mbC [x,y,z]$.
We shall eliminate the terms $s^2 f_2$ and $s^2 g_2$ by replacing $t$ and $u$.
Replacing $t$ with $t-s f_2$, we can eliminate $s^2 f_2$, that is, we may assume that $f_2 = 0$.
Then, replacing $u$ with $u - s g_2 + g_2 g_5$ and $t$ with $t + f_3 g_2$, we can eliminate the term $s^2 g_2$ and thus defining polynomials can be written as $F_1 = s t + u f_3 + t f_5 + f_{12}$ and $F_2 = s u + u g_5 + t^2 + t g_7 + g_{14}$.
This example explains all the necessary technique to obtain the defining equations in Lemma \ref{Slinkcoord}.
We can choose $x,y,z$ as local coordinates of $X$ at $\msp$ and the Kawamata blowup $\varphi \colon Y \to X$ is the weighted blowup with $\wt (x,y,z) = \frac{1}{5}(1,2,3)$.
It is obvious that the section $u$ vanishes along the exceptional divisor $E$ of $\varphi$ to order at least $4/5$.
By the equation $F_1 = 0$, we have $t (s+f_5) = - u f_3 - f_{12}$.
The section $s + f_5$ does not vanish at $\msp$ and $- u f_3 - f_{12}$ vanishes along $E$ to order at least $7/5$, which shows that $t$ vanishes along $E$ to order at least $7/5$.
By the equation $F_2 = 0$, we have $u (s+g_5) = - t^2 - t g_7 - g_{14}$.
A similar argument shows that $u$ vanishes along $E$ to order at least $9/7$.
This shows that sections $x,y,z,t$ and $u$ lift to plurianticanonical sections of $Y$. 
\end{Ex}

\begin{Lem} \label{qsmFano}
$X'$ has only terminal cyclic quotient singular points except at $\msp'$.
\end{Lem}

\begin{proof}
It is enough to show that $X'$ is quasismooth outside $\msp'$, that is, the affine cone $C_{X'}$ is smooth outside the line over $\msp'$.
Quasismoothness of $X$ implies that of $Z$ outside  the set $(x_3 = x_4 = 0)$, which in turn implies quasismoothness of $X'$ outside the set $(x_3 = 0)$.
This can be verified by comparing Jacobian matrices of the affine cones of $X$, $Z$ and $X'$ keeping in mind the relation $x_4 = x_3 w$.

Assume that $X'$ is not quasismooth at $\msq' = (\xi_0:\cdots:\xi_3:\omega) \ne \msp'$.
For $i = 1,2$, we write $G_i = H_i + x_3 H_{i 3} + x_4 H_{i 4} +\widetilde{G}_i$, where $\widetilde{G}_i \in (x_3,x_4)^2$ and $H_i,H_{i 3}, H_{i 4} \in \mbC [x_0,x_1,x_2]$.
For a polynomial $\Phi = \Phi (x_0,\dots,x_4)$, we define $\Phi' := \Phi (x_0,x_1,x_2,x_3,x_3 w)$.
From the above argument, the point $\msq'$ is contained in $(x_3 = 0)_{X'} = (x_3 = H_1 = H_2 = 0)$.
It follows that $H_1$, $H_2$ vanish at $\msq'$ and $\xi_3 = 0$.
Let $\msq = (\xi_0:\xi_1:\xi_2:0:0:\xi_5)$ be the point of $X$, where $\xi_5 := \omega H_{14} (\msq') - H_{24} (\msq')$.
Note that $\msq$ is indeed contained in $X$ and $\msq \ne \msp$ since $(x_3 = x_4 = H_1 = H_2 = 0) \subset X$ and $(\xi_0, \xi_1, \xi_2) \ne (0,0,0)$.
We have
\[
\begin{split}
J (\msq) &:= \left( \frac{\prt F_i}{\prt x_j} (\msq) \right)_{1 \le i \le 2, 0 \le j \le 5} \\ 
&= 
\begin{pmatrix} 
\displaystyle \frac{\prt H_1}{\prt x_0} (\msq) & \displaystyle \frac{\prt H_1}{\prt x_1} (\msq) & \displaystyle \frac{\prt H_1}{\prt x_2} (\msq) & \xi_5 + H_{13} (\msq) & H_{14} (\msq) & 0 \\
&&&&& \\
\displaystyle \frac{\prt H_2}{\prt x_0} (\msq) & \displaystyle \frac{\prt H_2}{\prt x_1} (\msq) & \displaystyle \frac{\prt H_2}{\prt x_2} (\msq) & H_{23} (\msq) & \xi_5 + H_{24} (\msq) & 0
\end{pmatrix}.
\end{split}
\]
Since $H_i$, $H_{i3}$ and $H_{i4}$ are polynomials in $x_0,x_1$ and $x_2$, we have $(\prt H_i/\prt x_j) (\msq) = (\prt H_i/\prt x_j) (\msq')$ and $(\prt H_{ik} / \prt x_j) (\msq) = (\prt H_{ik} / \prt x_j) (\msq')$ for $i = 1,2$, $j = 0,1,2$ and $k = 3,4$.
Now the equations
\[
\frac{\prt (w G'_1 - G'_2)}{\prt x_j} (\msq') = \omega \frac{\prt H_1}{\prt x_j} (\msq') - \frac{\prt H_2}{\prt x_j} (\msq') = 0
\]
for $j = 0,1,2$, and
\[
\frac{\prt (w G'_1 - G'_2)}{\prt x_3} (\msq') = \omega (H_{1 3} (\msq') + \omega H_{1 4} (\msq')) - (H_{2 3} (\msq') + \omega H_{24} (\msq')) = 0,
\]
which follow from the assumption that $X'$ is not quasismooth at $\msq'$, imply that $\rank J (\msq) \le 1$.
Therefore, $X$ is not quasismooth at $\msq$ and this is a contradiction.
\end{proof}

An explicit local analysis shows that $X'$ has a $cDV$ or $cDV/n$ singular point at $\msp'$, from which we see that $X'$ has only terminal singularities.
But note that we cannot deduce from this (global) $\mbQ$-factoriality of $X'$.

\begin{Lem} \label{qsmwFano}
$Y'$ has only cyclic quotient singular points.
\end{Lem}

\begin{proof}
We shall show that $Y'$ has only cyclic quotient singularities.
By Lemma \ref{qsmFano}, it is enough to consider singular points of $Y'$ which lie on the exceptional divisor $E'$ of $\varphi'$.
We re-choose homogeneous coordinates so that $G_1 = G_1 (x_0,x_1,x_2,x_3,x_4)$ does not involve the variable $x_3$, that is, $G_1 = G_1 (x_0,x_1,x_2,x_4)$.
This can be done by filtering off terms divisible by $x_3$ in $F_1$ and then replacing $x_5$.
 
We restrict $X'$ to the open set $V := (w \ne 0)$ of $\mbP (a_0,a_1,a_2,a_3,b)$.
Note that $V$ is the quotient of $\mbA^4$ by the $(\mbZ / b \mbZ)$-action of type $\frac{1}{b} (a_0,a_1,a_2,a_3)$ and $X' \cap V$ is defined by the equation $F' := G_1 (x_0,x_1,x_2,x_3) - G_2 (x_0,x_1,x_2,x_3,x_3) = 0$ which is obtained by setting $w = 1$ in the defining equation $w G'_1 - G'_2 = 0$ of $X'$.  
Thus $Y'$ is embedded in $\Wbl (V)$, where $\Wbl (V) \to V$ is the weighted blowup at the origin with $\wt (x_0,x_1,x_2,x_3) = \frac{1}{b} (a_0,a_1,a_2,a_4)$. 
We call this weight the blowup weight.
Let $\mbP := \mbP (a_0,a_1,a_2,a_4)$ be the weighted projective space with homogeneous coordinates $X_0$, $X_1$, $X_2$ and $X_3$ for which we have $\Wbl (V) \subset V \times \mbP$.
The lowest weight part of $F'$ with respect to the blowup weight is $G_1 (x_0,x_1,x_2,x_3)$.
It follows that $E'$ is the weighted hypersurface
\[
E' = (H = 0) \subset \mbP,
\]
where $H := G_1 (X_0,X_1,X_2,X_3)$.
It is easy to see that $\Sing (Y') \cap E' \subset \Sing (E')$.
Let $\msq' \in \mbP$ be a point of $E'$.
The second lowest part of $F'$ with respect to the blowup weight is $G_2 (x_0,x_1,x_2,0,x_3)$.
It follows that $Y'$ has at most cyclic quotient singularity at $\msq'$ if either $E' \subset \mbP$ is quasismooth at $\msq'$ or $G_2 (X_0,X_1,X_2,0,X_3)$ does not vanish at $\msq'$.
Now assume that $Y'$ has a non-cyclic quotient singular point $\msq' = (\xi_0:\xi_1:\xi_2:\xi_4) \in E'$.
Then, we have $(\prt H/\prt X_i) (\msq') = 0$ for $i = 0,1,2,3$ and $G_2 (X_0,X_1,X_2,0,X_3)$ vanishes at $\msq'$.
We define $\msq := (\xi_0:\xi_1:\xi_2:0:\xi_4:0) \in \mbP (a_0,\dots,a_5)$.
We see that $\msq$ is a point of $X$. 
Moreover, we have $(\prt G_1/\prt x_i) (\msq) =  (\prt H/\prt X_i)(\msq') =0$ for $i = 0,1,2$, $(\prt G_1/\prt x_3)(\msq) = 0$, $(\prt G_1/\prt x_4)(\msq) = (\prt H/\prt X_3)(\msq') = 0$ and $G_2 (\msq) = 0$.
Thus, we have $(\prt F_1/\prt x_i) (\msq) = (\prt (x_5 x_3 + G_1) /\prt x_i) (\msq) = 0$ for $i = 0,1,\dots,5$.
This shows that $X$ is not quasismooth at $\msq$ and this derives a contradiction.
Thus, $Y'$ has only cyclic quotient singular points.
\end{proof}

\begin{Lem} \label{lem:indetY'}
The weighted blowup $\varphi' \colon Y' \to X'$ resolves the indeterminacy of $\pi'$ and the induced birational morphism $\psi' \colon Y' \to Z$ is a flopping contraction.
\end{Lem}

\begin{proof}
We see that sections $x_0$, $x_1$, $x_2$ and $x_3 w$ lift to plurianticanonical sections of $Y'$ and they define the morphism $\psi' \colon Y' \to Z$, which contracts finitely many curves $(x_3 = G'_1 = G'_2 = 0)$.
It follows that $Z$ is the anticanonical model of $Y'$ from which the assertions follow.
\end{proof}

\begin{Lem} \label{SLQFanoQfact}
$X'$ is a $\mbQ$-Fano threefold.
\end{Lem}

\begin{proof}
We see that $Y'$ is $\mbQ$-factorial by Lemma \ref{qsmwFano}, and the Picard number of $Y'$ is equal to $2$ by Lemma \ref{lem:indetY'}.
It follows that $\varphi'$ is a $K_{Y'}$-negative extremal divisorial contraction from a $\mbQ$-factorial threefold and $X'$ is a $\mbQ$-Fano threefold.
\end{proof}

Combining the above results, we obtain the following.

\begin{Thm} \label{linkFano}
Let $X$ be a member of the family $\mcG_i$ with $i \in I^*_F$ and $\msp$ a distinguished singular point of $X$ marked $X'_d \subset \mbP (1, b_1,b_2,b_3, b_4)$ in the third column of the big table.
Then there is a Sarkisov link
\[
\xymatrix{
Y \ar@{-->}^{\tau}[r] \ar[d]_{\varphi} & Y' \ar[d]^{\varphi'} \\
X \ar@{-->}_{\sigma_{\msp}}[r] & X'}
\]
to a $\mbQ$-Fano weighted hypersurface $X' = X'_d$ of degree $d$ in $\mbP (1,b_1,\dots,b_4)$, where $\varphi$ is the Kawamata blowup at $\msp$, $\varphi'$ is an extremal divisorial contraction to a unique non-quotient terminal singular point of $X'$ and $\tau$ is a flop.
\end{Thm}

\begin{proof}
This follows from Lemmas \ref{SLQFanoflopping1}, \ref{lem:indetY'} and \ref{SLQFanoQfact}.
We explain that $\tau \colon Y \ratmap Y$ is indeed a flop.
One way to see this is by the Cox ring approach:
The diagram in the statement is obtained as the restriction of the $2$-ray game of the ambient toric variety $Q \supset Y$ (see Remark \ref{Slinkcoord}).
From this we see that $\tau$ is a log flip of a pair $(X,\Delta)$ for a doundary $\Delta$ which is the restriction of a suitable divisor on $Q$.
Another way to see this is the following:
We see that $\tau$ is not an isomorphism since $X$ is clearly not isomorphic to $X'$.
Pick a sufficiently small positive multiple $\Delta'$ of a $\psi'$-ample divisor on $Y'$ and set $\Delta := \tau_*^{-1} \Delta'$.
We see that $\Delta'$ is not $\psi$-ample since $\tau$ is not an isomorphism.
It follows that $-\Delta$ is $\psi$-ample since $\rho (Y/Z) = 1$.
This shows that $\tau$ is the log flip of $(Y,\Delta)$ and thus $\tau$ is the flop of $\psi$. 
\end{proof}

In many cases $X$ admits several distinguished singular points and there are several Sarkisov links.

\begin{Prop}
Let $X$ be a member of $\mcG_i$ with $i \in I_F^*$ and assume that there are several distinguished singular points on $X$.
Then the target $\mathbb{Q}$-Fano threefold $X'$ of the Sarkisov link centered at each distinguished singular point is uniquely determined by $X$.
\end{Prop}

\begin{proof}
Let $X = X_{d_1,d_2} \subset \mbP (a_0,\dots,a_5)$ be a member of $\mcG_i$ admitting several distinguished singular points.
Assume that all the distinguished singular points on $X$ are of the same type, that is, $X$ is a member of $\mcG_i$ for $i \in \{7,10,19,23,30,35,42,50,58\}$.
In this case the weights $a_0 \le \cdots \le a_5$ satisfies $a_4 = a_5$ and the distinguished singular points are the two points in $(x_0 = \cdots = x_3 = 0)_X$.
By a suitable choice of homogeneous coordinates, defining polynomials of $X$ can be written as
\[
F_1 = x_4 a + x_5 b + c \quad \text{and} \quad 
F_2 = x_4 x_5 - d,
\]
where $a,b,c,d \in \mbC [x_0,x_1,x_2,x_3]$.
The distinguished singular points of $X$ are $\msp_4$ and $\msp_5$ and the target $\mbQ$-Fano threefolds of the Sarkisov links centered at those points are the same and it is the weighted hypersurface defined by $w^2 a b + w c + d = 0$.

Assume that there are distinguished singular points $\msq_1$ and $\msq_2$ of distinct type.
We may assume that $\msq_1 = \msp_4$ and $\msq_2 = \msq_5$ after replacing homogeneous coordinates.
Then, by a suitable choice of coordinates, the defining polynomials of $X$ can be written as
\[
F_1 = x_4 x_j + x_5 x_k + f \quad \text{and} \quad
F_2 = x_4 x_5 - g,
\]
where $0 \le j, k \le 3$, $j \ne k$, and $f,g \in \mbC [x_0,x_1,x_2,x_3]$.
The target $\mathbb{Q}$-Fano threefolds of the Sarkisov links centered at $\msq_1$ and $\msq_2$ are the same and it is the weighted hypersurface defined by $w^2 x_j x_k + w f + g = 0$.
This completes the proof.
\end{proof}

\subsection{Sarkisov link to a del Pezzo fibration}

In this subsection, we treat a member $X$ of $\mcG_i$ with $i \in I^*_{dP}$ and construct a Sarkisov link from $X$ to a del Pezzo fiber space $X' \to \mbP^1$ centered at each distinguished singular point.
The total space $X'$ is embedded as a hypersurface in a weighted projective space bundle over $\mbP^1$.
The construction is given in an explicit and elementary way, which can  be completely understood in terms of Cox rings.
We refer the readers to \cite{BZ} and \cite{Ahmadinezhad2} for the construction of Sarkisov links in terms of Cox rings and also for weighted projective space bundles. 

\begin{Def}
A {\it weighted projective space bundle} over $\mbP^m$ is a toric variety $W$ defined by the following data.
\begin{enumerate}
\item The Cox ring of $W$ is $\Cox (W) = \mbC [x_0,\dots,x_m,y_0,\dots,y_n]$.
\item The irrelevant ideal of $W$ is $I = (x_0,\dots,x_m) \cap (y_0,\dots,y_n)$.
\item The action of $(\mbC^*)^2$ on $\mbA^{m+n+2} = \Spec \Cox (W)$ is given by
\[
\begin{pmatrix}
1 & \dots & 1 & -\omega_0 & -\omega_1 & \dots & -\omega_n \\
0 & \dots & 0 & c_0 & c_1 & \dots & c_n
\end{pmatrix},
\]
where $\omega_i$ are non-negative integers and $c_i$ are positive integers.
\end{enumerate}
\end{Def}

Let $W$ be a weighted projective space bundle as in the above definition.
It is naturally isomorphic to the geometric quotient of $\Spec \Cox (W) \setminus V (I)$ by the action of $(\mbC^*)^2$.
The homogeneous coordinate ring $\Cox (W)$ has a $\Div (W)$-grading $\Cox (W) = \bigoplus_{\alpha \in \Div (X)} \Cox (W)_{\alpha}$, where $\Div (W) \cong \mbZ^2$ is the Weil divisor class group of $W$.
For integers $e_1$ and $e_2$, we denote by 
\[
\begin{pmatrix}
e_1 \\
e_2
\end{pmatrix} 
\subset 
\begin{pmatrix}
1 & \dots & 1 & -\omega_0 & -\omega_1 & \dots & -\omega_n \\
0 & \dots & 0 & c_0 & c_1 & \dots & c_n
\end{pmatrix},
\] 
a Weil divisor on $W$ whose linear equivalence class belongs to ${}^t (e_1, e_2) \in \mbZ^2 \cong \Div (X)$.
We note that the projection to the coordinates $x_0,\dots,x_m$ yields a morphism $W \to \mbP^m$ whose fiber is a weighted projective space $\mbP (c_0,c_1,\dots,c_n)$.

In the following, let $X = X_{d_1,d_2} \subset \mbP (a_0,a_1,\dots,a_5)$ be a member of $\mcG_i$ with $i \in I_{dP}$ and $\msp$ a distinguished singular point marked $\operatorname{dP}_k$ in the third column of the big table.
We observe that $d_1 = d_2$ and we put $d := d_1 = d_2$.

\begin{Lem}
We can choose homogeneous coordinates $x_0,\dots,x_5$ with the following properties.
\begin{enumerate}
\item $\msp = \msp_5$.
\item $a_0 = a_1$ and $a_5 \ge a_j$ for $j = 0,\dots,4$.
\item Defining polynomials of $X$ can be written as $F_1 = x_5 x_0 + G_1$ and $F_2 = x_5 x_1 + G_2$, where $G_1, G_2 \in \mbC [x_0,\dots,x_4]$.
\end{enumerate}
\end{Lem}

\begin{proof}
This follows from quasismoothness of $X$ and $\deg F_1 = \deg F_2$. 
\end{proof} 

We fix homogeneous coordinates of $\mbP (a_0,\dots,a_5)$ as above.
Let $\mbP := \mbP (a_0,\dots,a_4)$ be the weighted projective space with homogeneous coordinates $x_0,\dots,x_4$ and $Z := (x_0 G_2 - x_1 G_1 = 0) \subset \mbP$.
Then $Z$ is the image of the projection $\pi \colon X \ratmap \mbP$ from $\msp$.
The sections $x_0,\dots,x_4$ lift to plurianticanonical sections on $Y$ since $a_j \le a_5$ for $j = 0,\dots,4$, and they define a birational morphism $\psi \colon Y \to Z$ which contracts finitely many curves $(x_0 = x_1 = G_1 = G_2 = 0)$.
It follows that the Kawamata blowup $\varphi \colon Y \to X$ resolves the indeterminacy of $\pi$ and the induced morphism $\psi$ is a flopping contraction.

Let $W$ be the weighted projective space bundle over $\mbP^1$ defined by
\[
\begin{pmatrix}
1 & 1 & 0 & 0 & 0 & -1 \\
0 & 0 & a_2 & a_3 & a_4 & e
\end{pmatrix},
\]
with the generators of Cox ring $w_0,w_1,x_2,x_3,x_4$ and $v$, where $e := a_0 = a_1$.
Let $X'$ be the hypersurface of $W$ defined by $w_0 G'_2 - w_1 G'_1 = 0$, where $G'_j := G_j (w_0 v,w_1 v, x_2,x_3,x_4)$, which is of type
\[
\begin{pmatrix}
1 \\
d
\end{pmatrix}
\subset 
\begin{pmatrix}
1 & 1 & 0 & 0 & 0 & -1 \\
0 & 0 & a_2 & a_3 & a_4 & e
\end{pmatrix}.
\]
The open subset $X' \setminus (v=0)_{X'}$ is isomorphic to $Z \setminus (x_0=x_1=0)$ and has only terminal cyclic quotient singularities.
Assume that $a_j \ge 2$ for some $j = 2,3,4$.
In this case, the set $(v = x_k = x_j = 0)_{X'}$, where $\{j,k,l\} = \{2,3,4\}$, consists of a unique point denoted by $\msp'_j$ and the singularity of $X'$ at $\msp'_j$ is of type $\frac{1}{a_j} (a_k,a_l,e)$, which is terminal.  
Quasismoothness of $X$ along $(x_0 = x_1 = 0)$ implies smoothness of $X'$ along $(v=0)_{X'} \setminus \{\msp'_j \mid a_j \ge 2\}$.
This shows that $X'$ has only terminal cyclic quotient singularities and in particular $X'$ is $\mbQ$-factorial.

There are two projections $\Pi \colon W \to \mbP^1$ and $\Psi \colon W \to \mbP (e, e, a_2,a_3,a_4)$ defined as follows:
$\Pi$ is the projection to the coordinates $w_0, w_1$ and $\Psi$ is defined by 
\[
(w_0,w_1; x_2,x_3,x_4,v) \mapsto (w_0 v, w_1 v,x_2,x_3,x_4).
\]
We define $\pi' = \Pi|_{X'} \colon X' \to \mbP^1$ and $\psi' = \Psi |_{X'} \colon X' \to \mbP (e,e,a_2,a_3,a_4)$.
We see that $\pi' \colon X' \to \mbP^1$ is a del Pezzo fibration of degree $k = d/(e a_2 a_3 a_4)$.
The image of $\psi'$ is $Z$ and we claim that  $\psi' \colon X' \to Z$ is a flopping contraction.
Indeed we see that the anticanonical divisor $-K_{X'}$ is the restriction of a divisor on $W$ of type ${}^t (0,1)$ by the adjunction formula.
It follows that $w_0 v$, $w_1 v$, $x_2,x_3,x_4$ are plurianticanonical sections on $X'$ and they define the morphism $\psi'$, which implies that $\psi'$ is a $K_{Y'}$-trivial contraction.
The morphism $\Psi \colon W \to \mbP (a,a,a_2,a_3,a_4)$ contracts the divisor $(v = 0) \cong \mbP^1 \times \mbP (a_2,a_3,a_4)$ to the surface $(x_0 = x_1 = 0) \cong \mbP (a_2,a_3,a_4)$.
It follows that $\psi'$ contracts finitely many curves $(v = G'_1 = G'_2 = 0)$.
This shows that $\psi$ is small and thus it is a flopping contraction.

\begin{Thm} \label{SLdP}
Let $X = X_{d_1,d_2} \subset \mbP (a_0,a_1,\dots,a_5)$ be a member of the family $\mcG_i$ with $i \in I_{dP}$ and $\msp$ a distinguished singular point of $X$ marked $\dP_k$ in the third column of the big table.
Then, there is a Sarkisov link
\[
\xymatrix{
Y \ar[d]_{\varphi} \ar@{-->}[rd]^{\tau} & \\
X \ar@{-->}[r]_{\sigma_{\msp} \hspace{4mm}}& X'/\mbP^1}
\]
centered at $\msp$, to a del Pezzo fiber space $\pi' \colon X' \to \mbP^1$ of degree $k$, where $\varphi \colon Y \to X$ is the Kawamata blowup at $\msp$ and $\tau$ is a flop.
\end{Thm}

\begin{proof}
By the above argument, we have the following diagram
\[
\xymatrix{
Y \ar[d]_{\varphi} \ar[rd]^{\psi} \ar@{-->}[rr] & & X' \ar[ld]_{\psi'} \ar[d]^{\pi'} \\
X & Z & \mbP^1}
\]
where $\varphi$ is the Kawamata blowup at $\msp$ and $\psi$, $\psi'$ are flopping contractions.
This shows that the induced map $\sigma_{\msp} \colon X \ratmap X'$ is a Sarkisov link.
\end{proof}

\begin{Rem} \label{rem:dPCox}
In \cite{Ahmadinezhad}, the Sarkisov link $\sigma_{\msp} \colon X \ratmap X'$, or more precisely its inverse $\sigma^{-1}_{\msp}$, is constructed in terms of the Cox ring in the case where $\dP_2$ is marked in the third column of the big table.
\end{Rem}

\section{Birational involutions} \label{sec:birinv}

In this section, we construct a birational involution of a $\mbQ$-Fano WCI centered at a singular point, which is a Sarkisov link of type II.
The argument of this section is quite similar to that of \cite{CPR} and the basic idea behined this is the theory of ``unprojection".
The constructions of quadratic and elliptic involutions in this section provide explicit examples of type II unprojection.
Note also that the construction of Sarkisov links in the previous section provides explicit examples of type I unprojection (or Kustin-Miller unprojection).
We refer the readers to \cite{PR} and \cite{Pap} for detailed accounts of type I and type II unprojections, respectively, and also to \cite{Reid2} for an introduction to this subject.
Let $X = X_{d_1,d_2} \subset \mbP (a_0,\dots,a_5)$ be a member of $\mcG_i$ with defining polynomials $F_1$ and $F_2$ of degree $\deg F_i = d_i$.

\subsection{Quadratic involutions} \label{sec:QI}

Let $\msp$ be a singular point of $X$ marked Q.I. in the third column of the big table.
Throughout this subsection, we assume the following additional assumption.

\begin{Assumption} \label{assmpQI}
If $X$ is a member of $\mcG_{30}$ or $\mcG_{40}$ (resp.\ $\mcG_{64}$), then the monomial $y^2 z$ (resp.\ $z^2 s$) appears in one of the defining polynomials of $X$ with non-zero coefficient.
\end{Assumption}

\begin{Rem}
Assumption \ref{assmpQI} is necessary to prove Lemma \ref{QIcoord} below for a member of $\mcG_{30}$ or $\mcG_{40}$ (resp.\ $\mcG_{64}$) and its singular point of type $\frac{1}{3} (1,1,2)$ (resp.\ $\frac{1}{5} (1,2,3)$).
If Assumption \ref{assmpQI} fails, then we can exclude those points as a maximal center (see Section \ref{sec:singpt3}).
\end{Rem}

\begin{Lem} \label{QIcoord}
We can choose homogeneous coordinates $x_{i_0}, \dots, x_{i_3}, \xi$ and $\zeta$ such that
\begin{enumerate}
\item $\msp$ is the vertex $\msp_{\xi}$ at which only the coordinate $\xi$ does not vanish,
\item $a_{i_2} \le a_{i_3} < \deg \xi < \deg \zeta < 2 \deg \xi$, and
\item defining polynomials $F_1$ and $F_2$ of $X$ are written as
\[
\begin{split}
F_1 &= \xi^2 x_{i_0} + \xi a + \zeta^2 + b, \\
F_2 &= \xi x_{i_1} + \zeta c + d,
\end{split}
\]
where $a,b,c,d \in k [x_{i_0},x_{i_1},x_{i_2},x_{i_3}]$ are homogeneous polynomials.
Here, we do not assume that $d_1 \le d_2$.
\end{enumerate}
\end{Lem}

\begin{proof}
Possibly re-choosing homogeneous coordinates, we can assume that $\msp$ is a vertex.
Let $\xi$ be the homogeneous coordinate which does not vanish at $\msp$.
Assume first that $X$ is a member of $\mcG_i$ with $i \in I^* \setminus \{10,15\}$.
By quasismoothness of $X$ and Assumption \ref{assmpQI}, we can choose homogeneous coordinates $x_{i_0}, \dots,x_{i_3}$ and $\zeta$ so that $a_{i_2} \le a_{i_3} \le \deg \zeta$ and defining polynomials are written as
\[
\begin{split}
F_1 &= \xi^2 x_{i_0} + \xi ( \zeta \alpha + a) + \zeta^2 + b, \\
F_2 &= \xi x_{i_1} + \zeta c + d,
\end{split}
\]
where $\alpha,a,b,c,d \in k [x_{i_0},x_{i_1},x_{i_2},x_{i_3}]$.
Replacing $\zeta \mapsto \zeta - \frac{1}{2} \xi \alpha$, $x_{i_0} \mapsto x_{i_0} + \frac{1}{4} \alpha^2$ and $x_{i_1} \mapsto x_{i_1} + \frac{1}{2} \alpha c$, we get the desired equations.
Assume that $X = X_{6,7} \subset \mbP (1^2,2^2,3,5)$ is a member of $\mcG_{15}$ and $\msp$ is a point of type $\frac{1}{2} (1,1,1)$.
We may assume that $\msp = \msp_2$.
By quasismoothness, we can write $F_1 = y_0^2 y_1 + y_0 (z \alpha + a) + z^2 + b$, where $\alpha, a,b \in k [x_0,x_1,y_1,s]$.
We see that $y_0 s \in F_2$ because otherwise $X$ contains the WCI curve $(x_0 = x_0 = y_1 = z = 0)$ of type $(1,1,2,3)$ and this is impossible by $(\mathrm{C}_2)$.
It follows that we can write $F_2 = y_0 s + z c + d$ for some $c, d \in k[x_0,x_1,y_1,s]$.
Replacing $z, y_1,s$, we get the desired equations.
The case where $X \in \mcG_{10}$ can be proved similarly.
The assertion (2) can be checked easily.
\end{proof}

In the following, we fix homogeneous coordinates and defining polynomials of $X$ as in Lemma \ref{QIcoord}.
We put $a_{\xi} := \deg \xi$ and $a_{\zeta} := \deg \zeta$.
Let $\varphi \colon Y \to X$ be the Kawamata blowup of $X$ at $\msp$ with the exceptional divisor $E$.
We explain the argument of this subsection.
We shall construct a homogeneous polynomial $v$ which vanishes along $E$ to a sufficiently large order and then construct the anticanonical model $Z$ of $Y$ explicitly.
We observe that the anticanonical map is a birational morphism $\psi \colon Y \to Z$ and $Z$ is a weighted hypersurface in $\mbP (a_{i_0},a_{i_1},a_{i_2},a_{i_3},\deg v)$ with homogeneous coordinates $x_{i_0},\dots,x_{i_3}$ and $v$ defined by a quadratic equation with respect to $v$.
As a consequence, we have the diagram
\[
\xymatrix{
Y \ar[d]_{\varphi} \ar[rd]^{\psi} \ar@{-->}[rrr] & & & Y \ar[d]^{\varphi} \ar[ld]_{\psi} \\
X & Z \ar[r]^{\iota_Z} & Z & X}
\]
where $\iota_Z$ is the biregular involution given by the double cover $Z \to \mbP (a_{i_0},a_{i_1},a_{i_2},a_{i_3})$. 
Finally, we conclude that if $\psi$ contracts a divisor, then $\msp$ is not a maximal center, and if $\psi$ is small, then $\psi^{-1} \circ \iota_Z \circ \psi \colon Y \ratmap Y$ is the flop and the above diagram gives the Sarkisov link $\iota_{\msp} \colon X \ratmap X$ centered at $\msp$.

Multiplying $F_1$ by $x_{i_1}$, we get
\[
x_{i_1} F_1 = x_{i_1} \xi (x_{i_0} \xi + a) + x_{i_1} (\zeta^2 + b).
\]
Subtracting $h_1 F_2$ from $x_{i_1} F_1$, where $h_1 := x_{i_0} \xi +  a$, we get
\[
x_{i_1} F_1 - h_1 F_2 
= - (c \zeta + d)(x_{i_0} \xi + a) + x_{i_1} (\zeta^2 + b).
\]
Filtering off terms divisible by $\zeta$, we get
\begin{equation} \label{eqQI1}
x_{i_1} F_1 - h_1 F_2 
= v \zeta - d (x_{i_0} \xi + a) + x_{i_1} b,
\end{equation}
where
\begin{equation} \label{eqQI2}
v := - x_{i_0} c \xi + x_{i_1} \zeta - a c.
\end{equation}

\begin{Lem}
The sections $x_{i_0}, \dots, x_{i_3}$ and $v$ lift to plurianticanonical sections of $Y$.
\end{Lem}

\begin{proof}
We see that $x_{i_2}$, $x_{i_3}$ and $\zeta$ can be chosen as local coordinates of $X$ at $\msp$ and $\varphi \colon Y \to X$ is the weighted blowup with $\wt (x_{i_2}, x_{i_3}, \zeta) = 1/a_{\xi} (a_{i_2}, a_{i_3}, a_{\zeta} - a_{\xi})$.
It follows that $x_{i_2}$, $x_{i_3}$ and $\zeta$ vanish along $E$ to order $a_{i_2}/a_{\xi}$, $a_{i_3}/a_{\xi}$ and $(a_{\zeta} - a_{\xi})/a_{\xi}$, respectively.
By the equations $F_1 = 0$ and $F_2 = 0$, and by explicit observations in each case, we see that $x_{i_0}$ and $x_{i_1}$ vanish along $E$ to order at least $a_{i_0}/a_{\xi}$ and $a_{i_1}/a_{\xi}$, respectively.
By the equation \eqref{eqQI1}, we have $v \zeta = d (x_{i_0} \xi + a) - x_{i_1} b$.
We see that $\deg (d a) = \deg (x_{i_1} b) > \deg (d x_{i_0}) = \deg v + a_{\zeta} - a_{\xi}$ and thus the section on the right-hand side of the equation vanishes along $E$ to order $\ge (\deg v + a_{\zeta} - a_{\xi})/a_{\xi}$.
It follows that $v$ vanishes along $E$ to order at least $\deg v/a_{\xi}$ and lifts to a plurianticanonical section of $Y$.
\end{proof}

We can view \eqref{eqQI1}, \eqref{eqQI2} and the defining equation $F_2 = 0$ as relations which are linear in variables $\xi$ and $\zeta$:
\[
M \cdot 
\begin{pmatrix}
\xi \\
\zeta \\
1
\end{pmatrix}
= 
\begin{pmatrix}
0 \\
0 \\
0
\end{pmatrix},
\]
where
\[
M := 
\begin{pmatrix}
x_{i_1} & c & d \\
- x_{i_0} d & v & - a d + x_{i_1} b \\
x_{i_0} c & - x_{i_1} & v + a c
\end{pmatrix}.
\]
We see that the determinant $\det M$ is divisible by $x_{i_1}$ and we define a homogeneous polynomial $G = G (x_{i_0},x_{i_1},x_{i_2},x_{i_3},v) := (\det M)/x_{i_1}$.
We have
\[
G = v^2 + a c v + x_{i_1}(x_{i_1} b - a d) + x_{i_0} (b c^2 + d^2).
\]

Let $\mbP = \mbP (a_{i_0},a_{i_1},a_{i_2},a_{i_3},a_v)$ be the weighted projective space with homogeneous coordinates $x_{i_0},\dots,x_{i_3}$ and $v$, where $a_v = \deg v$, and let $\pi \colon X \ratmap \mbP$ be the rational map defined by $x_{i_0}, \dots, x_{i_3}$ and $v$.
The image of $\pi$ is the weighted hypersurface $Z := (G = 0) \subset \mbP$ and the rational map $\pi \colon X \ratmap Z$ is birational.

\begin{Thm} \label{birinvQI}
Let $X$ be a member of $\mcG_i$ with $i \in I^*$ satisfying Assumption \ref{assmpQI} and $\msp$ a singular point of $X$ marked Q.I. in the big table.
Then, one of the following holds.
\begin{enumerate}
\item $\msp$ is not a maximal center.
\item There is a birational involution $\iota_{\msp} \colon X \ratmap X$ centered at $\msp$ which is a Sarkisov link of type II.
\end{enumerate}
\end{Thm}

\begin{proof}
The Kawamata blowup $\varphi \colon Y \to X$ of $X$ at $\msp$ resolves the indeterminacy of the birational map $\pi \colon X \ratmap Z$ and let $\psi \colon Y \to Z$ be the induced birational morphism.
We see that $\psi$ is not an isomorphism since it contracts the strict transform on $Y$ of $(x_{i_1} = d = c = 0)_X$ to $(x_{i_1} = d = c = v = 0) \subset Z$.

Assume first that $\psi$ contracts a divisor.
If $\msp$ is a maximal center, then the $2$-ray game starting with $\varphi$ ends with a Mori fiber space and one obtains a Sarkisov link.
In particular, if $\msp$ is a maximal center, then the morphism $\psi$ is either a small or a $K_Y$-negative contraction.
By the construction of $\psi$, it is a $K_Y$-trivial contraction and it is not small by the assumption.
This is a contradiction and $\msp$ is not a maximal center.

Assume that $\psi$ is small.
Let $\iota_Z \colon Z \to Z$ be the biregular involution over $\mbP (a_{i_0},\dots,a_{i_3})$ which interchanges two roots of the equation $G = 0$.
Put $\iota_Y = \psi^{-1} \circ \iota_Z \circ \psi \colon Y \ratmap Y$.
Then, $\iota_Y$ is a flop and the diagram
\[
\xymatrix{
Y \ar[d]_{\varphi} \ar[rd]^{\psi} \ar@{-->}[rrr]^{\iota_Y} & & & Y \ar[d]^{\varphi} \ar[ld]_{\psi} \\
X & Z \ar[r]^{\iota_Z} & Z & X}
\]
gives a Sarkisov link of type II.
We explain that $\iota_Y$ is indeed a flop.
Let $\Gamma \subset Y$ be an irreducible curve which is contracted to a point $\msq \in Z$ by $\psi$.
Then $Z$ is not $\mbQ$-factorial at $\msq$ and thus $Z$ cannot be quasismooth at $\msq$.
In particular, $\prt G/\prt v = 2 v + a c$ vanishes at $\msq$.
It follows that $\iota_Z (\msq) = \msq$ since $\iota_Z$ is defined by $\iota_Z^* x_{i_j} = x_{i_j}$ and $\iota_Z^*v = - v - ac$.
Hence $\iota_Z^*$ operates the class group of $Z$ at $\msq$ as $- \operatorname{Id}$.
This shows that $\iota_Y$ is a flop. 
The desired birational involution $\iota_{\msp}$ of $X$ is the one induced by $\iota_Y$.
\end{proof}

\begin{Rem}
The case (1) in Theorem \ref{birinvQI} does occur at least for some very special members.
For example, let $X = X_{14,18} \subset \mbP (1,2,6,7,8,9)$ be a member of $\mcG_{75}$ and $\msp = \msp_4$ the point of type $\frac{1}{8} (1,1,7)$.
Defining polynomials of $X$ can be written as $F_1 = t z + u c + d$ and $F_2 = t^2 y + t a + u^2 + b$.
We assume that $c = 0$.
Note that $X$ is quasismooth if $a,b$ and $d$ are general.
In this case, the divisor $(z = 0)_X$ is contracted to the curve $(z = d = v = 0) \subset Z$ by $\pi \colon X \ratmap Z$.
Therefore, the Kawamata blowup of $X$ at $\msp$ leads to a bad link and $\msp$ is not a maximal center.
\end{Rem}

\subsection{Elliptic involutions}

Let $\msp$ be a singular point of $X$ marked E.I. in the third column of the big table.

\begin{Lem} \label{EIcoord}
We can choose homogeneous coordinates $x_{i_0}, x_{i_1}, x_{i_2}, \xi, \zeta$  and $\eta$ of $\mbP (a_0,\dots,a_5)$ such that
\begin{enumerate}
\item $\msp$ is the vertex $\msp_{\xi}$ at which only the coordinate $\xi$ does not vanish,
\item $a_{i_1} \le a_{i_2} < a_{\xi} < a_{\eta} < a_{\zeta}$ and $a_{\eta} < 2 a_{\xi}$, and
\item defining polynomials $F_1$ and $F_2$ of $X$ are written as
\[
\begin{split}
F_1 &= \xi \zeta + \eta^2 + \eta a + b, \\
F_2 &= \xi^2 x_{i_0} + \xi (\eta c + d) + \zeta \eta + \eta^2 e + \eta f + g,
\end{split}
\]
where $a, b, \dots, g \in \mbC [x_{i_0}, x_{i_1}, x_{i_2}]$ are homogeneous polynomials.
\end{enumerate}
\end{Lem}

\begin{proof}
Assume that $X$ is a member of $\mcG_i$ with $i \in I^* \setminus \{16,27\}$.
By quasismoothness, we can write defining polynomials as
\[
\begin{split}
F_1 &= \xi \zeta + \zeta \alpha + \eta^2 + \eta a + b, \\
F_2 &= \xi^2 x_{i_0} + \xi (\zeta \beta + \eta c + d) + \zeta \eta + \zeta \gamma + \eta^2 e + \eta f + g,
\end{split}
\]
for some homogeneous polynomials $\alpha, \beta, \gamma, a, b, \dots, g \in \mbC [x_{i_0}, x_{i_1}, x_{i_2}]$.
Replacing $F_2$ with $F_2 - \beta F_1$, we may assume that $\beta = 0$.
Replacing $\xi \mapsto \xi - \alpha$ and $\eta \mapsto \eta - \gamma$, we get the desired equations.

Let $X = X_{8,9} \subset \mbP (1,2,3^3,4,5)$ be a member of $\mcG_{27}$ and $\msp$ a point of type $\frac{1}{3} (1,1,2)$.
We can assume that $\msp = \msp_2$.
If the coefficient of $z_0 t$ in the defining polynomial of degree $8$ is zero, then $X$ contains the WCI curve $(x = y = z_1 = s = 0)$ of type $(1,2,3,4)$.
This cannot happen because of ($\mathrm{C}_2$).
It follows that we can write $F_1 = z_0 t + s^2 + s a + b$ and $F_2 = z_0^2 z_1 + z_0 (s c + d) + t s + s^2 e + s f + g$ for homogeneous polynomials $a,b,\dots,g \in \mbC [x,y,z_1]$.

Finally, let $X = X_{6,7} \subset \mbP (1^2,2,3^2,4)$ be a member of $\mcG_{16}$ and $\msp$ a point of type $\frac{1}{2} (1,1,1)$.
We can assume that $\msp = \msp_2$.
Let $F_1$ and $F_2$ be defining polynomials of $X$ with degree $6$ and $7$, respectively.
We see that either $y^2 z_0 \in F_2$ or $y^2 z_1 \in F_2$ because otherwise $X$ contains a WCI curve of type $(1,1,3,4)$, which is impossible by $(\mathrm{C}_2)$.
Hence we can write $F_1 = y s + f_6$ and $F_2 = y^2 z_0 + y g_5 + g_7$, where $f_6$, $g_5$ and $g_7$ do not involve $y$.
We see that $z_1^2 \in f_6$ because otherwise $X$ contains the WCI curve $(x_0 = x_1 = z_0 = s = 0)$ of type $(1,1,3,4)$.
We see that $z_1 s \in g_7$ because otherwise $X$ contains the WCI curve $(x_0 = x_1 = z_0 = F_1 = 0)$ of type $(1,1,3,6)$.
Therefore, we can write $F_1 = y s + z_1^2 + z_1 a + b$ and $F_2 = y^2 z + y (z_1 c + d) + s z_1 + z_1^2 + z_1 f + g$ for some homogeneous polynomials $a, b, \dots,g \in \mbC [x_0,x_1,z_0]$ after re-choosing homogeneous coordinates.
This completes the proof.
\end{proof}

We fix homogeneous coordinates and defining polynomials as in Lemma \ref{EIcoord}.
We shall construct sections $v = \zeta^2 + \cdots$ and $w = \zeta^3 + \cdots$ which vanish along $E$ to a sufficiently large order.
The argument is similar to the case of quadratic involutions, but it becomes quite complicated.
We shall show that the anticanonical map $\psi \colon Y \ratmap Z$ is a birational morphism and $Z$ is the weighted hypersurface in $\mbP (a_{i_0},a_{i_1},a_{i_2},2 a_{\zeta},3 a_{\zeta})$ with homogeneous coordinates $x_{i_0},x_{i_1},x_{i_2},v$ and $w$ defined by an equation which is quadric in $w$ and is cubic in $v$.
Let $\iota_Z \colon Z \to Z$ be the biregular involution given by the double cover $Z \to \mbP (a_{i_0},a_{i_1},a_{i_2},2 a_{\zeta})$.
Then, as in the case of quadratic involutions, we conclude that the map $\psi^{-1} \circ \iota_Z \circ \psi \colon Y \ratmap Y$ is a flop and this gives a Sarkisov link $\iota_{\msp} \colon X \ratmap X$ centered at $\msp$ in the case where $\psi$ is small.

Multiplying $\zeta$ to $F_2$ and then eliminating the terms involving $\xi \zeta$ by subtracting $(x_{i_0} \xi + c \eta + d) F_1$, we get
\[
\zeta F_2 - h_1 F_1
= \zeta (\zeta \eta + e \eta^2 + f \eta + g) - (x_{i_0} \xi + c \eta + d)(\eta^2 + a \eta + b),
\]
where $h_1 := x_{i_0} \xi + c \eta + d$.
Filtering off terms divisible by $\eta$, we get
\begin{equation} \label{eqEI1}
\zeta F_2 - h_1 F_1
= v \eta - b (x_{i_0} \xi + d) + g \zeta,
\end{equation}
where we define
\begin{equation} \label{eqEI2}
v := - (\xi x_{i_0} + \eta c + d)(\eta + a) - b c + \zeta^2 + \zeta \eta e + \zeta f.
\end{equation}
Multiplying $\zeta$ to the equation \eqref{eqEI1} and then eliminating $\zeta^2$ and $\xi \zeta$ in terms of $v$ and $F_1$, respectively, we get
\[
\begin{split}
\zeta^2 F_2 - h_2 F_1 
= v \zeta \eta + b & x_{i_0} (\eta^2 + a \eta + b) - b d \zeta \\
&+ g (v + (x_{i_0} \xi + c \eta + d)(\eta + a) + b c - e \zeta \eta - f \zeta),
\end{split}
\]
where $h_2 := \zeta h_1 - x_{i_0} b$.
Filtering off terms divisible by $\eta$, we get
\begin{equation} \label{eqEI3}
\zeta^2 F_2 - h_2 F_1 
= w \eta + x_{i_0} b^2 - b d \zeta + g (v + a (x_{i_0} \xi + d) + b c - f \zeta),
\end{equation}
where
\begin{equation} \label{eqEI4}
w := x_{i_0} g \xi + (v - e g) \zeta + (x_{i_0} b + c g) \eta + x_{i_0} a b + g (d + ac).
\end{equation}
On $X = (F_1 = F_2 = 0)$, the equations \eqref{eqEI1}, \eqref{eqEI3} and \eqref{eqEI4} can be seen as linear relations with respect to $\xi, \zeta$ and $\eta$.
We construct one more relation.
Multiplying $\xi$ to \eqref{eqEI4} and subtracting $g F_2$, we have
\[
\begin{split}
w \xi - g F_2 = (v - e g) \xi \zeta - g \zeta & \eta + x_{i_0} b \xi \eta - e g \eta^2 + (x_{i_0} a b + a c) \xi - f g \eta - g^2.
\end{split}
\]
Multiplying $\eta$ to both sides of the equation \eqref{eqEI1} and adding it to the last displayed equation, we get
\[
w \xi + (\zeta \eta - g) F_2 - \eta h_1 F_1
= (v - e g) (\xi \zeta + \eta^2) + a (x_{i_0} b + c g) \xi - (b d + f g) \eta - g^2.
\]
Finally, using $\xi \zeta + \eta^2 = F_1 - (a \eta + b)$, we obtain
\begin{equation} \label{eqEI5}
\begin{split}
(\zeta \eta - g) F_2 - h_3 F_1
= (a & (x_{i_0} b + c g) - w) \xi \\
&- (a (v - e g) + b d + f g)\eta - b (v-e g) - g^2,
\end{split}
\end{equation}
where $h_3 := h_1 \eta + v - e g$.
We see that $\deg v = 2 a_{\zeta}$ and $\deg w = 3 a_{\zeta}$ since $v = \zeta^2 + \cdots$ and $w = v \zeta + \cdots = \zeta^3 + \cdots$.

\begin{Lem} \label{EIpacsec}
The sections $x_{i_0}, x_{i_1}, x_{i_2}, v$ and $w$ lift to plurianticanonical sections of $Y$.
\end{Lem}

\begin{proof}
By Lemma \ref{EIcoord}, we can choose $x_{i_1}, x_{i_2}$ and $\eta$ as local coordinates of $X$ at $\msp$ and the Kawamata blowup $\varphi \colon Y \to X$ at $\msp$ is the weighted blowup with $\wt (x_{i_1}, x_{i_2}, \eta) = 1/a_{\xi} (a_{i_1}, a_{i_2}, a_{\eta}-a_{\xi})$.
It follows that $x_{i_0}$, $x_{i_1}$ and $\eta$ vanish along $E$ to order $a_{i_1}/a_{\xi}$, $a_{i_2}/a_{\xi}$ and $(a_{\eta} - a_{\xi})/a_{\xi}$, respectively.
If $X$ belongs to $\mcG_{16}$, then $x_{i_0}$ vanishes along $E$ to order at least $a_{i_0}/a_{\xi}$ by the equation $F_2 = 0$.
If $X$ belongs to either $\mcG_{27}$ or $\mcG_{41}$, then the divisor $(x_{i_0} = 0)_X$ is Cartier at $\msp$, otherwise $a_{i_0} < a_{\xi}$.
In any case, we see that $x_{i_0}$ vanishes along $E$ to order at least $a_{i_0}/a_{\xi}$.
By the equation $F_1 = 0$, the section $\zeta$ vanishes along $E$ to order at least $2 (a_{\eta} - a_{\xi})/a_{\xi}$.

By the equation \eqref{eqEI1}, the section $v \eta$ vanishes along $E$ to order at least
\[
\begin{split}
\min \{ \deg x_{i_0} b, \deg b d, 2 (a_{\eta} - a_{\xi}) + \deg g \}/a_{\xi} &= (a_{i_0} + \deg b)/a_{\xi} \\
&= (\deg v + (a_{\eta} - a_{\xi}))/a_{\xi}.
\end{split}
\] 
It follows that $v$ vanishes along $E$ to order at least $\deg v /a_{\xi}$.
Arguing similarly for the equation \eqref{eqEI3}, we see that $w$ vanishes along $E$ to order at least $\deg w/a_{\xi}$.
This completes the proof.
\end{proof}

Let $\mbP = \mbP (a_{i_0},a_{i_1},a_{i_2},2 a_{\zeta}, 3 a_{\zeta})$ be the weighted projective space with homogeneous coordinates $x_{i_0},x_{i_1},x_{i_2},v$ and $w$, and let $\pi \colon X \ratmap \mbP$ the rational map defined by sections $x_{i_0}, x_{i_1},x_{i_2},v$ and $w$.
We have four relations \eqref{eqEI1}, \eqref{eqEI3}, \eqref{eqEI4} and \eqref{eqEI5} which are linear in $\xi$, $\zeta$ and $\eta$.
Put $v' := v - e g$ and $w' := w - a (x_{i_0} b + c g)$.
Then we can write down those relations as
\[
M \cdot 
\begin{pmatrix}
\xi \\
\zeta \\
\eta \\
1
\end{pmatrix}
=
\begin{pmatrix}
0 \\
0 \\
0 \\
0
\end{pmatrix}
\]
where $M$ is the matrix
\[
\begin{pmatrix}
- x_{i_0} b & g & v & - b d \\
x_{i_0} a g & - (b d + f g) & w & g (v + a d + b c) + x_{i_0} b^2 \\
x_{i_0} g & v' & x_{i_0} b + c g & g d - w' \\
w' & 0 & a v' + b d + f g & b v' + g^2
\end{pmatrix}.
\]
We see that the determinant $\det M$ is divisible by $w' g$ and we set $G := (\det M) / w' g$.
The homogeneous polynomial $G = G (x_{i_0}, x_{i_1}, x_{i_2}, v , w)$ is of degree $6 a_{\zeta}$, and it is quadratic in $w$ and is cubic in $v$. 
We define $Z = (G = 0) \subset \mbP$.
Then $Z$ is the image of $\pi$ and the induced map $\pi \colon X \ratmap Z$ is birational.

\begin{Thm} \label{birinvEI}
Let $X$ be a member of $\mcG_i$ with $i \in I^*$ and $\msp$ a singular point of $X$ marked E.I. in the third column of the big table.
Then one of the following holds.
\begin{enumerate}
\item The point $\msp$ is not a maximal center.
\item There is a birational involution $\iota_{\msp} \colon X \ratmap X$ which is a Sarkisov link of type II.
\end{enumerate}
\end{Thm}

\begin{proof}
The proof is similar to the case of quadratic involutions.
The Kawamata blowup of $X$ at $\msp$ resolves the indeterminacy of the birational map $\pi$ and let $\psi \colon Y \to Z$ be the induced birational morphism.
Let $C$ be a component of $(b=g=0)_X$ such that $\eta \ne 0$ at the generic point of $C$.
By the relations \eqref{eqEI1} and \eqref{eqEI4}, both $v$ and $w$ vanish along $C$.
It follows that the strict transform of $C$ on $Y$ is contracted to the set $(b=g=v=w=0) \subset Z$.
In particular, $\psi$ is not an isomorphism.
If $\psi \colon Y \to Z$ contracts a divisor, then $\msp$ is not a maximal center by the same argument as in the proof of Theorem \ref{birinvQI}.
We assume that $\psi$ is small, that is, it is a flopping contraction.
Let $\sigma \colon Z \to Z$ be the biregular involution interchanging two points in the fiber of the double covering $Z \to \mbP$, and let $\iota_Y \colon Y \ratmap Y$ the induced birational map.
The same argument as in the proof of Theorem \ref{birinvQI} shows that $\iota_Y$ is a flop:
each flopping curve on $Y$ is contracted by $\psi$ to a point $\msq$ such that $\iota_Z (\msq) = \msq$ and thus $\iota_Z^*$ operates on the class group of $Z$ at $\msq$ as $-\operatorname{Id}$, which implies that $\iota_Y$ is a flop. 
This gives rise to a Sarkisov link and the induced map $\iota_{\msp} \colon X \ratmap X$ is the desired birational involution.
\end{proof}

\begin{Rem}
In Theorem \ref{birinvEI}, we do not describe the locus contracted by $\psi$.
We refer the reader to \cite[Section 4.10]{CPR} for more detailed descriptions of elliptic involutions of $\mbQ$-Fano weighted hypersurfaces, such as the description of flopping curves etc.
\end{Rem}

\section{Exclusion of curves} \label{sec:curve}

The aim of this section is to exclude curves as maximal centers.

\subsection{Exclusion of most of the curves}

The following criterion enables us to exclude most of the curves as maximal centers.

\begin{Lem} \label{rawexclcurve}
If a curve $\Gamma \subset X$ is a maximal center, then $\Gamma$ is contained in the smooth locus of $X$ and $(A^3) > (A \cdot \Gamma)$.
In particular, $(A^3) > 1$.
\end{Lem}

\begin{proof}
See STEP $1$ of the proof of \cite[Theorem 5.1.1]{CPR}.
\end{proof}

\begin{Cor} \label{exclcurve0}
Let $X$ be a member of $\mcG_i$ with $i \in I^*$ and $\Gamma$ an irreducible curve on $X$.
Then $\Gamma$ is not a maximal center except possibly for the following cases.
\begin{itemize}
\item $X$ is a member of $\mcG_4$ and $1 \le \deg \Gamma \le 2$.
\item $X$ is a member of $\mcG_i$ with $i \in \{5,6,7,9,11\}$ and $\deg \Gamma = 1$.
\end{itemize}
\end{Cor}

\begin{proof}
This follows from Lemma \ref{rawexclcurve}.
\end{proof}

\begin{Lem} \label{exclcurve1}
Let $X$ be a member of the family $\mcG_i$ with $i \in \{5, 6, 7, 9, 11\}$ and $\Gamma$ an irreducible curve of degree $1$ on $X$.
Then $\Gamma$ is not a maximal center.
\end{Lem}

\begin{proof}
The ambient weighted projective space is of the form $\mbP (1^3,2,a_4,a_5)$, where $2 \le a_4 \le a_5$.
We may assume that $\Gamma$ does not pass through any singular point of $X$.
Under the assumption, we see that $\Gamma = (x_0 = x_3 = x_4 = x_5 = 0) \cong \mbP^1$ after re-choosing homogeneous coordinates.
It follows that $a_5 A$ isolates $\Gamma$ and $M := a_5 \varphi^*(A) - E$ is a test class, where $\varphi \colon Y \to X$ is the blowup of $X$ along $\Gamma$ with the exceptional divisor $E$. 
We compute $(M \cdot B^2)$, where $B = -K_Y = \varphi^*A - E$.
We have $((\varphi^*A)^2 \cdot E) = 0$, $(\varphi^*A \cdot E^2) = -\deg \Gamma = -1$ and $(E^3) = - \deg \mcN_{\Gamma/X} = -\deg \Gamma + 2 - 2 p_a (\Gamma) = 1$.
It follows that
\[
(M \cdot B^2)
= a_5 (A^3) - (a_5 + 2) - 1 < 0,
\]
where the last strict inequality follows from explicit computations in each instance. 
By Corollary \ref{criexclmstc}, $\Gamma$ is not a maximal center.
\end{proof}

\subsection{Exclusion of curves of low degree on $X \in \mcG_4$}
We treat a member $X = X_{4,4} \subset \mbP (1^4,2,3)$ of the family $\mcG_4$ and exclude curves of degree $1$ and $2$ as a maximal center.
We denote by $\pi \colon \mbP (1^4,2,3) \ratmap \mbP^3$ the projection to the coordinates $x_0,\dots,x_3$.

\begin{Lem} \label{exclcurveno4_2}
Let $\Gamma$ be an irreducible curve of degree $2$ on a member $X$ of the family $\mcG_4$.
Then $\Gamma$ is not a maximal center.
\end{Lem}

\begin{proof}
We may assume that $\Gamma$ does not pass through the singular point of $X$.
We see that $\Gamma$ is disjoint from the indeterminacy locus of $\pi$ and the image $\pi (\Gamma)$ is either a conic or a line in $\mbP^3$.
We denote by $\varphi \colon Y \to X$ the blowup of $X$ along $\Gamma$ and set $B := -K_Y = \varphi^*A - E$.

Assume first that $\pi (\Gamma)$ is a conic.
Replacing homogeneous coordinates, we can write $\Gamma = (x_0 = q = y = z = 0)$, where $q = q (x_1,x_2)$ is of degree $2$.
It follows that $3A$ isolates $\Gamma$ and thus $M := 3 \varphi^*A - E$ is a test class.
We compute
\[
(M \cdot B^2) = 3 \times \frac{8}{3} - 5 \times 2 - (-2+2) = -2.
\]
This shows that $\Gamma$ is not a maximal center.

Assume next that $\pi (\Gamma)$ is a line.
Replacing homogeneous coordinates, we can write $\Gamma = (x_0 = x_1 = y^2 + h_4 = z = 0)$, where $h_4 = h_4 (x_2,x_3)$ is of degree $4$.
We see that $(x_0 = x_1 = z = 0)_X = \Gamma$ and $\Gamma$ is a nonsingular elliptic curve.
Thus $3 A$ isolates $\Gamma$ and $p_a (\Gamma) = 1$.
Setting $M := 3 \varphi^*A - E$, we compute
\[
(M \cdot B^2) = 3 \times \frac{8}{3} - 5 \times 2 - (-2+2-2) = 0.
\]
This shows that $\Gamma$ is not a maximal center.
\end{proof}

In the remainder of this subsection, we exclude curves of degree one on a general member $X$ of the family No.~$4$.
If a curve $\Gamma$ of degree one in $\mbP (1^4,2,3)$ is contained in some member of the family No.~$4$, then $\Gamma$ is a WCI curve of type $(1,1,2,3)$.
Let $\Gamma$ be an irreducible curve of degree one contained in $X$.
Then $\pi (\Gamma)$ is a line on $\mbP^3$.
We define $\Pi_{\Gamma} := \pi^{-1} \pi (\Gamma)$, which is isomorphic to $\mbP (1^2,2,3)$.

Let $\mcS$ be the space parametrizing members $X$ of the family No.~$4$ with defining polynomials of the form $F_1 = z x_0 + f_4$ and $F_2 = z x_1 + g_4$, where $f_4$ and $g_4$ do not involve $z$.
A WCI curve $\Gamma$ of type $(1,1,2,3)$ satisfies $\Gamma = \Pi_{\Gamma} \cap (y + h_2 = z + h_3 = 0)$ for some homogeneous polynomials $h_2$ and $h_3$ in variables $x_0,\dots,x_3$.
Those curves form a $11 = \dim \operatorname{Gr} (2,4) + 3 + 4$ dimensional family and the members parametrized by $\mcS$ containing a given $\Gamma$ form a $10$-codimensional subfamily.
If $\Pi_{\Gamma} \cap (x_0 = x_1 = 0) \cong \mbP (2,3)$, then $\Gamma = (x_2 = x_3 = y+h_2=z+h_3= 0)$.
Those curves form a $7$-dimensional family.
It follows that a general $X$ does not contain $\Gamma$ such that $\Pi_{\Gamma} \cap (x_0 = x_1 = 0) \cong \mbP (2,3)$.

\begin{Lem} \label{lemexclcdeg1no4}
Let $X$ be a general member of the family No.~$4$.
For any irreducible curve $\Gamma$ of degree one contained in the smooth locus of $X$, the following assertions hold.
\begin{enumerate}
\item A general member $S$ of the pencil $|\mcI_{\Gamma,X} (A)|$ is smooth along $\Gamma$.
\item The scheme-theoretic intersection $C := \Pi_{\Gamma} \cap X$ is reduced at the generic point of $\Gamma$.
\item 
Let $C_{\operatorname{red}} = \Gamma \cup \Gamma_1 \cup \cdots \cup \Gamma_l$ be the irreducible decomposition of the reduced scheme of $C$ and let $S \in |\mcI_{\Gamma,X} (A)|$ be a general member.
For any $i = 1,\dots,l$, there is an effective $1$-cycle $\Delta_i$ on $S$ such that $(\Gamma \cdot \Delta_i)_S \ge \deg \Delta_i := (A \cdot \Delta_i) > 0$ and $(\Gamma_j \cdot \Delta_i)_S \ge 0$ for $j \ne i$.
\end{enumerate}
\end{Lem}

\begin{proof}
Let $\Gamma$ be an irreducible curve of degree $1$ contained in the smooth locus of $X$.
We may assume that $\Pi_{\Gamma} \cap (x_0 = x_1 = 0) \cong \mbP (1,2,3)$ since $X$ is general.
After replacing homogeneous coordinates, we can write $\Gamma = (x_0 = x_2 = y = z = 0)$.
A general member $S \in |\mcI_{\Gamma,X} (A)|$ is cut out on $X$ by $\alpha_0 x_0 + \alpha_2 x_2$.
Note that $\bar{f}_4 := f_4 (0,x_1,0,x_3,0) = 0$ and $\bar{g}_4 := g_4 (0,x_1,0,x_3,0) = 0$ since $X$ contains $\Gamma$.
It follows that $\bar{f}_4 = \alpha y^2 + y f_2 (x_1,x_3)$ and $\bar{g}_4 = \beta y^2 + y g_2 (x_1,x_3)$.
By counting dimensions, we may assume the following for a general $X$.
\begin{itemize}
\item At least one of $f_2$ and $g_2$ is not divisible by $x_1$.
\item Neither $f_2$ nor $g_2$ is divisible by $x_1^2$.
\item If $\alpha = 0$, then neither $f_2$ nor $g_2$ is divisible by $x_1$. 
\end{itemize}
This is because a member $X$ of the family No.~4 containing a given $\Gamma$ and does not satisfy one of the above conditions form a $12 = 10 + 2$ codimensional subfamily.
The surface $S$ is smooth along $\Gamma$ if and only if the matirix
\[
M :=
\begin{pmatrix}
\prt f_4/\prt x_0 & \prt f_4/\prt x_2 & \prt f_4/\prt y & 0 \\
 \prt g_4/\prt x_0 & \prt g_4/\prt x_2 & \prt g_4/\prt y & x_1 \\
 \alpha_0 & \alpha_2 & 0 & 0
 \end{pmatrix}
 \]
 has rank $3$ at every point of $\Gamma$.
 By our choice of $X$, $M$ has rank $3$ at $\msp_3$, where $\{\msp_3\} = \Gamma \cap (x_0 = x_1 = 0)$. 
 Quasismoothness of $X$ implies that the matrix
 \[
\begin{pmatrix}
\prt f_4/\prt x_0 & \prt f_4/\prt x_2 & \prt f_4/\prt y & 0 \\
 \prt g_4/\prt x_0 & \prt g_4/\prt x_2 & \prt g_4/\prt y & x_1
 \end{pmatrix}
 \]
has rank $2$ at every point of $\Gamma$.
It follows that $\rank M = 3$ at every point of $\Gamma \setminus  \{\msp_3\}$ for a general choice of $S$. 
Therefore, (1) holds for a general $X$.
The scheme $C = \Pi_{\Gamma} \cap X$ is isomorphic to the weighted complete intersection
\[
(\bar{f}_4 = z x_1 + \bar{g}_4 = 0) \subset \mbP (1,1,2,3),
\]
Note that $\bar{f}_4$ and $\bar{g}_4$ are divisible by $y$ since $X$ contains $\Gamma$.
We write $\bar{f}_4 = \alpha y^2 + y f_2$ and $\bar{g}_4 = \beta y^2 + y g_2$, where $\alpha, \beta \in \mbC$ and $f_2, g_2 \in \mbC [x_1,x_3]$.
Set $\Gamma_1 = (x_0 = x_1 = x_2 = y = 0)$.
We see that $C = C_1 \cup C_2 \subset \mbP (1,1,2,3)$, where $C_1 = (y = z x_1 + \beta y^2 + y g_2 = 0) = \Gamma \cup \Gamma_1$ and $C_2 = (\alpha y + f_2 = z x_1 + \beta y^2 + y g_2 = 0)$.

Assume that $\alpha = 0$.
We can write $f_2 = f_1 f'_1$, where $\deg f_1 = \deg f'_1 = 1$ and hence $C_1 = \Gamma_2 \cup \Gamma_3$, where $\Gamma_2 = (f_1 = z x_1 + \beta y^2 + y g_2 = 0)$ and $\Gamma_3 = (f'_1 = z x_1 + \beta y^2 + y g_2 = 0)$.
By the generality of $X$, $f_2 = f_1 f'_1$ is not divisible by $x_1$.
Since $X$ does not contain $\msp_4$, we have $\beta \ne 0$.
It follows that $\Gamma_2$ and $\Gamma_3$ are irreducible.
Note that the case where $\Gamma_2 = \Gamma_3$ can happen.
It follows that $C_{\operatorname{red}} = \Gamma \cup \Gamma_1 \cup \Gamma_2 \cup \Gamma_3$ or $C_{\operatorname{red}} = \Gamma \cup \Gamma_1 \cup \Gamma_2$.

In the following we assume that $\alpha \ne 0$.
We may assume that $\alpha = 1$ by re-scaling $y$.
We may further assume that $\beta = 0$ by replacing $F_2$ with $F_2 - \beta F_1$ and then replacing $x_1$.

We treat the case where neither $f_2$ nor $g_2$ is divisible by $x_1$.
We see that $C_2 = (x_0 = x_2 = y + f_2 = z x_1 - f_2 g_2 = 0)$ is irreducible and we set $\Gamma_2 = C_2$.

We treat the case where $f_2$ is divisible by $x_1$ but $g_2$ is not divisible by $x_1$.
We write $f_2 = x_1 f_1$.
We have $C_1 = \Gamma_1 \cup \Gamma_2$, where $\Gamma_2 = (x_0 = x_2 = y + x_1 f_1 = z - f_1 g_2 = 0)$.

In any of the above cases, we have $(\Gamma \cdot \Gamma_i)_S \ge 1$ since $\Gamma \cap \Gamma_i \ne \emptyset$, and $\deg \Gamma_i \le 1$.
It follows that $(\Gamma \cdot \Gamma_i)_S \ge \deg \Gamma_i > 0$ and $(\Gamma_j \cdot \Gamma_i) \ge 0$ for $j \ne i$.
Thus (3) is verified by setting $\Delta_i = \Gamma_i$.
We also observe that (2) is true.

Finally, we treat the case where $g_2$ is divisible by $x_1$ but $f_2$ is not divisible by $x_1$.
We write $g_2 = x_1 g_1$.
Then, $C_1 = \Gamma_2 \cup \Gamma_3$, where $\Gamma_2 = (x_0 = x_2 = y + f_2 = z-f_2 g_1=0)$ and $\Gamma_3 = (x_0 = x_1 = x_2 = y+f_2=0)$.
We see that $(\Gamma \cdot \Gamma_1)_S = 1$, $(\Gamma \cdot \Gamma_2 )_S = 2$ and $(\Gamma \cdot \Gamma_3)_S = 0$ since $\Gamma$ and $\Gamma_1$ (resp.\ $\Gamma$ and $\Gamma_2$) meet at one point (resp.\ two points) and $\Gamma \cap \Gamma_3 = \emptyset$.
Moreover, we have $(\Gamma_2 \cdot \Gamma_3)_S = 1$ since $\Gamma_2$ and $\Gamma_3$ meet at one point at which $S$ is smooth.
Let $S'$ be a member of $|\mcI_{\Gamma,X} (A)|$ such that $S' \ne S$.
Then we have $A|_S \sim S'|_S = \Gamma + \Gamma_1 + \Gamma_2 + \Gamma_3$.
We have $1 = (A \cdot \Gamma_2) = (\Gamma \cdot \Gamma_2)_S + (\Gamma_1 \cdot \Gamma_2)_S + (\Gamma_2^2)_S + (\Gamma_3 \cdot \Gamma_2)_S = 3 + (\Gamma_2^2)_S$, which implies $(\Gamma_2^2)_S = -2$.
For $i = 1,2$, we can verify (3) by setting $\Delta_i = \Gamma_i$.
We shall construct $\Delta_3$.
We set $\Delta_3 = \Gamma_3 + \varepsilon \Gamma_2$ for some $1/3 < \varepsilon < 1/2$.
We have $\deg \Delta_3 = 1/3 + \varepsilon$, $(\Gamma \cdot \Delta_3)_S = 2 \varepsilon$, $(\Gamma_1 \cdot \Delta_3)_S = (\Gamma_1 \cdot \Gamma_3)_S \ge 0$ and $(\Gamma_2 \cdot \Delta_3)_S = 1 - 2 \varepsilon$.
Therefore $\deg \Delta_3 > (\Gamma \cdot \Delta_3)_S > 0$ and $(\Gamma_j \cdot \Delta_3) \ge 0$ for $j = 1,2$.
This completes the proof.
\end{proof}

\begin{Lem} \label{exclcurveno4_1}
Let $X$ be a member of the family $\mcG_4$ and $\Gamma$ an irreducible curve of degree $1$ on $X$.
Then $\Gamma$ is not a maximal center.
\end{Lem}

\begin{proof}
It is enough to consider $\Gamma$ which is contained in the smooth locus of $X$.
We assume that $\Gamma$ is a maximal center.
Then there is a movable linear system $\mcH \subset |n A|$ and an exceptional divisor $E$ over $X$ with center $\Gamma$ such that $\mult_E \mcH > n$.
Let $S$ and $S'$ be general members of the pencil $|\mcI_{\Gamma,X} (A)|$.
We can write
\[
\begin{split}
S' |_S &= C = \Gamma + \sum\nolimits_{i=1}^r  c_i \Gamma_i, \\
A|_S &\sim \frac{1}{n} \mcH |_S = \frac{1}{n} \mcL + \gamma \Gamma + \sum\nolimits_{i=1}^l \gamma_i \Gamma_i,
\end{split}
\]
where $c_i > 0$, $\Gamma_i$ is an irreducible curve and $\mcL$ is a movable linear system on $S$.
We have $\gamma \ge (1/n) \mult_E \mcH > 1$.
Possibly re-ordering, we may assume that $\gamma_l/c_l \ge \cdots \ge \gamma_1/c_1$.

Assume that $\gamma \le \gamma_1/c_1$.
In this case, we set
\[
D := (A - \gamma S')|_S \sim \frac{1}{n} L + \sum\nolimits_{i=1}^l (\gamma_i - c_i \gamma) \Gamma_i,
\]
where $L \in \mcL$.
We have
\[
(1-\gamma) (A \cdot \Gamma) = (D \cdot \Gamma)_S
= \frac{1}{n} (L \cdot \Gamma) + \sum\nolimits_{i=1}^l (\gamma_i - c_i \gamma) (\Gamma_i \cdot \Gamma) \ge 0
\]
since $L$ is nef and $\Gamma_i \ne \Gamma$.
It follows that $\gamma \le 1$.
This is a contradiction and we must have $\gamma > \gamma_1/c_1$.
In this case, we set
\[
D_1 := (c_1 A-\gamma_1 S')|_S \sim \frac{c_1}{n} L + (c_1 \gamma - \gamma_1)\Gamma + \sum\nolimits_{i \ge 2} (c_1 \gamma_i - c_i \gamma_1)\Gamma_i,
\]
where $L \in \mcL$.
Let $\Delta_1$ be the effective 1-cycle on $S$ given in Lemma \ref{lemexclcdeg1no4}.
We have
\[
(c_1-\gamma_1) (A \cdot \Delta_1) = (D_1 \cdot \Delta_1)_S 
\ge (c_1 \gamma - \gamma_1) (\Gamma \cdot \Delta_1)_S
\]
since $L$ is nef and $(\Gamma_i \cdot \Delta_1)_S \ge 0$ for $i \ge 2$.
This, together with $(\Gamma \cdot \Delta_1)_S \ge (A \cdot \Delta_1) = \deg \Delta_1$ implies $\gamma \le 1$.
This is a contradiction and $\Gamma$ cannot be a maximal center.
\end{proof}

We summarize the conclusion of this section.

\begin{Thm} \label{exclcurve}
Let $X$ be a member of the family $\mcG_i$ with $i \in I^*$.
Then no curve on $X$ is a maximal center.
\end{Thm}

\begin{proof}
This follows from Corollary \ref{exclcurve0} and Lemmas \ref{exclcurve1}, \ref{exclcurveno4_2} and \ref{exclcurveno4_1}.
\end{proof}

\section{Exclusion of nonsingular points} \label{sec:nspt}

The aim of this section is to exclude nonsingular points as a maximal center.

\subsection{Easy cases}

In this subsection, we exclude most of the nonsingular points by applying the following criterion.

\begin{Thm} \label{nsptisol}
Let $X$ be a $\mbQ$-Fano WCI and $\msp$ a nonsingular point of $X$.
If there is a positive integer $l$ such that $l A$ is a $\msp$-isolating class and that $l \le 4/(A^3)$, then $\msp$ is not a maximal center.
\end{Thm}

\begin{proof}
See Proof of (A) in pages 210 and 211 in \cite{CPR}.
\end{proof}

Let $X = X_{d_1,d_2} \subset \mbP (a_0,\dots,a_5)$ be a member of $\mcG_i$ with $i \in I^*$.
Throughout this subsection, we assume $a_0 \le a_1 \le \cdots \le a_5$.
We write $x_0,\dots,x_5$ (or $x,y,\dots$ etc.) for homogeneous coordinates on $\mbP$, and $\msp = (\xi_0: \dots :\xi_5)$ (or $\msp = (\xi :\eta :\cdots)$ etc.) for the coordinates of a point $\msp$.

\begin{Def}[{\cite[Definition 5.6.3]{CPR}}]
Let $\{g_i\}$ be a finite set of homogeneous polynomials in the variables $x_0,\dots,x_5$.
We say that the set $\{g_i\}$ {\it isolates} $\msp$ if $\msp$ is a component of the set
\[
X \cap \bcap_i (g_i = 0).
\]
\end{Def}

\begin{Lem}[{\cite[Lemma 5.6.4]{CPR}}] \label{lemisolnspt}
If a finite set $\{g_i\}$ of homogeneous polynomials isolates $\msp$, then $l A$ isolates $\msp$, where $l = \max \{\deg g_i\}$.
\end{Lem}

\begin{Prop} \label{isoclassnspt}
Let $\msp = (\xi_0 : \cdots : \xi_5)$ be a nonsingular point of $X$ and assume that $\xi_j \ne 0$ for some $0 \le j \le 5$.
Then $m A$ isolates $\msp$, where
\[
m = \max_{0 \le l \le 5, l \ne j} \{ \lcm (a_j,a_l) \}.
\]
If in addition $\msp$ is not contained in $\Exc (\pi_k)$ for some $k \ne j$, then $m_k A$ isolates $\msp$, where
\[
m_k = \max_{0 \le l \le 5, l \ne j, k} \{ \lcm (a_j,a_l) \}.
\]
\end{Prop}

\begin{proof}
For $l \ne j$, we put $a'_l = a_l / \gcd (a_j, a_l)$ and $a'_j = a_j / \gcd (a_j,a_l)$.
The set
\[
\{ g_0,\dots, \hat{g}_j,\dots,g_5\}, \text{ where } g_l = \xi_j^{a'_l} x_l^{a'_j} - \xi_l^{a'_j} x^{a'_l}_j,
\]
obviously isolates $\msp$.
We have $\deg g_l = a_l a'_j = a_j a'_l = \lcm (a_j,a_l)$.
It follows from Lemma \ref{lemisolnspt} that $m A$ isolates $\msp$.
If $\msp \notin \Exc (\pi_k)$ then the set 
\[
\{ g_0,\dots,\hat{g}_j,\dots,\hat{g}_k,\dots,g_5 \}
\] 
isolates the point $\msp$, which completes the proof.
\end{proof}

\begin{Prop} \label{exclsmpts123}
Let $X = X_{d_1,d_2} \subset \mbP (a_0,\dots,a_5)$ be a member of $\mcG_i$ with $i \in I^*$.
If
\[
a_2 a_5 \le 4/(A^3),
\]
then no nonsingular point of $X$ is a maximal center.
\end{Prop}

\begin{proof}
Let $\msp = (\xi_0: \cdots : \xi_5)$ be a nonsingular point of $X$.
We shall show that $l A$ is a $\msp$-isolating class for some $l \le 4/(A^3)$, which completes the proof by Theorem \ref{nsptisol}.
We divide the proof into three cases.

Case 1: $a_4 a_5 \le 4 / (A^3)$ ($31$ families).
In this case, Proposition \ref{isoclassnspt} implies that $l A$ isolates $\msp$ for some $l \le a_4 a_5$.

Case 2: $a_3 a_5 \le 4 / (A^3) < a_4 a_5$ ($11$ families).
In this case, $X$ is a member of $\mcG_i$ with
\[
i \in \{14,24,38,39,43,50,52,56,57,82,83\}.
\]
Assume that $\xi_j = 0$ for $j = 0,1,2,3$.
Then, by checking individually, $\msp$ is a singular point of $X$.
It follows that $\xi_j \ne 0$ for some $0 \le j \le 3$ and applying Proposition \ref{isoclassnspt} we see that $l A$ isolates $\msp$ for some $l \le a_3 a_5$.

Case 3: $a_2 a_5 \le 4 / (A^3) < a_3 a_5$ ($18$ families).
In this case, $X$ is a member of $\mcG_i$ with
\[
i \in \{5,7,13,17,19,23,29,30,34,41,42,49,55,65,69,72,74,77\}.
\]
Assume that $\xi_j = 0$ for $j = 0,1,2$.
Then, by checking individually, $\msp$ is a singular point of $X$ except possibly when $X$ is a member of $\mcG_7$.
If $X$ is not a member of $\mcG_7$, then $\xi_j \ne 0$ for some $0 \le j \le 2$ and applying Proposition \ref{isoclassnspt} we see that $l A$ isolates $\msp$, where $l \le a_2 a_5$.
If $X = X_{4,6} \subset \mbP (1^3,2,3^2)$ is a member of $\mcG_7$, then  $\pi_3 \colon X \to \mbP (1^3,3^2)$ does not contract any curve and Proposition \ref{isoclassnspt} implies that $3 A$, where $3 = 4/(A^3)$, isolates $\msp$.
This completes the proof.
\end{proof}

\begin{Prop} \label{exclsmpts4}
Let $X = X_{d_1,d_2} \subset \mbP (a_0,\dots,a_5)$ be a member of $\mcG_i$ with $i \in I^*$.
If
\[
a_1 a_5 \le 4/(A^3) < a_2 a_5 \text{ and } 4 /(A^3) \le a_2 a_4,
\]
then no nonsingular point of $X$ is a maximal center.
\end{Prop}

\begin{proof}
There are $4$ families: $X$ is a member of $\mcG_i$ with $i \in \{8,10,12,15\}$.
If $\xi_0 \ne 0$ or $\xi_1 \ne 0$, then $l A$ isolates $\msp$ for some $l \le a_1 a_5$ by Proposition \ref{isoclassnspt}.
Thus we may assume that $\xi_0 = \xi_1 = 0$.

Let $X = X_{4,6} \subset \mbP (1^2, 2^3,3)$ be a member of $\mcG_8$.
We see that $\pi_5$ does not contract any curve and $2 A$ isolates $\msp$.
We have $2 < 4 / (A^3) = 4$.

Let $X = X_{5,6} \subset \mbP (1^2,2^2,3^2)$ be a member of $\mcG_{10}$ and $\msp = (0:0:\eta_0:\eta_1:\zeta_0:\zeta_1)$.
We see that at least one of $\eta_0$ and $\eta_1$ is nonzero because otherwise $\msp$ is a singular point.
Similarly, we have that at least one of $\zeta_0$ and $\zeta_1$ is nonzero.
It follows that the set $\{x_0,x_1,\eta_1 y_0 - \eta_0 y_1, \zeta_1 z_0 - \zeta_0 z_1\}$ isolates $\msp$ because otherwise $X$ contains a WCI curve of type $(1,1,2,3)$, which is impossible by ($\mathrm{C}_2$).
Thus, $3 A$ is a $\msp$-isolating class and we have $3 < 4/A^3 = 24/5$.

Let $X = X_{6,6} \subset \mbP (1^2,2^2,3,4)$ be a member of $\mcG_{12}$ and $\msp = (0:0:\eta_0:\eta_1:\zeta:\sigma) \in X$.
The projection $\pi_4$ does not contract any curve and thus $4 A$ isolates $\msp$ by Proposition \ref{isoclassnspt}.
We have $4 < 4 / (A^3) = 16/3$.

Finally, let $X = X_{6,7} \subset \mbP (1^2,2^2,3,5)$ be a member of $\mcG_{15}$ and $\msp = (0:0:\eta_0:\eta_1:\zeta:\sigma) \in X$.
We see that at least one of $\eta_0$ and $\eta_1$ is nonzero because otherwise $\msp$ is the singular point $\msp_5$.
If $\zeta = 0$, then the set $\{x_0,x_1,\eta_1 y_0 - \eta_0 y_1, z\}$ isolates $\msp$ because otherwise $X$ contains a WCI curve of type $(1,1,2,3)$, which is impossible by ($\mathrm{C}_2$).
If $\zeta \ne 0$, then the set $\{x_0,x_1,\eta_1 y_0 - \eta_0 y_1, \eta_0 \zeta s - \sigma y_0 z\}$ isolates $\msp$.
Thus $l A$ is a $\msp$-isolating class for some $l \le 5 < 4/(A^3) = 40/7$.
\end{proof}

In the remainder of this subsection, we treat families satisfying $a_1 a_5 \le 4/(A^3) < a_2 a_5$ and $a_2 a_4 \le 4/(A^3)$.
There are $14$ such families. 

\begin{Lem} \label{lem:spsingpts}
Let $X \subset \mbP (a_0,\dots,a_5)$ be a member of $\mcG_i$ which satisfies
\[
a_1 a_5 \le 4/(A^3) < a_2 a_5 \text{ and } a_2 a_4 < 4/(A^3).
\]
Then the following assertions hold.
\begin{enumerate}
\item The set $(x_0 = x_1 = 0) \cap \Exc (\pi_5)$ consists of singular points of $X$ except possibly when $X$ belongs to $\mcG_i$ with $i \in \{18,21,28,36\}$.
\item The set $(x_0 = x_1 = x_2 = 0)_X$ consists of singular points of $X$ execpt possibly when $X$ belongs to $\mcG_i$ with $i \in \{21,36\}$.
\end{enumerate}
\end{Lem}

\begin{proof}
The families satisfying the assumption are the families $\mcG_i$, where
\[
i \in \{16,18,21,22,25,26,28,32,33,36,40,48,54,61\}.
\]
Set $\Sigma := (x_0 = x_1 = 0) \cap \Exc (\pi_5)$ and $\Sigma' := (x_0 = x_1 = x_2 = 0)_X$.

We shall first prove (1).
Let $X = X_{6,7} \subset \mbP (1^2,2,3^2,4)$ be a member of $\mcG_{16}$.
Defining polynomials can be written as $F_1 = s y + f_6$ and $F_2 = s z_0 + g_7$, where $f_6$ and $g_7$ do not involve $s$.
By ($\mathrm{C}_2$), we have $z_1^2 \in f_6$ because otherwise $X$ contains the WCI curve $(x_0 = x_1 = y = z_0 = 0)$ of type $(1,1,2,3)$.
It follows that $\Sigma = (x_0 = x_1 = y = z_0 = 0)_X = \{\msp_5\}$.
The case for members of $\mcG_{26}$ can be proved similarly.

Let $X = X_{7,8} \subset \mbP (1^2,2,3,4,5)$ be a member of $\mcG_{22}$.
Quasismothness of $X$ in particular implies that the monomial $s^2$ appears in the defining polynomial of degree $8$ with non-zero coefficient.
Then we have $\Sigma = (x_0 = x_1 = y = z = 0)_X = \{\msp_5\}$.
The cases for members of $\mcG_{i}$, where $i \in \{33, 48\}$, can be verified similarly.

Let $X = X_{8,9} \subset \mbP (1^2,2,3,4,7)$ be a member of $\mcG_{25}$.
Quasismoothness of $X$ implies that the monomials $s^2$ and $z^3$ appear in the defining polynomials of degree $8$ and $9$, respectively, with non-zero coefficient.
Thus we have $\Sigma = (x_0 = x_1 = y = 0)_X = \{\msp_5\}$.
The cases for members of $\mcG_{i}$, where $i \in \{32,40,54,61\}$, can be verified similarly.

We shall prove (2).
If $X$ is a member of $\mcG_{16}$ (resp.\ $\mcG_{18}$, resp.\ $\mcG_{26}$), then $\Sigma'$ consists of singular points if and only if $X$ does not  contain a WCI curve of type $(1,1,2,3)$ (resp.\ $(1,1,2,3)$, resp.\ $(1,1,3,4)$).
Thus, the assertion follows from ($\mathrm{C}_2$).
If $X$ is a member of $\mcG_i$ with $i \in \{22, 28, 33, 48\}$, then $\Sigma'$ consists of singular points if and only if $z s$ does not appear in defining polynomials for $X$.
The assertion follows from $(\mathrm{C}_1)$.
Finally, if $X$ is a member of $\mcG_{i}$ with $i \in \{25,32,40,54,61\}$, then $\Sigma = \Sigma'$ and it consists of singular points by (1).
\end{proof}

\begin{Prop} \label{exclsmpts5}
Let $X = X_{d_1,d_2} \subset \mbP (a_0,\dots,a_5)$ be a member of $\mcG_i$ with $i \in I^*$ which satisfies
\[
a_1 a_5 \le 4/(A^3) < a_2 a_5 \text{ and } a_2 a_4 < 4/(A^3).
\]
Then no nonsingular point of $X$ is a maximal center.
\end{Prop}

\begin{proof}
Let $\msp = (\xi_0:\cdots:\xi_5)$ be a nonsingular point of $X$.
If $\xi_i \neq 0$ for $i = 0$ or $1$, then $l A$ isolates $\msp$ for some $l \le a_1 a_5$ by Proposition \ref{isoclassnspt}.
Thus we may assume that $\xi_0 = \xi_1 = 0$.

We shall show that $\msp$ is not a maximal center if $\xi_2 = 0$.
By (2) of Lemma \ref{lem:spsingpts}, it is enough to consider a member $X$  of $\mcG_{i}$ with $i \in \{21,36\}$.
Assume that $X = X_{6,9} \subset \mbP (1^2,2,3,4,5)$ is a member of $\mcG_{21}$ and let $F_1$, $F_2$ be defining polynomials of degree $6$, $9$, respectively.
We may assume that $z^2 \notin F_1$ because otherwise $\msp$ is a singular point.
In this case, we can write $F_2 (0,0,0,z,s,t) = t s + z^3$ after rescaling $z$ and $(x_0 = x_1 = y = 0)_X = (x_0 = x_1 = y = t s + z^3 = 0)$.
If we write $\msp = (0:0:0:\zeta:\sigma:\tau)$, then none of $\zeta$, $\sigma$ and $\tau$ is zero because $\msp$ is a nonsingular point of $X$.
We see that $\{x_0,x_1,y,\sigma^2 t z - \zeta \tau s^2\}$ isolates $\msp$, which shows that $8 A$ isolates $\msp$.
Thus, $\msp$ is not a maximal center since $8 < 4/(A^3) = 80/9$.
Let $X = X_{8,12} \subset \mbP (1^2,3,4,5,7)$ be a member of $\mcG_{36}$ and let $F_1$, $F_2$ be defining polynomials of degree $8$, $12$, respectively.
We may assume that $z^2 \notin F_1$ because otherwise $\msp$ is a singular point.
By quasismoothness of $X$, we have $z^3 \in F_2$ and thus we can write $F_2 (0,0,0,z,s,t) = t s + z^3$.
It follows that $\{x_0,x_1,y, \zeta \sigma^2 t^2 - \tau^2 z s^2\}$ isolates $\msp$ and $14 A$ is a $\msp$-isolating class.
Thus, $\msp$ is not a maximal center since $14 < 4/(A^3) = 35/2$.

It remains to consider nonsingular points with $\xi_2 \ne 0$. 
If $X$ is not a member of $\mcG_{18}$, $\mcG_{21}$, $\mcG_{28}$ or $\mcG_{36}$, then $\msp$ is not contained in $\Exc (\pi_5)$ by (1) of Lemma \ref{lem:spsingpts}.
In this case, $a_2 a_4 A$ isolates $\msp$ by Proposition \ref{isoclassnspt}.
Let $X$ be a member of $\mcG_i$ with $i \in \{18,21,28,36\}$.
It remains to exclude nonsingular point $\msp$ contained in $(x_0 = x_1 = 0) \cap \Exc (\pi_5)$.
By quasismoothness of $X$, we may assume that $(x_0 = x_1 = 0) \cap \Exc (\pi_5) = (x_0 = x_1 = x_j = 0)_X$ for some $j$ with $2 \le j \le 4$.
We see that neither $\xi_k$ nor $\xi_l$ is zero, where $\{j,k,l\} = \{2,3,4\}$. Here we need the assumption ($\mathrm{C}_2$) (resp.\ ($\mathrm{C}_3$)) if $i = 18$ (resp.\ $i = 21$). 
See Example \ref{ex:no21nonsingpt} below.
Now we have $a_5 = a_k + a_l$ and the set $\{x_0,x_1,x_j,\xi_k \xi_l x_5 - \xi_5 x_k x_l\}$ isolates $\msp$.
Moreover, we see that $a_1 =1$ and thus $a_5 A = a_1 a_5 A$ isolates $\msp$.
This shows that $\msp$ is not a maximal center.
\end{proof}

\begin{Ex} \label{ex:no21nonsingpt}
Let $X = X_{6,9} \subset \mbP (1^2,2,3,4,5)$ be a member of $\mcG_{21}$ and $\msp = (0:0:\eta:\zeta:\sigma:\tau)$ a nonsingular point contained in $\Sigma := (x_0 = x_1 = 0) \cap \Exc (\pi)$.
We shall see that neither $\eta$ nor $\zeta$ is zero.
Defining polynomials of $X$ can be written as $F_1 = t x_0 + f_6$ and $F_2 = t s + g_9$, where $f_6$ and $g_9$ do not involve $t$.
We have $\Sigma = (x_0 = x_1 = s = 0)_X$ and $\Sigma \cap (y = 0)_X = \{ \msp_5 \}$.
It follows that $\eta \ne 0$.
We see that $\Sigma \cap (z = 0)_X = \{\msp_5\}$ if and only if $y^3 \in f_6$.
If $y^3 \notin f_6$, then the subscheme $(x_0 = x_1 = 0)_X$ is reducible, which is impossible by ($\mathrm{C}_3$).
This shows that $\zeta \ne 0$.
\end{Ex}

\subsection{The remaining cases}

In this subsection, we treat members $X_{4,4} \subset \mbP (1^4,2,3)$ of $\mcG_4$, $X_{4,5} \subset \mbP (1^3,2^2,3)$ of $\mcG_6$, $X_{4,4} \subset \mbP (1^3,2,3,4)$ of $\mcG_9$ and $X_{6,6} \subset \mbP (1^3,2,3,5)$ of $\mcG_{11}$.

\begin{Lem} \label{lem:nsptNo9}
Let $X$ be a member of one of the families $\mcG_4$, $\mcG_9$ and $\mcG_{11}$, and $\msp$ a nonsingular point which is not contained in $\Exc (\pi_5)$.
Then $\msp$ is not a maximal center.
\end{Lem}

\begin{proof}
If $X$ is a member of $\mcG_4$ (resp.\ $\mcG_6$, resp.\ $\mcG_{11}$), then it is easy to see that $A$ (resp.\ $2 A$, resp.\ $3A$) isolates $\msp$, which together with $4/(A^3) = 3/2 > 1$ (resp.\ $4/(A^3) = 12/5 > 2$, resp.\ $4/(A^3) = 10/3 > 3$) shows that $\msp$ is not a mxaimal center.
Assume that $X \in \mcG_9$ and let $\msp = (\xi_0:\xi_1:\xi_2:\eta:\zeta:\sigma)$.
We see that $(x_0 = x_1 = x_2 = 0)_X$ consists of singular points since $y z$ appears in the defining polynomial of degree $5$ by ($\mathrm{C}_1$).
It follows that $\xi_i \ne 0$ for some $i = 0, 1,2$ and $3 A$ isolates $\msp$.
Thus, $\msp$ is not a maximal center since $3 < 4/(A^3) = 16/5$.
\end{proof}

In the following we consider a nonsingular point $\msp \in \Exc (\pi_5)$.
We shall define a linear system $\Lambda_{\msp}$ on $X$ and a curve $\Gamma \subset X$ for a member $X$ of $\mcG_{4}$, $\mcG_6$, $\mcG_9$ and $\mcG_{11}$, and then study singularities of a general member $S \in \Lambda_{\msp}$.
We denote by $\varphi \colon Y \to X$ the blouwp of $X$ at $\msp$ with exceptional divisor $E \cong \mbP^2$ and $\Gamma' \subset Y$ the birational transform of $\Gamma$.

\begin{Lem} \label{lem:S_No9}
Let $X$ be a general member of the family No.~$9$ and $\msp \in \Exc (\pi_5)$ a nonsingular point.
Then the following assertions hold for the linear system $\Lambda_{\msp} := |\mcI^2_{\msp,X} (5 A)|$.
\begin{enumerate}
\item The base locus of $\Lambda_{\msp}$ is $\Gamma \cup \{\msq\}$ set-theoretically, where $\Gamma$ is a WCI curve of type $(1,1,2,3)$ and $\msq$ is the singular point of type $\frac{1}{2} (1,1,1)$.
\item A general member $S$ of $\Lambda_{\msp}$ has a simple double point at $\msp$, a Du Val singular point of type $\frac{1}{4} (1,3)$ at $\msp_5$ and is nonsingular along $\Gamma \setminus \{\msp, \msp_5\}$.
\item We have $(E \cdot \Gamma')_Y = 1$.
\end{enumerate}
\end{Lem}

\begin{proof}
We can write defining polynomials of $X$ as $F_1 = s x_0 + z f_2 + f_5$ and $F_2 = s y + z^2 + z g_3 + g_6$, where $f_j, g_j \in \mbC [x_0,x_1,x_2,y]$.
Note that WCI curves of type $(1,1,2,3)$ contained in $(x_0 = y = 0)$ form a $2$-dimensional family and members of the family No.~$9$ with the above defining polynomials containing such a WCI curve form a $2$-codimensional family.
After re-choosing homogeneous coordinates, we can assume that $\msp = (0:0:1:0:0:0)$.
By counting dimensions, we can assume that $x_2^2 \in f_2$ for a general $X$.
The base locus $\Bs |\mcI^2_{\msp,X} (5 A)|$ is $(x_0 = x_1 = y^2 x_2 = y z = 0)_X = Z_1 \cup Z_2$ set-theoretically, where $Z_1 := (x_0 = x_1 = y = 0)_X$ and $Z_2 := (x_0 = x_1 = x_2 = z = 0)_X$.
We have $Z_2 = \{ \msq,\msp_5\}$  and $Z_1$ coincides with $\Gamma := (x_0 = x_1 = y = z = 0)$ set-theoretically since $x_2^2 \in f_2$.
Since $S$ is cut out on $X$ by a section of the form $\alpha_0 s x_0 + \alpha_1 s x_1 + \cdots$, the singularity of $S$ at $\msp_5$ is of type $\frac{1}{4} (1,3)$.
A straightforward computation shows that $S$ has a simple double point at $\msp$ and that $S$ is nonsingular along $\Gamma \setminus \{\msp,\msp_5\}$.
We see that $E$ and $\Gamma'$ meet transversally at one point, so that we have $(E \cdot \Gamma')_Y = 1$.
\end{proof}

We treat a member $X$ of the family No.~$4$ and No.~$11$.

\begin{Lem} \label{lemexclnspexc1}
Let $X$ be a general member of the family No.~$4$ $($resp.\ No.~$11$$)$ and $\msp$ a nonsingular point contained in $\Exc (\pi)$.
Then there is a linear system $\Lambda_{\msp} \subset |2 A|$ with the following properties.
\begin{enumerate}
\item The base locus of $\Lambda_{\msp}$ is set-theoretically a WCI curve $\Gamma$ of type $(1,1,1,2)$ $($resp.\ $(1,1,2,3)$$)$.
\item A general member $S$ has a simple double point at $\msp$, a quotient singular point of type $\frac{1}{3} (1,1)$ $($resp.\ $\frac{1}{5} (1,3)$$)$ at $\msp_5$, and is smooth along $\Gamma \setminus \{\msp,\msp_5\}$.
\item We have $(E \cdot \Gamma')_Y = 1$.
\end{enumerate} 
\end{Lem} 

\begin{proof}
Let $X$ be a general member of the family No.~$4$ with defining polynomials of the form $F_1 = z x_0 + f_4$ and $F_2 = z x_1 + g_4$ for some $f_4, g_4$ not involving $z$, so that it is quasismooth and the matrix
\[
\begin{pmatrix}
\prt f_4/\prt x_2 & \prt f_4/\prt x_3 & \prt f_4/\prt y \\
\prt g_4/\prt x_2 & \prt g_4/\prt x_3 & \prt g_4/\prt y
\end{pmatrix}
\]
has rank $2$ at the generic point of any component of $\Exc (\pi_5) = (x_0 = x_1 = 0)_X$.
This is possible since the condition is equivalent to the condition that the above matrix has rank $2$ at each point of $(x_0 = x_1 = f_4 = g_4 = 0) \subset \mbP (1,1,1,2)$.
After replacing homogeneous coordinates, we assume that $\msp = \msp_3$ and defining polynomials of $X$ can be written as $F_1 = z x_0 + y x_3^2 + h_4$ and $F_2 = z x_0 + x_2 x_3^3 + h'_4$, where $h_4, h'_4 \in (x_0,x_1,x_2,y)^2$ do not involve $z$.
Let $\Lambda_{\msp}$ be the linear system spanned by the sections $\{x_j x_k \mid 0 \le j, k \le 2\}$ and $S \in \Lambda_{\msp}$ a general member.
The base locus of $\Lambda_{\msp}$ is the WCI curve $\Gamma = (x_0 = x_1 = x_2 = y = 0) \subset X$ of type $(1,1,1,2)$.
Let $y + q (x_0,x_1,x_2)$, where $\deg q = 2$, be the section which cut out $S$ on $X$.
It is easy to see that $S$ has a singular point of type $\frac{1}{3} (1,1)$ at $\msp_5$.
Jacobian matrix of the affine cone of $S$ restricts to
\[
\begin{pmatrix}
z & 0 & 0 & 0 & x_3^2 & 0 \\
0 & z  & x_3^2 & 0 & 0 & 0 \\
0 & 0 & 0 & 0 & 1 & 0
\end{pmatrix}
\]
on $\Gamma$.
This implies that $S$ is smooth along $\Gamma \setminus \{\msp, \msp_5\}$.
Let $\Phi$ be the weighted blowup of $\mbP (1^4,2,3)$ at $\msp$ with $\wt (x_0,x_1,x_2,y,z) = (1,1,1,2,1)$, and we define $\varphi \colon Y \to X$ and $\varphi_S \colon T \to S$ to be the birational morphisms induced by $\Phi$.
We can choose $x_0,x_1,z$ as local coordinates of $X$ at $\msp$ and thus $\varphi$ is the blowup of $X$ at $\msp$.
We have $(E \cdot \Gamma') = 1$ since $E$ and $\Gamma'$ meet transversally at one point.
The exceptional divisor of $\varphi_S$ is isomorphic to the weighted complete intersection
\[
\begin{split}
& (z x_0 + y = x_2 = y + q = 0) \subset \mbP (1,1,1,2,1) \\
&\cong (z x_0 - q (x_0,x_1,0) = 0) \subset \mbP^2,
\end{split}
\]
which is clearly a (nondegenerate) conic in $\mbP^2$ for a general choice of $S$.

Let $X$ be a general member of the family No.~$11$.
By a similar argument as above, we may assume that $\msp = \msp_2$ and defining polynomials of $X$ are written as $F_1 = s x_0 + z x_2^3 + h_6$ and $F_2 = s x_1 + y x_2^4 + h'_6$, where $h_6, h'_6 \in (x_0,x_1,y,z)^2$ do not involve $s$, after replacing homogeneous coordinates.
Let $\Lambda_{\msp}$ be the linear system spanned by the sections $x_0^2$, $x_0 x_1$, $x_1^2$ and $y$.
The base locus of $\Lambda_{\msp}$ is the WCI curve $\Gamma = (x_0 = x_1 = y = z = 0) \subset X$ of type $(1,1,2,3)$.
We see that a general member $S \in \Lambda_{\msp}$ has a quotient singular point of type $\frac{1}{5} (1,3)$ at $\msp_5$ and is smooth along $\Gamma \setminus \{\msp, \msp_5\}$.
Let $\Phi$ be the weighted blowup of $\mbP (1^3,2,3,5)$ at $\msp$ with $\wt (x_0,x_1,y,z,s) = (1,1,2,1,1)$, and define $\varphi \colon Y \to X$ and $\varphi_S \colon T \to S$ to be the birational morphisms induced by $\Phi$.
We see that $\varphi$ is the blowup of $X$ at $\msp$ and the exceptional divisor of $\varphi_S$ is a conic in $\mbP^2$.
We also see that $(E \cdot \Gamma') = 1$.
This completes the proof.
\end{proof}

We treat a member $X = X_{4,5} \subset \mbP (1^3,2^2,3)$ of the family No.~$6$.

\begin{Lem} \label{lemexclnsptno6}
Let $X$ be a general member of the family No.~$6$ and $\msp$ a nonsingular point contained in $\Exc (\pi_5)$.
Then there is a linear system $\Lambda_{\msp} \subset |2 A|$ with the following properties.
\begin{enumerate}
\item The base locus of $\Lambda_{\msp}$ is set-theoretically a WCI curve $\Gamma$ of type $(1,1,2,2)$.
\item A general member $S \in \Lambda_{\msp}$ has a simple double point at $\msp$, a quotient singular point of type $\frac{1}{3} (1,1)$ at $\msp_5$ and is smooth along $\Gamma \setminus \{\msp,\msp_5\}$.
\item We have $(E \cdot \Gamma')_Y = 1$.
\end{enumerate}
\end{Lem}

\begin{proof}
We may assume that defining polynomials of $X$ are $F_1 = z x_0 + f_4$ and $F_2 = z y_0 + g_5$, where $f_4$ and $g_5$ do not involve $z$.
Under this setting, let $X$ be a general member so that it is quasismooth and the following sets, considered as a subset of $\mbP (1^3,2^2)$, are empty.
\begin{equation} \tag{i}
\left( x_0 = y_0 = f_4 = g_5 = \frac{\prt f_4}{\prt y_1} = 0 \right).
\end{equation}
\begin{equation} \tag{ii}
\left( x_0 = y_0 = f_4 = g_5 = \frac{\prt f_4}{\prt y_0} \frac{\prt g_4}{\prt y_1} - \frac{\prt f_4}{\prt y_0} \frac{\prt g_5}{\prt y_1} = 0 \right).
\end{equation}
\begin{equation} \tag{iii}
\left( x_0 = y_0 = f_4 = g_5 = \frac{\prt f_4}{\prt x_1}  \frac{\prt g_5}{\prt y_1} - \frac{\prt f_4}{\prt y_1} \frac{\prt g_5}{\prt x_1} = 0 \right).
\end{equation}
Thus, after replacing homogeneous coordinates, we can assume that $\msp = \msp_2$ and defining polynomials of $X$ are written as $F_1 = z x_0 + y_1 x_2^2 + h_4$ and $F_2 = z y_0 + (\alpha_0 x_0 + \alpha_1 x_1) x_2^4 + \beta y_0 x_2^3 + h_5$, where $\alpha_0,\alpha_1, \beta \in \mbC$ with $\alpha_1, \beta \ne 0$ and $h_4, h_5 \in (x_0,x_1,y_0,y_1)^2$.
Let $\Phi$ be the weighted blowup of $\mbP (1^3,2^2,3)$ at $\msp$ with $\wt (x_0,x_1,y_0,y_1,z) = (1,1,2,2,1)$.
The birational morphism $\varphi \colon Y \to X$ induced by $\Phi$ is the blowup of $X$ at $\msp$ since we can choose $x_0,x_1,z$ as local coordinates of $X$ at $\msp$.
We see that sections $y_0$ and $y_1$ vanishes along the exceptional divisor $E$ of $\varphi$ to order at least $2$ and let $\Lambda_{\msp}$ be the linear system spanned by the sections $x_0^2,x_0 x_1, x_1^2, y_0$ and $y_1$.
It is then easy to verify that the base locus of $\Lambda_{\msp}$ is the curve $\Gamma = (x_0 = x_1 = y_0 = y_1 = 0)$, $S$ is smooth along $\Gamma \setminus \{\msp,\msp_5\}$ and the singularity of $S$ at $\msp_5$ is of type $\frac{1}{3} (1,1)$.
It is also easy to see that the birational morphism $\varphi_S \colon T \to S$ induced by $\varphi$ resolves the singularity of $S$ at $\msp$ and its exceptional divisor $F = E|_T$ is a conic in $E \cong \mbP^2$.
We also see that $E$ and $\Gamma'$ meet at one, so that $(E \cdot \Gamma')_Y = 1$.
\end{proof}

We apply the following result to exclude nonsingular points contained in $\Exc (\pi_5)$.

\begin{Thm}[{\cite[Corollary 3.5]{Corti2}}] \label{redsurf}
Let $\msp \in X$ be a smooth threefold germ and $\mcH$ a movable linear system on $X$.
Assume that
\[
K_X + \frac{1}{\mu} \mcH
\]
is not canonical at $\msp$ for some rational number $\mu > 0$ and let $\varphi \colon Y \to X$ be the blowup of $X$ at $\msp$ with the exceptional divisor $E$.
Then either
\begin{enumerate}
\item $m = m_E (\mcH) > 2 \mu$, or
\item there is a line $\Gamma \subset E \cong \mbP^2$ such that
\[
K_Y + \left( \frac{m}{\mu} - 1 \right) E + \frac{1}{\mu} \mcH_Y
\]
is not log canonical at the generic point of $\Gamma$, where $\mcH_Y$ is the birational transform of $\mcH$.
\end{enumerate}
\end{Thm}

\begin{Prop} \label{exclsmptsrem}
Let $X$ be a member of one of the families $\mcG_4$, $\mcG_6$, $\mcG_9$ and $\mcG_{11}$.
Then no nonsigular point is a maximal center.
\end{Prop}

\begin{proof}
By Lemma \ref{lem:nsptNo9}, it is enough to consider a nonsingular point $\msp \in \Exc (\pi_5)$.
Assume that $\msp$ is a maximal center of a movable linear system $\mcH \subset |nA|$.
Let $\Lambda_{\msp}$ be the linear system as in Lemmas \ref{lem:S_No9}, \ref{lemexclnspexc1} and \ref{lemexclnsptno6}, $S \in \Lambda_{\msp}$ a general member and $\Gamma$ the base curve of $\Lambda_{\msp}$.
Let $\varphi \colon Y \to X$ be the blowup of $X$ at $\msp$ with exceptional divisor $E$ so that the exceptional divisor $F = E|_T$ of the induced morphism $\varphi_S \colon T \to S$ is a conic in $E \cong \mbP^2$.
We write $\varphi^* \mcH = \mcH_Y + m E$, where $\mcH_Y$ is a movable linear system on $Y$.
Theorem \ref{redsurf} implies that either $m > 2 n$ or there is a line on $E$ along which $K_Y + (m/n - 1)E + (1/n) \mcH_Y$ is not log canonical.
We write
\[
\varphi^*A|_T
\sim \varphi^* (\frac{1}{n} \mcH)|_T
= \frac{1}{n} \mcL + \beta F + \alpha \Gamma',
\]
where $\Gamma'$ is the birational transform of $\Gamma$ and $\mcL$ is a movable linear system on $T$.
Then we have
\[
\begin{split}
\left( K_Y + T + \left(\frac{m}{n}-1 \right) E + \frac{1}{n} \mcH_Y \right)|_T 
&= K_T + \left(-E + \varphi^* (\frac{1}{n} \mcH) \right)|_T \\
&= K_T + (\beta-1)F + \alpha \Gamma' + \frac{1}{n} \mcL.
\end{split}
\]
It follows from the inequality $\beta \ge m/n$ and the connectedness theorem of Koll\'{a}r and Shokurov \cite[Theorem 17.4]{Kollar+} that either $\beta > 2$ or $K_T + (\beta-1)F + \alpha \Gamma' + (1/n) \mcL$ is not log canonical at two points.
Let $L$ be a $\mbQ$-divisor on $T$ such that $n L \in \mcL$.
We shall derive a contradiction by computing $(L^2)_T$.
We see that $(F^2)_T = -2$ and $(F \cdot \Gamma')_T = (E \cdot \Gamma')_Y = 1$.

Let $X$ be a member of $\mcG_9$.
In this case, $T$ has a unique singular point along $\Gamma'$, which is of type $\frac{1}{4} (1,3)$.
Thus we compute $({\Gamma'}^2)_T = -2 - (K_T \cdot \Gamma')_T + 3/4 = - 9/4$,  where the term $3/4$ comes from the different (see \cite[Chapter 16]{Kollar+} for the different).
Thus we compute
\[
(L^2)_T = \frac{25}{4} - 2 \beta^2 - \frac{9}{4} \alpha^2 - \frac{1}{2} \alpha + 2 \alpha \beta 
= - 2 \left(\beta - \frac{\alpha}{2} \right)^2 - \frac{7}{4} \alpha^2 - \frac{1}{2} \alpha + \frac{25}{4}.
\]
We have $(L^2)_T \ge 0$ since $L$ is nef.
We have $\alpha < 2$ because otherwise we get
\[
(L^2)_T \le - \frac{7}{4} \alpha^2 - \frac{1}{2} \alpha + \frac{25}{4} \le - \frac{7}{4} < 0,
\]
which is a contradiction.
Assume that $\beta > 2$.
Then, since $\alpha < 2$, we get
\[
(L^2)_T < \left( 2 - \frac{\alpha}{2} \right)^2 - \frac{7}{4} \alpha^2 - \frac{1}{2} + \frac{25}{4} 
= - \frac{9}{4} \left( \alpha - \frac{7}{9} \right)^2 - \frac{7}{18} < 0.
\]
This is a contradiction and we get $\beta \le 2$.
From the above argument, $K_T + (\beta-1) E|_T + \alpha \Gamma' + 1/n \mcL$ is not log canonical at two distinct points on $E|_T$.
By Theorem \ref{multineq2}, we have $(L^2)_T > 4 (2-\beta) + 4 (2-\beta)(1-\alpha)$.
We compute
\[
4(2-\beta) + 4 (2-\beta)(1-\alpha) - (L^2)_T
= 2 \left(\beta - \frac{4-\alpha}{2} \right)^2 + \frac{7}{4} (\alpha-1)^2 \ge 0,
\]
which derives a contradiction.
Therefore $\msp$ is not a maximal center.

Let $X$ be a member of $\mcG_4$ (resp.\ $\mcG_{11}$, resp.\ $\mcG_6$).
We can compute in the same way that $({\Gamma'}^2)_T = -2 - 1/3 + 2/3 = -5/3$ (resp.\ $({\Gamma'}^2)_T = -2 - 1/5 + 4/5 = -7/5$, resp.\ $({\Gamma'}^2)_T = -2 - 1/3 + 2/3 = - 5/3$).
This enables us to compute $(L^2)_T$ and we can prove in the same way that $\alpha < 2$, $\beta \le 2$ and finally conclude that
\[
4 (2-\beta)+4(2-\beta)(1-\alpha)-(L^2)_T
= 2 \left(\beta - \frac{4-\alpha}{2}\right)^2 + \frac{7}{6} \left(\alpha - \frac{10}{7}\right)^2 + \frac{2}{7} > 0
\]
if $X \in \mcG_4$,
\[
4 (2-\beta)+4(2-\beta)(1-\alpha)-(L^2)_T
= 2 \left(\beta - \frac{4-\alpha}{2}\right)^2 + \frac{9}{5} (\alpha-1)^2 + \frac{19}{5} > 0
\]
if $X \in \mcG_{11}$, and
\[
4 (2-\beta)+4(2-\beta)(1-\alpha)-(L^2)_T
= 2 \left( \beta - \frac{4-\alpha}{2} \right)^2 + \frac{7}{6} \left( \alpha - \frac{10}{7} \right)^2 + \frac{16}{7} > 0
\]
if $X \in \mcG_6$.
This is a contradiction and $\msp$ is not a maximal center.
\end{proof}

As a conclusion, we have the following result.

\begin{Thm} \label{exclnspt}
Let $X$ be a member of the family $\mcG_i$ with $i \in I^*$.
Then no nonsingular point of $X$ is a maximal center.
\end{Thm}

\begin{proof}
This follows from Propositions \ref{exclsmpts123}, \ref{exclsmpts4}, \ref{exclsmpts5} and \ref{exclsmptsrem}.
\end{proof}

\section{Exclusion of singular points} \label{sec:singpt}

The aim of this section is to exclude the remaining singular points as a maximal center.
Throughout this section, $X$ is a member of $\mcG_i$ with $i \in I^*$ and $\msp$ is a singular point of $X$ with empty third column in the big table.
We denote by $\varphi \colon Y \to X$ the Kawamata blowup of $X$ at $\msp$ with the exceptional divisor $E$ and recall that $B = -K_Y$.

\subsection{Exclusion by the test class method} \label{sec:singpt1}

In this subsection, let $\msp$ be a singular point of $X$ marked a divisor of the form $b B + e E$ and a set of monomials in the second column of the big table.
There are 61 instances.
We define $S$ to be the birational transform on $Y$ of a general member of the linear system $|A|$.
In this subsection, we assume the following additional assumption.

\begin{Assumption} \label{assumpsingpt}
If $X$ is a member of $\mcG_{71}$, then the monomial $z^2 s$ appears in the defining polynomial of degree $16$ with non-zero coefficient.
\end{Assumption} 

\begin{Rem}
Assumption \ref{assumpsingpt} is an extra assumption on generality beyond quasismoothness and it is necessary to exclude the singular point $\msp_2$ of $X \in \mcG_{71}$ as a maximal center by the test class method.
We will prove in Section \ref{sec:singpt3} that $\msp_2$ is not a maximal center even if the above assumption is not satisfied.
\end{Rem}
  
We explain the set of polynomials $\Lambda$ in the second column of the big table, which is determined up to a suitable choice of homogeneous coordinates.
If the subscript $\sharp$ is not given in $\Lambda$, then $\Lambda$ is described after fixing any choice of homogeneous coordinates.
If the subscript $\sharp$ is given, then $\Lambda$ is described in homogeneous coordinates $x_0,\dots,x_5$ of $\mbP (a_0,\dots,a_5)$ such that $\msp$ is a vertex $\msp_j$ with $j$ minimum possible. 
For example, if $X$ is a member of $\mcG_{14}$ (resp.\ $\mcG_{33}$) and $\msp$ is of type $\frac{1}{2} (1,1,1)$ (resp.\ $\frac{1}{3} (1,1,2)$), then we choose coordinates so that $\msp = \msp_1$ (resp.\ $\msp = \msp_2$).
A coordinate with a prime means that it is a tangent coordinate of $X$ at $\msp$.
For example, if $X = X_{10,12} \subset \mbP (1,2,3,4,5,8)$ is a member of $\mcG_{43}$ and $\msp$ is of type $\frac{1}{2} (1,1,1)$, then we can choose homogeneous coordinates so that $\msp = \msp_1$ and defining polynomials can be written as $F_1 = y (u + f_8) + f_{10}$ and $F_2 = y^4 (s + g_4) + y^3 g_6 + y^2 g_8 + y g_{10} + g_{12}$, where $f_j$ and $g_j$ do not involve $y$.
In this case, we define $s' = s + g_4$ and $u' = u + f_8$.

\begin{Lem} \label{lem:isolmonom}
The set of monomials in the third column of the big table isolates the point $\msp$.
\end{Lem}

\begin{proof}
Let $\Lambda$ be the set of polynomials in the second column of the big table.
Then, by checking each instance individually, the set of common zeros of polynomials in $\Lambda$ is a finite set of closed points including the point $\msp$ except possibly when $X \in \mcG_{71}$ and $\msp$ is of type $\frac{1}{4} (1,1,3)$.
If $X \in \mcG_{71}$ and $\msp$ is of type $\frac{1}{4} (1,1,4)$, then the set of common zero can contain a curve but it does not pass through $\msp$.
It follows that $\Lambda$ isolates $\msp$ in this case as well (see Examples \ref{ex:singpt1}, \ref{ex:singpt1No37} and \ref{ex:singpt1No71} below). 
\end{proof}


\begin{Prop} \label{exclsingpt1}
Let $X$ be a member of the family $\mcG_i$ with $i \in I^*$ and $\msp$ a singular point of $X$ marked a divisor $b B + e E$ and a set of monomials in the second column of the big table.
Then the divisor $M := b B + e E$ is a test class with $(M \cdot B^2) \le 0$.
In particular, $\msp$ is not a maximal center.
\end{Prop}

\begin{proof}
Let $\Lambda = \{f_1,\dots,f_l\}$ be the set in the second column which isolates $\msp$.
Suppose that $f_j$ vanishes along $E$ to order $c_j/r$, where $r$ is the index of $\msp$, and set $b_j := \deg f_j$.
We see that the birational transform of $(f_j = 0)_X$ on $Y$ defines the divisor $M_i \sim b_j B + e_j E$, where $e_j := (b_j - c_j)/r$ is a non-negative integer. 
Let $k$ be an index such that 
\[
e_k/b_k = \max_{1 \le j \le l} \{e_j/b_j\}.
\]
Then $M_k$ is the divisor $M = b B + e E$ in the middle column of the table and we observe that the inequality $b > r e$ holds for each instance.
We claim that $M = M_k$ is nef.
Let $m$ be a sufficiently divisible positive integer.
Then, around an open subset containing $E$, the base locus of the complete linear system $|m M|$ is contained in $E$ since $|m M|$ contains a suitable multiple of a divisor $M_j + k_j E$ for some $k_j \ge 0$, and each component of the intersection of $M_j$'s is either contained in $E$ or disjoint from $E$ by Lemma \ref{lem:isolmonom}.
Let $C$ be an irreducible curve on $Y$.
If $C$ is disjoint from $E$, then $(M \cdot C) = b (\varphi^*A \cdot C) > 0$.
If $C$ meets $E$ but is not contained in $E$, then $(M \cdot C) \ge 0$ since $C$ is not contained in the base locus of $|m M|$.
Suppose that $C$ is contained in $E$.
Note that we have $(B \cdot C) = - (1/r) (E \cdot C)$ and $(E \cdot C) < 0$ since $C$ is contracted by $\varphi$ and $-E$ is $\varphi$-ample.
It follows that
\[
(M \cdot C) = (bB+eE \cdot C) = - \frac{b - r e}{r} (E \cdot C) > 0
\]
since $b - e r > 0$.
This shows that $M$ is nef.
The computation of $(M \cdot B^2)$ is straightforward and we omit the proof of the inequality $(M \cdot B^2) \le 0$. 
By Corollary \ref{criexclmstc}, $\msp$ is not a maximal center.
\end{proof}

\begin{Ex} \label{ex:singpt1}
We explain how to apply $(\mathrm{C}_1)$ or $(\mathrm{C}_2)$ in order to obtain Lemma \ref{lem:isolmonom} for a member of $\mcG_i$ with $i \ne 37,71$ (see Examples \ref{ex:singpt1No37} and \ref{ex:singpt1No71} below for $\mcG_{37}$ and $\mcG_{71}$).

For a member $X$ of $\mcG_i$ with $i \in \{27,33,35,39,46,47,52,53,70\}$, the condition $(\mathrm{C}_1$) or $(\mathrm{C}_2)$ ensures that there is a tangent coordinate of $X$ at $\msp$ with the largest possible degree.
For example, let $X = X_{9,10} \subset \mbP (1^2,3,4,5,6)$ be a member of $\mcG_{33}$ and $\msp$ a point of type $\frac{1}{3} (1,1,2)$ with defining polynomials $F_1, F_2$ of degree respectively $9,10$.
We may assume that $\msp = \msp_2$.
Quasismoothness of $X$ implies that $y t \in F_1$ but does not imply $y^2 z \in F_2$.
If $y^2 z \notin F_2$, then $X$ contains the WCI curve $(x_0 = x_1 = s = t = 0)$ of type $(1,1,5,6)$.
Thus $(\mathrm{C}_2)$ ensures that $y^2 z \in F_2$ and defining polynomials can be written as $F_1 = y t + f_9$, $F_2 = y^2 z + y g_7 + g_{10}$, where $f_9,g_7,g_{10}$ do not involve $y$.
It is then easy to see that $\{x_0,x_1,z,t\}$ isolates $\msp$.

For a member $X$ of $\mcG_i$, $(\mathrm{C}_2)$ ensures that the indicated monomials indeed isolates $\msp$.
For example, let $X = X_{8,9} \subset \mbP (1,2,3^2,4,5)$ be a member of $\mcG_{27}$ and $\msp$ a point of type $\frac{1}{2} (1,1,1)$.
We may assume that $\msp = \msp_1$.
The existence of the tangent coordinate $y^2 t$ follows by applying $(\mathrm{C}_2)$ as above and defining polynomials of $X$ can be written as $F_1 = y^2 s + y f_6 + f_8$, $F_2 = y^2 t + y g_7 + g_9$, where $f_j, g_j$ do not involve $y$.
We have $(x = z = s = t_0 = 0)_X = (x = s = t = y f_6  = g_9 = 0) = Z_1 \cup Z_2$, where $Z_1 = (x = y = s = t = g_9 = 0)$ and $Z_2 = (x = s = t = f_6 = g_9 = 0)$.
The set $Z_1$ consists of singular points of type $\frac{1}{3} (1,1,2)$.
The set $Z_2$ contains a point other than $\msp$ if and only if it contains a WCI curve of type $(1,3,4,5)$, which is impossible by $(\mathrm{C}_2)$.
Thus $\{x,s,t\}$ isolates $\msp$.
\end{Ex}

\begin{Ex} \label{ex:singpt1No37}
Let $X = X_{8,12} \subset \mbP (1,2,3,4,5,6)$ be a member of $\mcG_{37}$ with defining polynomials $F_1,F_2$ of degree respectively $8,12$ and $\msp$ a point of type $\frac{1}{2} (1,1,1)$.
By $(\mathrm{C}_1)$, we have $s^2 \in F_1$.
Then we have $(x = y = z = t = 0)_X = \emptyset$ and thus $y$ does not vanish at $\msp$.
After replacing $s,t$, we may assume that $\msp = \msp_1$.
In this case the set $\{x,z,s,u\}$ isolates $\msp$.
\end{Ex}

\begin{Ex} \label{ex:singpt1No71}
Let $X = X_{14,16} \subset \mbP (1,4,5,6,7,8)$ be a member of $\mcG_{71}$ and $\msp$ a point of type $\frac{1}{4} (1,1,3)$.
We may assume that $\msp = \msp_1$ and defining polynomials can be written as $F_1 = y^2 s + y f_{10} + f_{14}$ and $F_2 = y^2 u + y g_{12} + g_{16}$, where $f_j$ and $g_j$ do not involve $y$.
We have $F_1 (0,y,z,s,0,0) = y^2 s + \alpha y z^2$ and $F_2 (0,y,z,s,0,0) = \beta y s^2 + \gamma z^2 s$, where $\alpha, \beta, \gamma \in \mbC$.
By quasismoothness of $X$, we see that $\alpha \ne 0$ and $\beta \ne 0$.
It follows that $(x = t = u = 0)_X = \{\msp\} \cup Z_1 \cup Z_2$, where $Z_1 = (x = t = u = y = \gamma z^2 s = 0)$ and $Z_2 = (x = t = u = y s + \alpha z^2 = \beta y s + \gamma z^2 = 0)$.
If $\alpha = \gamma / \beta$ then $Z_2$ is a WCI curve of type $(1,7,8,10)$ but this does not happen by ($\mathrm{C}_2$).
Thus, we have $\alpha \ne \gamma / \beta$ and $Z_1 = \{\msp, \msp_3\}$.
Although $Z_1$ can be a curve (this is exactly the case where $\gamma = 0$), it does not contain $\msp$.
This shows that $\{x,t,u\}$ isolates $\msp$.
%
\end{Ex}

\subsection{Exclusion by determining the Kleiman-Mori cone} \label{sec:singpt2}

In this subsection, let $\msp \in X$ be a singular point of index $r$ marked $T \in |m B|$ in the second column of the big table.
We note that $(B^3) \le 0$ in this case.
As in the previous subsection, we define $S \in |B|$ to be the birational transform on $Y$ of a general member of the linear system $|A|$.

\begin{Lem}[{\cite[Corollary 5.4.6]{CPR}}] \label{detcone}
Assume that $(B^3) \le 0$ and there is a divisor $T \in |b B + eE|$, where $b > 0$, $b/r \ge e \ge 0$, with the following properties.
\begin{enumerate}
\item The scheme theoretic intersection $\Gamma := S \cap T$ is an irreducible and reduced curve.
\item $(T \cdot \Gamma) \le 0$.
\end{enumerate}
Then we have $\bNE (Y) = R + Q$, where $R$ is the extremal ray spanned by a curve contracted by $\varphi$ and $Q = \mbR_+ [\Gamma]$.
\end{Lem}

\begin{Thm}[{\cite[Theorem 5.4.8]{CPR}}] \label{thmexclsingpt}
Assume that $(B^3) \le 0$ and there exists a divisor $T$ as in Lemma \ref{detcone}.
Then $B^2 \notin \Int \bNE (Y)$.
In particular, $\msp$ is not a maximal center.
\end{Thm}

\begin{proof}
This is the consequence of Lemma \ref{detcone} and we refer the reader to \cite{CPR} for the proof.
\end{proof}

We shall construct a divisor $T$ as in Lemma \ref{detcone}.
Let $a_0,a_1,\dots,a_5$ be the weights of the ambient weighted projective space and assume that $a_0 \le a_1 \le \cdots \le a_5$.
We define $T$ to be the birational transform of a general member of the linear system $|\mcI_{\msp,X} (a_1 A)|$.
Note that $|\mcI_{\msp,X} (a_1 A)|$ is a pencil and we have $|\mcI_{\msp,X} (a_1 A)| = |a_1 A|$ if $i \notin \{20,34,59\}$.
By the construction, we have $T \in |a_1 B|$ and hence if $\Gamma$ is irreducible and reduced, then $(T \cdot \Gamma) = (S \cdot T^2) = a_1^2 (B^3) \le 0$ is automatically satisfied.

\begin{Lem} \label{TKM}
The scheme theoretic intersection $\Gamma = S \cap T$ of $S$ and $T$ is an irreducible and reduced curve, and $(T \cdot \Gamma) \le 0$.
\end{Lem}

\begin{proof}
%
It is enough to show that $\Gamma$ is irreducible and reduced.
After replacing homogeneous coordinates, we may assume that $\Gamma$ is the birational transform of $(x_0 = x_1 = 0)_X$.
After further replacing coordinates, defining polynomials of $X$ are written as $F_1 = G_1 + H_1$ and $F_2 = G_2 + H_2$, where $G_1, G_2$ are the ones listed in Tables \ref{table:gammaqsm} or \ref{table:gammanonqsm} and $H_1, H_2 \in (x_0,x_1)$.
The conditions on the coefficients $\alpha, \beta, \gamma \in \mbC$ in the tables are the consequence of $(\mathrm{C}_1)$ and $(\mathrm{C}_2)$.
This is the crucial part of the proof which cannot be done systematically.
We need a case-by-case verification for each instance.
We explain all the necessary technique in Examples \ref{ex:eqgamma} below.

It follows that $\Gamma$ is the birational transform of the WCI curve $C := (G_1 = G_2= 0)$ in $(x_0 = x_1 = 0) \cong \mbP (a_2,\cdots,a_5)$.
The assertion that $\Gamma$ is irreducible and reduced follows from the same assertion for $C$.
If $X$ belongs to a family listed in Table \ref{table:gammaqsm}, then $C$ is quasismooth and thus $C$ is irreducible and reduced.
If $X$ belongs to a family listed in Table \ref{table:gammanonqsm}, then case-by-case verifications show that $C$ is irreducible and reduced (see Example \ref{ex:irrgamma} below).
\end{proof}

\begin{Prop} \label{exclsingpt2}
Let $X$ be a member of $\mcG_i$ with $i \in I^*$ and $\msp$ a singular point of $X$ marked $T \in |m B|$ in the second column of the big table for some $m$.
Then $\msp$ is not a maximal center.
\end{Prop}

\begin{proof}
This follows from Lemma \ref{TKM} and Theorem \ref{thmexclsingpt}.
\end{proof}

\begin{table}[bt]
\begin{center}
\caption{Quasismooth $\Gamma$}
\label{table:gammaqsm}
\begin{tabular}{c|c|c}
No. & Equations & Conditions \\
\hline
19 & $s_0 y + s_1 y + z^2 +\alpha y^3 = s_1 s_0 + \beta y^4 = 0$ & $\beta \ne 0 \ (\mathrm{C}_2)$ \\
22 & $t y + \alpha s z = t z + s^2 + \beta z^2 y + y^4 = 0$ & $\alpha \ne 0 \ (\mathrm{C}_1)$, $\alpha^2 + \beta^2 \ne 0 \ (\mathrm{C}_2)$ \\
42 & $s_0 y + s_1 y + z^2 = s_1 s_0 + y^3 = 0$ & \\
48 & $t y + \alpha s z = t z + s^2 + y^3 = 0$ & $\alpha \ne 0 \ (\mathrm{C}_1)$ \\
55 & $t y + z^2 = t z + s^2 + z y^2 = 0$ & \\
57 & $u z + \alpha t s = u s + t^2 + \beta s z^3 = 0$ & $\alpha \ne 0 \ (\mathrm{C}_1)$, $\beta \ne 0 \ (\mathrm{C}_2)$ \\
63 & $u z + t s = u t + \alpha t z^2 + s^3 = 0$ & $\alpha \ne 0 \ (\mathrm{C}_2)$ \\
73 & $u z + s^3 = u s + t^2 + z^4 = 0$ & \\
74 & $u z + s^2 = u s + t^2 + z^6 = 0$ & \\
76 & $t z + s^2 = u^2 + t^2 s + z^4 = 0$ & \\
77 & $t y + z^2 = t z + s^2 + y^3 = 0$ & 
\end{tabular}
\end{center} 
\quad \\[7mm]

\begin{center}
\caption{Non-quasismooth $\Gamma$}
\label{table:gammanonqsm}
\begin{tabular}{c|c|c}
No. & Equations & Conditions \\
\hline
25 & $s^2 + \alpha z^2 y + y^4 = t y + z^3 = 0$ & $\alpha \ne 0 \ (\mathrm{C}_1)$ \\
29 & $t y + z^2 = t z + s^2 + \alpha z^2 y + \beta z y^3 + \gamma y^5 = 0$ & \\
31 & $\alpha t z + s_0 s_1 = t^2 + \beta z^2 s_0 + \gamma z^2 s_1 = 0$ & $\alpha \ne 0 \ (\dagger t z)$, $\beta, \gamma \ne 0 \ (\mathrm{C}_2)$ \\
32 & $\alpha s y^2 + z^3 + \beta z y^3 = t y + s^2 = 0$ & $\alpha \ne 0 \ (\mathrm{C}_1)$ \\
34 & $\alpha t y_1 + \beta s y_1^2 + z^3 + \gamma z y_1^3 = t z + s^2 = 0$ & $\alpha \ne 0 \ (\mathrm{C}_2)$ \\
37 & $t z + \alpha s^2 = u^2 + \beta s^3 + z^4 = 0$ & $\alpha \ne 0 \ (\mathrm{C}_1)$, $\beta \ne 0 \ (\mathrm{C}_2)$ \\
39 & $u z + t^2 = u s + \alpha t z^2 + \beta s^2 z = 0$ & $\alpha \ne 0 \ (\mathrm{C}_1)$ \\
51 & $t z + s^2 = u^2 + \alpha s^2 z = 0$ & $\alpha \ne 0 \ (\mathrm{C}_3)$ \\
59 & $t z + s^2 = t^2 + s y_1^2 + \alpha z^2 y_1 = 0$ & $\alpha \ne 0 \ (\mathrm{C}_2)$ \\
61 & $s^2 + y^3 = t y + z^3 = 0$ & \\
65 & $t^2 + \alpha s z^3 = u z + s^3 = 0$ & $\alpha \ne 0 \ (\dagger z^3 s)$
\end{tabular}
\end{center}
\end{table}

\begin{Ex} \label{ex:eqgamma}
We explain how to obtain equations in Tables \ref{table:gammaqsm} and \ref{table:gammanonqsm} by examples.

Let $X = X_{16,18} \subset \mbP (1^2,6,8,9,10)$ be a member of $\mcG_{77}$.
This is the simplest example.
$\Gamma$ is the birational transform of $C := (x_0 = x_1 = 0)_X$ and defining polynomials can be written as $F_1 = \alpha t y + \beta z^2 + H_1$ and $F_2 = \gamma t z + \delta s^2 + \varepsilon y^3 + H_2$, where $\alpha, \beta, \dots, \varepsilon \in \mbC$ and $H_1, H_2 \in (x_0,x_1)$.
Quasismoothness of $X$ implies that all the coefficients $\alpha, \beta, \dots, \varepsilon$ are non-zero.
After re-scaling $y,z,s,t$, we may assume that $\alpha,\beta,\dots,\varepsilon$ are $1$ and thus we get the equations for $C = (x_0 = x_1 = 0)_X$.

In many cases conditions $(\mathrm{C}_1)$ and $(\mathrm{C}_2)$ are required to obtain the desired equations.
The applications of $(\mathrm{C}_1)$ is immediate: the indicated monomial appears in the defining polynomial of $X$. 
We explain how to apply $(\mathrm{C}_2)$.
Let $X = X_{6,8} \subset \mbP (1^2,2,3,4^2)$ be a member of $\mcG_{19}$.
Then $\Gamma$ is the birational transform of $C = (x_0 = x_1 = 0)_X$.
Defining polynomials of $X$ can be written as $F_1 = \alpha s_0 y + \beta s_1 y + \gamma z^2 + \delta y^3 + H_1$ and $F_2 = f_8 (y,s_0,s_1) + \varepsilon z^2 y + H_2$.
By quasismoothness of $X$, we may assume that $f_8 (y,s_0,s_1) = s_0 s_1$ after replacing $s_0$ and $s_1$, and then we see that $\alpha, \beta, \gamma, \delta$ are non-zero.
Thus, rescaling coordinates, we may assume that $F_1 = s_0 y + s_1 y + z^2 + y^3 + H_1$ and $F_2 = s_0 s_1 + \varepsilon z^2 y + H_2$.
If $\varepsilon = 0$, then $C$ is clearly reducible, so we need to show that $\varepsilon \ne 0$ which does not follow from quasismoothness.
If $\varepsilon = 0$, then $X$ contains the WCI curve $(x_0 = x_1 = s_0 = F_1 = 0)$ of type $(1,1,4,6)$.
This is impossible by $(\mathrm{C})_2$ and thus $\varepsilon \ne 0$.
\end{Ex}

\begin{Ex} \label{ex:irrgamma}
We explain that $\Gamma$ is irreducible and reduced by examples.
$\Gamma$ is the birational transform of the WCI curve defined by the equations in Tables \ref{table:gammaqsm} and \ref{table:gammanonqsm}.
It is enough to show that $C$ is irreducible and reduced.
It is clear that $C$ is reduced since $C$ contains a nonsingular point.
Moreover, $C$ is quasismooth if $X$ is a member listed in Table \ref{table:gammaqsm}.
This can be checked by straightforward computations.
Thus $C$ is irreducible in this case.
We treat families listed in Table \ref{table:gammanonqsm}.

Let $X = X_{8,9} \subset \mbP (1^2,2,3,4,7)$ be a member of $\mcG_{25}$.
Then defining polynomials of $X$ are written as $F_1 = s^2 + \alpha z^2 y + y^4 + H_1$ and $F_2 = t y + z^3 + H_2$ for some $H_1,H_2 \in (x_0,x_1)$.
Since $X$ is quasismooth, we may assume that $H_1 = t x_0 + (\text{terms not involving $t$})$.
The projection $\pi \colon X \ratmap \mbP (1^2,2,3,4)$ from $\msp_5$ is a birational map and it contracts curves in $\Exc (\pi) := (x_0 = y = 0)_X$.
We see that $C \cap \Exc (\pi) = \{\msp_5\}$ and it is mapped to $\pi (C)$ defined by $s^2 + \alpha z^2 y + y^4 = 0$ in $(x_0 = x_1=0) \cong \mbP (2,3,4)$.
Now it is easy to see that $\pi (C)$ is irreducible since $\alpha \ne 0$, hence $C$ is irreducible.
The proof is similar for a member of $\mcG_i$ with $i \in \{32,61,65\}$.

Let $X = X_{8,10} \subset \mbP (1^2,2,4,5,6)$ be a member of $\mcG_{29}$.
Defining polynomials of $X$ can be written as $F_1 = t y + z^2 + H_1$ and $F_2 = t z + s^2 + f_{10} (z,y) + H_2$, where $f = \alpha z^2 y + \beta z y^3 + \gamma y^5$ and $H_1, H_2 \in (x_0,x_1)$.
Replacing $y$ and $z$ if necessary, we may assume that $H_1$ and $H_2$ do not involve $t$.
The projection $X \ratmap \mbP (2,4,5,6)$ is a birational map contracting curves in $\Exc (\pi) := (y = z = 0)_X$.
We see that $\Exc (\pi) \cap C = \{\msp_5\}$ and $C$ is mapped to $\pi (C)$ defined by $y (s^2 + f_{10}) - z^3$ in $\mbP (2,4,5)$.
We see that $\pi (C)$ is irreducible and thus $C$ is irreducible.
The proof is similar for a member of $\mcG_{34}$.

Let $X = X_{8,10} \subset \mbP (1,2,3,4^2,5)$ be a member of $\mcG_{31}$.
Then $C$ is defined in $\mbP (3,4,4,5)$ by $\alpha t z + s_0 s_1 = t^2 + \beta s_0 z + \gamma s_1 z = 0$.
Since $(z = 0)_C$ consists of finitely many points, it is enough to show that $C \cap (z \ne 0)$ is irreducible.
We see that $C \cap (z \ne 0)$ is a cyclic quotient of the curve $C_z$ defined in $\mbA^3$ by $\alpha t + s_0 s_1 = t^2 + \beta s_0 + \gamma s_1 = 0$.
By eliminating $t$, $C_z$ is isomorphic to $(s_0^2 s_1^2 + \alpha \beta s_0 + \alpha \gamma s_1 = 0) \subset \mbA^2$.
It follows from $\alpha,\beta,\gamma \ne 0$ that $C_z$ is irreducible and so is $C$.
The proofs for a member of $\mcG_{37}$ and $\mcG_{59}$ (resp.\ $\mcG_{51}$) can be done similarly by restricting $C$ to the open subset $(z \ne 0)$ (resp.\ $s \ne 0$).
\end{Ex}

We make explicit the condition $(\mathrm{C}_3)$ for $\mcG_{20}$, $\mcG_{21}$ and $\mcG_{51}$ in the following examples.

\begin{Ex} \label{ex:C3no20}
Let $X = X_{6,8} \subset \mbP (1,2^2,3^2,4)$ be a member of $\mcG_{20}$.
After replacing coordinates, we assume that $\msp = \msp_2$, $\Gamma$ is the birational transform of $C = (x = y_0 = 0)_X$ and defining polynomials can be written as $F_1 = \alpha s y_1 + z_0 z_1 + H_1$ and $s^2 + \beta s y_1^2 + \gamma z_0^2 y_1 + \delta z_0 z_1 y_1 + \varepsilon y_1 z_1^2 + H_2$, where $H_1,H_2 \in (x,y_0)$.
We shall show that $C$ is irreducible and reduced if and only if $\alpha$, $\gamma$, $\varepsilon$ are nonzero.
If $\alpha = 0$, then $C$ is clearly reducible and if $\gamma = 0$ (resp.\ $\varepsilon = 0$), then $C$ contains the WCI curve $(z_1 = s = 0)$ (resp.\ $(z_0 = s = 0)$) as a component. 
This proves the only if part.
The if part follows by restricting $C$ to the open subset $(z_1 \ne 0)$ as in Example \ref{ex:irrgamma}.
This explains the condition $(\mathrm{C}_3)$ for $\mcG_{20}$.
\end{Ex}

\begin{Ex} \label{ex:C3no21}
Let $X = X_{6,9} \subset \mbP (1^2,2,3,4,5)$ be a member of $\mcG_{21}$.
Then $\Gamma$ is the birational transform of $C = (x_0 = x_1 = 0)_X$ and defining polynomials of $X$ can be written as $F_1 = s y + \alpha z^2 + H_1$, $F_2 = t s + \beta t y^2 + \gamma z^3 + \delta z y^3 + H_2$, where $H_1, H_2 \in (x_0,x_1)$.
We shall show that $C$ is irreducible if and only if $\alpha \ne 0$, $\beta \ne 0$ and $(\gamma,\delta) \ne (0,0)$.
It is clear that $X$ is reducible if $\alpha = 0$ or $(\gamma, \delta) = (0,0)$.
If $\beta = 0$, then $C$ contains the curve $(x_0 = x_1 = z = s = 0)$ as a component and thus $C$ is not irreducible.
The if part follows by considering the projection from $\msp_5$ as in Example \ref{ex:irrgamma}.
This explains $(\mathrm{C}_3)$ for $\mcG_{21}$.
\end{Ex}

\begin{Ex} \label{ex:C3no51}
Let $X = X_{10,14} \subset \mbP (1,2,4,5,6,7)$ be a member of $\mcG_{51}$.
Then $\Gamma$ is the birational transform of $C = (x = y = 0)$ and defining polynomials of $X$ can be written as $F_1 = \alpha t z + s^2 + H_1$, $F_2 = u^2 + \beta t z^2 + H_2$, where $H_1, H_2 \in (x,y)$.
We shall show that $C$ is irreducible and reduced if and only if $\alpha \ne 0$ and $\beta \ne 0$.
If $\alpha = 0$ or $\beta = 0$, then $C$ is not reduced.
The if part follows by restricting $C$ to the open subset $(s \ne 0)$ as in Example \ref{ex:irrgamma}.
This explains $(\mathrm{C}_3)$ for $\mcG_{51}$.
\end{Ex}

\subsection{Exclusion of singular points by the method of bad link} \label{sec:singpt3}

Let $\msp$ be a singular point of $X$ and $\varphi \colon Y \to X$ the Kawamata blowup of $X$ at $\msp$ with the exceptional divisor $E$.

\begin{Def} \label{defspcurve}
Assume that there are Weil divisors $S$ and $T$ on $Y$ with the following properties.
\begin{enumerate}
\item $S \in |m B|$ and $T \in |b B - e E|$ for some integers $m > 0$, $b > 0$ and $e \ge 0$.
\item The reduced scheme $\Gamma_{\reduced}$ of the scheme theoretic intersection $\Gamma := S \cap T$ consists of finitely many irreducible curves $\Gamma_1, \dots, \Gamma_l$ and each two of them are numerically proportional to each other.
\item $(B \cdot \Gamma_1) \le 0$.
\end{enumerate}
In this case, we call $\Gamma$ a {\it special curve} with respect to the point $\msp$.
\end{Def}

\begin{Lem} \label{nespcurve}
Assume that there is a special curve $\Gamma$ with respect to $\msp$.
Then $\bNE (Y) = R + Q$, where $R$ is the ray spanned by a curve contracted by $\varphi$ and $Q$ is the ray spanned by an irreducible component of $\Gamma_{\reduced}$.
\end{Lem}

\begin{proof}
This is a generalization of \cite[Corollary 5.4.6]{CPR}.
Let $S \in |m B|$, $T \in |b B - e E|$ be Weil divisors on $Y$ such that $\Gamma = S \cap T$ and $\Gamma_1, \dots,\Gamma_l$ be the irreducible components of $\Gamma_{\reduced}$ as in Definition \ref{defspcurve}.
Note that $(B \cdot \Gamma_j) \le 0$ since $\Gamma_j$ is numerically proportional to $\Gamma_1$ and $(B \cdot \Gamma_1) \le 0$.
Note also that $(E \cdot \Gamma_j) = r ((\varphi^*A \cdot \Gamma_j) - (B \cdot \Gamma_j)) \ge r (\varphi^* A \cdot \Gamma_j) > 0$.
Set $\alpha = - (B \cdot \Gamma_j)/(E \cdot \Gamma_j) \ge 0$ and $M := B + \alpha E$.
Since $\Gamma_j$ is not contracted by $\varphi$, we have $0 < (\varphi^*A \cdot \Gamma_j) = (B \cdot \Gamma_j) + (1/r) (E \cdot  \Gamma_j)$ and $1/r > - (B \cdot \Gamma_j)/(E \cdot \Gamma_j) = \alpha$.
We shall show that $M$ is nef, which will completes the proof since $(M \cdot \Gamma_j) = 0$.
Assume that $M$ is not nef.
Then there is an irreducible curve $C$ on $Y$ such that $(M \cdot C) < 0$.
We claim that $(E \cdot C) \ge 0$.
Assume to the contrary that $(E \cdot C) < 0$.
Then $C$ is contained in $E$ and we have $[C] \in R$.
Hence $(\varphi^*A \cdot C) = 0$ and $(E \cdot C) < 0$.
It follows that $(M \cdot C) = (B \cdot C) + \alpha (E \cdot C) = (-1/r + \alpha) (E \cdot C) > 0$ and this is a contradiction.
Thus $(E \cdot C) \ge 0$ and we have 
\[
(T \cdot C) = b (B \cdot C) - e (E \cdot C) = b (M \cdot C) - (e + b \alpha) (E \cdot C)  < 0
\]
and 
\[
(S \cdot C) = m (B \cdot C) = m (M \cdot C) - m \alpha (E \cdot C) < 0.
\]
This shows that $C$ is contained in $S \cap T = \Gamma$ and $C$ coincides with one of the irreducible components of $\Gamma_{\reduced}$.
This is a contradiction since $(M \cdot \Gamma_j) = 0$ and $(M \cdot C) < 0$.
Therefore $M$ is nef and this completes the proof.
\end{proof}

\begin{Prop}[{\cite[Theorem 5.5.1]{CPR}}] \label{propbl}
If there is a special curve $\Gamma$ with respect to $\msp$, then $\msp$ is not a maximal center.
\end{Prop}

\begin{proof}
This follows from the proof of \cite[Theorem 5.5.1]{CPR}.
To illustrate the idea, which is called the method of bad link, we give a rough sketch of the proof.
Assume that $\msp$ is a maximal center.
Then there is a Sarkisov link starting from $\varphi \colon Y \to X$.
By Lemma \ref{nespcurve}, the first contraction $\psi \colon Y \to Z$ is the contraction of the extremal ray $Q$ and we have $(K_Y \cdot Q) = - (B \cdot \Gamma_1) \ge 0$ since $\Gamma$ is a special curve with respect to $\msp$.
We only consider the case where $(K_Y \cdot Q) > 0$ which is always the case in this paper.
Then $\psi$ is an inverse flipping contraction contracting irreducible components of $\Gamma_{\reduced}$ and the inverse flip $\psi \colon Y' \to Z$ does exist.
It follows from the flip formalism that another extremal contraction $\varphi' \colon Y' \to X'$ is a contraction of a $K_{Y'}$-non-positive extremal ray and we see that it is a divisorial contraction which contracts the birational transform $T'$ of $T$ on $Y'$.
It follows that $X'$ has singularities worse than terminal and this contradicts to the existence of a Sarkisov link.
\end{proof}

We can exclude some singular points as a maximal center using the method.

\begin{Prop} \label{exclsingpt3}
Let $X$ be a member of $\mcG_{69}$ and $\msp$ the singular point of type $\frac{1}{5} (1,2,3)$. 
Then there exists a special curve $\Gamma$ with respect to $\msp$.
In particular, $\msp$ is not a maximal center.
\end{Prop}

\begin{proof}
It is enough to show the existence of a special curve with respect to $\msp$ by Proposition \ref{propbl}.
Let $X = X_{14,16} \subset \mbP (1^2,5,7,8,9)$ be a member of $\mcG_{69}$ and $\msp = \msp_2$ the point of type $\frac{1}{5} (1,2,3)$.
We can write defining polynomials of $X$ as $F_1 = y t + f_{14}$ and $F_2 = y^3 x_0 + y^2 g_6 + y g_{11} + g_{16}$, where $f_j$ and $g_j$ do not involve $y$.
We can choose $x_1$, $z$ and $s$ as local coordinates of $X$ at $\msp$ and they vanish along $E$ to order $1/5$, $2/5$ and $3/5$, respectively.
From the equations $F_1 = 0$ and $F_2 = 0$, we see that sections $t$ and $x_0$ vanish along $E$ to order at least $6/5$ and $4/5$, respectively.
Let $S$ and $T$ be the birational transforms of the hypersurfaces $(x_0 = 0)_X$ and $(x_1 = 0)_X$ on $Y$, respectively.
The scheme-theoretic intersection $\Gamma = S \cap T$ is an irreducible and reduced curve isomorphic to $(x_0 = x_1 = y t + z^2 = z t + s^2 =0)$.
We compute 
\[
(B \cdot \Gamma) = (B \cdot S \cdot T) = (\varphi^*A - \frac{1}{5} E)^2 \cdot (\varphi^*A - \frac{6}{5} E) = \frac{4}{45} - \frac{1}{5} < 0.
\]
This shows that $\Gamma$ is a special curve.
\end{proof}

In Section \ref{sec:QI}, we construct a quadratic involution of $X \in \mcG_i$ with $i \in \{30,40\}$ (resp.\ $i =64$) centered at the point of type $\frac{1}{3} (1,1,2)$ (resp.\ $\frac{1}{5} (1,2,3)$) under Assumption \ref{assmpQI}.
In Section \ref{sec:singpt1}, the point of type $\frac{1}{5} (1,2,3)$ of $X \in \mcG_{71}$ is excluded as a maximal center by the test class method under Assumption \ref{assumpsingpt}.
We can exclude them by the method of bad link when those assumptions fail as explained in the following examples.
These phenomena are observed in \cite[Remark 4.7]{CPR}.

\begin{Ex} \label{ex:30sf}
Let $X = X_{8,10} \subset \mbP (1^2,3,4,5^2)$ be a member of $\mcG_{30}$ and $\msp = \msp_2$ the point of type $\frac{1}{3} (1,1,2)$.
Assume that $y^2 z$ does not appear in the defining polynomial of degree $10$.
We can write defining polynomials as $F_1 = y (s_0 + s_1) + f_8$ and $F_2 = y^3 x_1 + y^2 g_4 + y g_7 + g_{10}$, where $f_j$ and $g_j$ do not involve $y$. 
We can further assume that $g_{10} (0,0,0,s_0,s_1) = s_0 s_1$.
We see that $x_0$, $x_1$, $z$, $s_0$ and $s_1$ vanish along $E$ to order $1/3$, $4/3$, $1/3$, $2/3$ and $2/3$, respectively.
Let $S$ and $T$ be the birational transforms of $(x_0 = 0)_X$ and $(x_1 = 0)_X$, respectively.
We see that $(x_0 = x_1 = 0)_X$ is the union of two curves $(x_0 = x_1 = s_0 = y s_1 + z^2 = 0)$ and $(x_0 = x_1 = s_1 = y s_0 + z^2 = 0)$.
It follows that $S \cap T$ consists of their birational transforms $\Gamma_0$ and $\Gamma_1$.
Moreover, it is easy to see that $\Gamma_0$ is numerically equivalent to $\Gamma_1$.
We compute
\[
(B \cdot \Gamma_j) =
\frac{1}{2} (B \cdot S \cdot T) = \frac{1}{2} (A - \frac{1}{3} E)^2 \cdot (A - \frac{4}{3} E) = \frac{2}{15} - \frac{1}{3} < 0,
\]
which shows that $\Gamma$ is a special curve.
\end{Ex}

\begin{Ex} \label{ex:40sf}
Let $X = X_{10,12} \subset \mbP (1^2,3,4,5,9)$ be a member of $\mcG_{40}$ and $\msp$ the point of type $\frac{1}{3} (1,1,2)$.
We can assume that $\msp = \msp_2$.
Assume that $y^2 z$ does not appear in the defining polynomial of degree $10$.
We can write defining polynomials as $F_1 = y^3 x_0 + y^2 f_4 + y f_7 + f_{10}$ and $F_2 = y t + g_{12}$, where $f_j$ and $g_j$ do not involve $y$.
By the assumption, $f_4$ does not involve $z$.
We can choose $x_0$, $z$ and $s$ as local coordinates at $\msp$ and explicit computation shows that $x_0$, $x_1$, $z$, $s$ and $t$ vanish along $E$ to order at least $1/3$, $4/3$, $1/3$, $2/3$ and $3/3$, respectively.
Let $S$ and $T$ be the birational transforms of the surfaces $(x_0 = 0)_X$ and $(x_1 = 0)_X$, respectively, and set $\Gamma := S \cap T$.
The reduced scheme $\Gamma_{\reduced}$ is an irreducible curve isomorphic to $(x_0 = x_1 = s = y t + z^3 = 0)$.
We compute $(B \cdot \Gamma_{\reduced})$ by considering embedded weighted blowup.
Let $\psi \colon V \to \mbP := \mbP (1,1,3,4,5,9)$ be the weighted blowup of $\mbP$ at $\msp$ with $\wt (x_0,x_1,z,s,t) = \frac{1}{3} (1,4,1,2,3)$.
We can identify $Y$ with the birational transform of $X$ via $\psi$.
Let $A_{\mbP}$ be a Weil divisor on $\mbP$ such that $\mcO_{\mbP} (A_{\mbP}) \cong \mcO_{\mbP} (1)$ and $B_V = \psi^* A_{\mbP} - (1/3) E_V$, where $E_V$ is the exceptional divisor of $\psi$, so that we have $A_{\mbP} |_X = A$, $B_V |_Y = B$ and $E_V |_Y = E$.
By a slight abuse of notation, we drop subscripts $\mbP$ and $V$ from $A_{\mbP}$, $B_V$ and $E_V$.
We see that sections $x_0$, $x_1$, $s$ and $y t + z^3$ on $\mbP$ lift to sections of $\psi^*A - 1/3 E$, $\psi^*A - 4/3 E$, $5 \psi^*A - 2/3 E$ and $12 \psi^*A - 3/3 E$, respectively.
Thus $(B \cdot \Gamma_{\reduced})$ can be computed as
\[
\begin{split}
& (\psi^*A - \frac{1}{3} E \cdot \psi^*A - \frac{1}{3} E \cdot \psi^*A - \frac{4}{3} E \cdot 5 \psi^*A - \frac{2}{3} E \cdot 12 \psi^* A - \frac{3}{3} E)_V \\
&= 60 (A^5)_{\mbP} - \frac{8}{3^4} (E^5)_V = \frac{1}{9} - \frac{8}{3^4} \times \frac{3^4}{1 \times 4 \times 1 \times 2 \times 3} = - \frac{2}{9}.
\end{split}
\] 
This shows that $\Gamma$ is a special curve.

If $X \in \mcG_i$ with $i \in \{51,64, 65\}$ then a similar argument shows that $\Gamma = S \cap T$ is a special curve, where $S$ and $T$ are the birational transform of $(x = 0)_X$ and $(y = 0)_X$.
\end{Ex}

\begin{Ex} \label{ex:71sf}
Let $X = X_{14,16}\subset \mbP (1,4,5,6,7,8)$ be a member of $\mcG_{71}$ and $\msp = \msp_2$ the point of type $\frac{1}{5} (1,2,3)$.
Assume that $z^2 s$ does not appear in the defining polynomial of degree $16$.
We can write defining polynomials as $F_1 = z^2 y + z f_9 + f_{14}$ and $F_2 = z^3 x + z^2 g_6 + z g_{11} + g_{16}$, where $f_j$ and $g_j$ do not involve $z$.
By the assumption, $g_6$ does not involve $s$.
We see that $x$, $y$, $s$, $t$ and $u$ vanish along $E$ to order $6/5$, $4/5$, $1/5$, $2/5$ and $3/5$, respectively.
Let $S$ and $T$ be the birational transforms on $Y$ of $(y = 0)_X$ and $(x = 0)_X$, respectively, and set $\Gamma := S \cap T$.
We see that $\Gamma_{\reduced}$ is an irreducible curve isomorphic to $(x = y = t = u = 0) \subset \mbP (1,4,5,6,7,8)$. 
An explicit computation involving embedded weighted blowup shows that $(B \cdot \Gamma_{\reduced}) = - 1/6 < 0$.
\end{Ex}

As a conclusion of this section, we have the following result.

\begin{Thm} \label{exclsingpts}
Let $X$ be a member of the family $\mcG_i$ with $i \in I^*$.
Then no singular points of $X$ with empty third column in the big table is a maximal center.
\end{Thm}

\begin{proof}
This follows from Propositions \ref{exclsingpt1}, \ref{exclsingpt2} and \ref{exclsingpt3}.
\end{proof}

\section{The big table} \label{sec:table}

In this section, we list 85 families of anticanonically embedded $\mbQ$-Fano $3$-fold weighted complete intersections of codimension $2$.
This list is taken from \cite[16.7]{IF}.

In each family, the anticanonical volume $(A^3)$ and one of Cases 1--5 are given, where the latter indicates how to exclude nonsingular points as a maximal center (see Section \ref{sec:nspt}).
If Case 1, 2 or 3 (resp.\ 4, resp.\ 5) is marked, then the proof is given in Proposition \ref{exclsmpts123} (resp.\ \ref{exclsmpts4}, resp.\ \ref{exclsmpts5}).
The first column indicates the number and type of singular points of $X$.
The second column indicates how to exclude the singular points as a maximal center (see Section \ref{sec:singpt}).
If a divisor of the form $b B + e E$ and a set of monomials (resp.\ a member of a linear system $T \in |m B|$) are indicated in the second column, then the corresponding singular points are excluded by the test class method (resp.\ by determining the Kleiman-Mori cone).
The singular point of type $\frac{1}{5} (1,2,3)$ in the family No.~69 is marked $T = B - E$ and this point is excluded by the method of bad link.
If the second column is empty, then there is a Sarkisov link centered at the corresponding singular point.
The mark Q.I. (resp.\ E.I.) means that the Sarkisov link is a quadratic (resp.\ elliptic) involution (see Section \ref{sec:birinv}), and the mark of the form $X'_d \subset \mbP (b_0,\dots,b_4)$ (resp.\ $\operatorname{dP}_k$) means that the target of the Sarkisov link is a $\mbQ$-Fano weighted hypersurface (resp.\ del Pezzo fibration of degree $d$) (see Section \ref{sec:Slink}). 
The forth column indicates generality conditions which are necessary to exclude the corresponding singular point as a maximal center or to construct a Sarkisov link centered at the point.
Finally, we explain the daggered monomial in the forth column.
It appears in the families No.~$40$, $64$, $65$ and $71$.
If the daggered monomial appears in one of the defining polynomials of $X$ with non-zero coefficient, then the corresponding singular point is excluded as a maximal center (or there is a quadratic involution centered at the point), otherwise the point is excluded as a maximal center by the method of bad link (see Section \ref{sec:singpt3}).

\newlength{\myheight}
\setlength{\myheight}{0.55cm}

\begin{center}
\begin{flushleft}
No. 1: $X_{2,3} \subset \mbP (1,1,1,1,1,1)$; $(A^3)
 = 6$, birationally rigid.
\end{flushleft} \nopagebreak
\end{center}

\begin{center}
\begin{flushleft}
No. 2: $X_{3,3} \subset \mbP (1,1,1,1,1,2)$; $(A^3)
 = 9/2$.
\end{flushleft} \nopagebreak
\end{center}

\begin{center}
\begin{flushleft}
No. 3: $X_{3,4} \subset \mbP (1,1,1,1,2,2)$; $(A^3)
 = 3$, birationally birigid.
\end{flushleft} \nopagebreak
\end{center}

\begin{center}
\begin{flushleft}
No. 4: $X_{4,4} \subset \mbP (1,1,1,1,2,3)$; $(A^3)
 = 8/3$.
\end{flushleft} \nopagebreak
\begin{tabular}{|p{103pt}|p{112pt}|p{103pt}|>{\centering\arraybackslash}p{26pt}|}
\hline
\parbox[c][\myheight][c]{0cm}{} $\msp_5 = \frac{1}{3} (1,1,2)$ &  & $\dP_2$ & \\
\hline
\end{tabular}
\end{center}

\begin{center}
\begin{flushleft}
No. 5: $X_{4,4} \subset \mbP (1,1,1,2,2,2)$; 
$(A^3) = 2$, Case 3.
\end{flushleft} \nopagebreak
\begin{tabular}{|p{103pt}|p{112pt}|p{103pt}|>{\centering\arraybackslash}p{26pt}|}
\hline
\parbox[c][\myheight][c]{0cm}{} $\msp_{3,4,5} \!=\! 4 \!\times \! \frac{1}{2} (1,1,1)$ &  & $\dP_2$ & \\
\hline
\end{tabular}
\end{center}

\begin{center}
\begin{flushleft}
No. 6: $X_{4,5} \subset \mbP (1,1,1,2,2,3)$; 
$(A^3) = 5/3$.
\end{flushleft} \nopagebreak
\begin{tabular}{|p{103pt}|p{112pt}|p{103pt}|>{\centering\arraybackslash}p{26pt}|}
\hline
\parbox[c][\myheight][c]{0cm}{} $\msp_{3,4} = 2 \times \frac{1}{2} (1,1,1)$ & & $X'_5 \subset (1,1,1,2,1)$ & $(\mathrm{C}_2)$ \\
\hline
\parbox[c][\myheight][c]{0cm}{} $\msp_5 = \frac{1}{3} (1,1,2)$ & & $X'_5 \subset (1,1,1,2,1)$ & \\
\hline
\end{tabular}
\end{center}

\begin{center}
\begin{flushleft}
No. 7: $X_{4,6} \subset \mbP (1,1,1,2,3,3)$; 
$(A^3) = 4/3$, Case 3.
\end{flushleft} \nopagebreak
\begin{tabular}{|p{103pt}|p{112pt}|p{103pt}|>{\centering\arraybackslash}p{26pt}|}
\hline
\parbox[c][\myheight][c]{0cm}{} $\msp_{4,5} = 2 \times \frac{1}{3} (1,1,2)$ & & $X'_6 \subset \mbP (1,1,1,2,2)$ & \\
\hline
\end{tabular}
\end{center}

\begin{center}
\begin{flushleft}
No. 8: $X_{4,6} \subset \mbP (1,1,2,2,2,3)$; 
$(A^3) = 1$, Case 4.
\end{flushleft} \nopagebreak
\begin{tabular}{|p{103pt}|p{112pt}|p{103pt}|>{\centering\arraybackslash}p{26pt}|}
\hline
\parbox[c][\myheight][c]{0cm}{} $\msp_{2,3,4} \!=\! 6 \!\times \! \frac{1}{2} (1,1,1)$ &  & Q.I. & \\
\hline
\end{tabular}
\end{center}

\begin{center}
\begin{flushleft}
No. 9 $X_{5,6} \subset \mbP (1,1,1,2,3,4)$; 
$(A^3) =5/4$.
\end{flushleft} \nopagebreak
\begin{tabular}{|p{103pt}|p{112pt}|p{103pt}|>{\centering\arraybackslash}p{26pt}|}
\hline
\parbox[c][\myheight][c]{0cm}{} $\msp_{3,5} = \frac{1}{2} (1,1,1)$ & & $X'_6 \subset (1,1,1,3,1)$ & $(\mathrm{C}_1)$ \\
\hline
\parbox[c][\myheight][c]{0cm}{} $\msp_5 = \frac{1}{4} (1,1,3)$ & & $X'_6 \subset (1,1,1,3,1)$ & \\
\hline
\end{tabular}
\end{center}

\begin{center}
\begin{flushleft}
No. 10: $X_{5,6} \subset \mbP (1,1,2,2,3,3)$; 
$(A^3) = 5/6$, Case 4.
\end{flushleft} \nopagebreak
\begin{tabular}{|p{103pt}|p{112pt}|p{103pt}|>{\centering\arraybackslash}p{26pt}|}
\hline
\parbox[c][\myheight][c]{0cm}{} $\msp_{2,3} = 3 \times \frac{1}{2} (1,1,1)$ &  & Q.I. & $(\mathrm{C}_2)$ \\
\hline
\parbox[c][\myheight][c]{0cm}{} $\msp_{4,5} = 2 \times \frac{1}{3} (1,1,2)$ & & $X'_6 \subset (1,1,2,2,1)$ & \\
\hline
\end{tabular}
\end{center}

\begin{center}
\begin{flushleft}
No. 11: $X_{6,6} \subset \mbP (1,1,1,2,3,5)$; 
$(A^3) = 6/5$.
\end{flushleft} \nopagebreak
\begin{tabular}{|p{103pt}|p{112pt}|p{103pt}|>{\centering\arraybackslash}p{26pt}|}
\hline
\parbox[c][\myheight][c]{0cm}{} $\msp_5 = \frac{1}{5} (1,2,3)$ &  & $\dP_1$ & \\
\hline
\end{tabular}
\end{center}

\begin{center}
\begin{flushleft}
No. 12: $X_{6,6} \subset \mbP (1,1,2,2,3,4)$; 
$(A^3) = 3/4$, Case 4.
\end{flushleft} \nopagebreak
\begin{tabular}{|p{103pt}|p{112pt}|p{103pt}|>{\centering\arraybackslash}p{26pt}|}
\hline
\parbox[c][\myheight][c]{0cm}{} $\msp_{2,3,5} \!=\! 4 \!\times \! \frac{1}{2} (1,1,1)$ &  & Q.I. & \\
\hline
\parbox[c][\myheight][c]{0cm}{} $\msp_5 = \frac{1}{4} (1,1,3)$ &  & $\dP_1$ & \\
\hline
\end{tabular}
\end{center}

\begin{center}
\begin{flushleft}
No. 13: $X_{6,6} \subset \mbP (1,1,2,3,3,3)$; 
$(A^3) = 2/3$, Case 3.
\end{flushleft} \nopagebreak
\begin{tabular}{|p{103pt}|p{112pt}|p{103pt}|>{\centering\arraybackslash}p{26pt}|}
\hline
\parbox[c][\myheight][c]{0cm}{} $\msp_{3,4,5} \!=\! 4 \!\times \! \frac{1}{3} (1,1,2)$ & & $\dP_1$ & \\
\hline
\end{tabular}
\end{center}

\begin{center}
\begin{flushleft}
No. 14: $X_{6,6} \subset \mbP (1,2,2,2,3,3)$; 
$(A^3) = 1/2$, Case 2.
\end{flushleft} \nopagebreak
\begin{tabular}{|p{103pt}|p{112pt}|p{103pt}|>{\centering\arraybackslash}p{26pt}|}
\hline
\parbox[c][\myheight][c]{0cm}{} $\msp_{1,2,3} \!=\! 9 \!\times \! \frac{1}{2} (1,1,1)$ & $B$, $\{x,y_1,y_2\}$ & & \\
\hline
\end{tabular}
\end{center}

\begin{center}
\begin{flushleft}
No. 15: $X_{6,7} \subset \mbP (1,1,2,2,3,5)$; 
$(A^3) = 7/10$, Case 4.
\end{flushleft} \nopagebreak
\begin{tabular}{|p{103pt}|p{112pt}|p{103pt}|>{\centering\arraybackslash}p{26pt}|}
\hline
\parbox[c][\myheight][c]{0cm}{} $\msp_{2,3} = 3 \times \frac{1}{2} (1,1,1)$ &  & Q.I. & $(\mathrm{C}_2)$ \\
\hline
\parbox[c][\myheight][c]{0cm}{} $\msp_5 = \frac{1}{5} (1,2,3)$ &  & $X'_7 \subset (1,1,2,3,1)$ & \\
\hline
\end{tabular}
\end{center}

\begin{center}
\begin{flushleft}
No. 16: $X_{6,7} \subset \mbP (1,1,2,3,3,4)$; 
$(A^3) = 7/12$, Case 5.
\end{flushleft} \nopagebreak
\begin{tabular}{|p{103pt}|p{112pt}|p{103pt}|>{\centering\arraybackslash}p{26pt}|}
\hline
\parbox[c][\myheight][c]{0cm}{} $\msp_{2,5} = \frac{1}{2} (1,1,1)$ &  & E.I. & $(\mathrm{C}_2)$ \\
\hline
\parbox[c][\myheight][c]{0cm}{} $\msp_{3,4} = 2 \times \frac{1}{3} (1,1,2)$ & & $X'_7 \subset (1,1,2,3,1)$ & $(\mathrm{C}_2)$  \\
\hline
\parbox[c][\myheight][c]{0cm}{} $\msp_5 = \frac{1}{4} (1,1,3)$ &  & $X'_7 \subset (1,1,2,3,1)$ & \\
\hline
\end{tabular}
\end{center}

\begin{center}
\begin{flushleft}
No. 17: $X_{6,8} \subset \mbP (1,1,1,3,4,5)$; 
$(A^3) = 4/5$, Case 3.
\end{flushleft} \nopagebreak
\begin{tabular}{|p{103pt}|p{112pt}|p{103pt}|>{\centering\arraybackslash}p{26pt}|}
\hline
\parbox[c][\myheight][c]{0cm}{} $\msp_5 = \frac{1}{5} (1,1,4)$ & & $X'_8 \subset (1,1,1,4,2)$ & \\
\hline
\end{tabular}
\end{center}

\begin{center}
\begin{flushleft}
No. 18: $X_{6,8} \subset \mbP (1,1,2,3,3,5)$; 
$(A^3) = 8/15$, Case 5.
\end{flushleft} \nopagebreak
\begin{tabular}{|p{103pt}|p{112pt}|p{103pt}|>{\centering\arraybackslash}p{26pt}|}
\hline
\parbox[c][\myheight][c]{0cm}{} $\msp_{3,4} = 2 \times \frac{1}{3} (1,1,2)$ &  & $X'_8 \subset (1,1,2,3,2)$ & $(\mathrm{C}_2)$ \\
\hline
\parbox[c][\myheight][c]{0cm}{} $\msp_5 = \frac{1}{5} (1,2,3)$ &  & $X'_8 \subset (1,1,2,3,2)$ & \\
\hline
\end{tabular}
\end{center}

\begin{center}
\begin{flushleft}
No. 19: $X_{6,8} \subset \mbP (1,1,2,3,4,4)$; 
$(A^3) = 1/2$, Case 3.
\end{flushleft} \nopagebreak
\begin{tabular}{|p{103pt}|p{112pt}|p{103pt}|>{\centering\arraybackslash}p{26pt}|}
\hline
\parbox[c][\myheight][c]{0cm}{} $\msp_{2,4,5} \!=\! 2 \!\times \! \frac{1}{2} (1,1,1)$ & $T \in |B|$ & & $(\mathrm{C}_2)$\\
\hline
\parbox[c][\myheight][c]{0cm}{} $\msp_{4,5} = 2 \times \frac{1}{4} (1,1,3)$ &  & $X'_8 \subset (1,1,2,3,2)$ & \\
\hline
\end{tabular}
\end{center}

\begin{center}
\begin{flushleft}
No. 20: $X_{6,8} \subset \mbP (1,2,2,3,3,4)$; 
$(A^3) = 1/3$, Case 1.
\end{flushleft} \nopagebreak
\begin{tabular}{|p{103pt}|p{112pt}|p{103pt}|>{\centering\arraybackslash}p{26pt}|}
\hline
\parbox[c][\myheight][c]{0cm}{} $\msp_{1,2,5} \!=\! 6 \!\times \! \frac{1}{2} (1,1,1)$ & $T \in |2 B|$  & & $(\mathrm{C}_3)$ \\
\hline
\parbox[c][\myheight][c]{0cm}{} $\msp_{3,4,5} \!=\! 2 \!\times \! \frac{1}{3} (1,1,2)$ &  & Q.I. & \\
\hline
\end{tabular}
\end{center}

\begin{center}
\begin{flushleft}
No. 21: $X_{6,9} \subset \mbP (1,1,2,3,4,5)$; 
$(A^3) = 9/20$, Case 5.
\end{flushleft} \nopagebreak
\begin{tabular}{|p{103pt}|p{112pt}|p{103pt}|>{\centering\arraybackslash}p{26pt}|}
\hline
\parbox[c][\myheight][c]{0cm}{} $\msp_{2,4} = \frac{1}{2} (1,1,1)$ & $T \in |B|$ & & $(\mathrm{C}_3)$ \\
\hline
\parbox[c][\myheight][c]{0cm}{} $\msp_4 = \frac{1}{4} (1,1,3)$ &  & $X'_9 \subset (1,1,2,3,3)$ & \\
\hline
\parbox[c][\myheight][c]{0cm}{} $\msp_5 = \frac{1}{5} (1,2,3)$ &  & $X'_9 \subset (1,1,2,3,3)$ & \\
\hline
\end{tabular}
\end{center}

\begin{center}
\begin{flushleft}
No. 22: $X_{7,8} \subset \mbP (1,1,2,3,4,5)$; 
$(A^3) = 7/15$, Case 5.
\end{flushleft} \nopagebreak
\begin{tabular}{|p{103pt}|p{112pt}|p{103pt}|>{\centering\arraybackslash}p{26pt}|}
\hline
\parbox[c][\myheight][c]{0cm}{} $\msp_{2,4} = 2 \times \frac{1}{2} (1,1,1)$ & $T \in |B|$ & & $(\mathrm{C}_{1,2})$ \\
\hline
\parbox[c][\myheight][c]{0cm}{} $\msp_3 = \frac{1}{3} (1,1,2)$ & & $X'_8 \subset (1,1,2,4,1)$ & $(\mathrm{C}_1)$  \\
\hline
\parbox[c][\myheight][c]{0cm}{} $\msp_5 = \frac{1}{5} (1,1,4)$ &  & $X'_8 \subset (1,1,2,4,1)$ & \\
\hline
\end{tabular}
\end{center}

\begin{center}
\begin{flushleft}
No. 23: $X_{6,10} \subset \mbP (1,1,2,3,5,5)$; 
$(A^3) = 2/5$, Case 3.
\end{flushleft} \nopagebreak
\begin{tabular}{|p{103pt}|p{112pt}|p{103pt}|>{\centering\arraybackslash}p{26pt}|}
\hline
\parbox[c][\myheight][c]{0cm}{} $\msp_{4,5} = 2 \times \frac{1}{5} (1,2,3)$ &  & $X'_{10} \subset (1,1,2,3,4)$ & \\
\hline
\end{tabular}
\end{center}

\begin{center}
\begin{flushleft}
No. 24: $X_{6,10} \subset \mbP (1,2,2,3,4,5)$; 
$(A^3) = 1/4$, Case 2.
\end{flushleft} \nopagebreak
\begin{tabular}{|p{103pt}|p{112pt}|p{103pt}|>{\centering\arraybackslash}p{26pt}|}
\hline
\parbox[c][\myheight][c]{0cm}{} $\msp_{1,2,4} \!=\! 7 \!\times \! \frac{1}{2} (1,1,1)$ &  $4 B + E$, $\{x,y'_1,s\}_{\sharp}$ & & \\
\hline
\parbox[c][\myheight][c]{0cm}{} $\msp_4 = \frac{1}{4} (1,1,3)$ &  & Q.I. & \\
\hline
\end{tabular}
\end{center}

\begin{center}
\begin{flushleft}
No. 25: $X_{8,9} \subset \mbP (1,1,2,3,4,7)$; 
$(A^3) = 3/7$, Case 5.
\end{flushleft} \nopagebreak
\begin{tabular}{|p{103pt}|p{112pt}|p{103pt}|>{\centering\arraybackslash}p{26pt}|}
\hline
\parbox[c][\myheight][c]{0cm}{} $\msp_{2,4} = 2 \times \frac{1}{2} (1,1,1)$ & $T \in |B|$ & & $(\mathrm{C}_1)$ \\
\hline
\parbox[c][\myheight][c]{0cm}{} $\msp_5 = \frac{1}{7} (1,3,4)$ &  & $X'_9 \subset (1,1,3,4,1)$ & \\
\hline
\end{tabular}
\end{center}

\begin{center}
\begin{flushleft}
No. 26: $X_{8,9} \subset \mbP (1,1,3,4,4,5)$; 
$(A^3) = 3/10$, Case 5.
\end{flushleft} \nopagebreak
\begin{tabular}{|p{103pt}|p{112pt}|p{103pt}|>{\centering\arraybackslash}p{26pt}|}
\hline
\parbox[c][\myheight][c]{0cm}{} $\msp_{3,4} = 2 \times \frac{1}{4} (1,1,3)$ &  & $X'_9 \subset (1,1,3,4,1)$ & $(\mathrm{C}_2)$ \\
\hline
\parbox[c][\myheight][c]{0cm}{} $\msp_5 = \frac{1}{5} (1,1,4)$ &  & $X'_9 \subset (1,1,3,4,1)$ & \\
\hline
\end{tabular}
\end{center}

\begin{center}
\begin{flushleft}
No. 27: $X_{8,9} \subset \mbP (1,2,3,3,4,5)$; 
$(A^3) = 1/5$, Case 1.
\end{flushleft} \nopagebreak
\begin{tabular}{|p{103pt}|p{112pt}|p{103pt}|>{\centering\arraybackslash}p{26pt}|}
\hline
\parbox[c][\myheight][c]{0cm}{} $\msp_{1,4} = 2 \times \frac{1}{2} (1,1,1)$ &  $4 B + E$, $\{x,s,t'\}_{\sharp}$ & & $(\mathrm{C}_2)$ \\
\hline
\parbox[c][\myheight][c]{0cm}{} $\msp_{2,3} = 3 \times \frac{1}{3} (1,1,2)$ &  & E.I. & $(\mathrm{C}_2)$ \\
\hline
\parbox[c][\myheight][c]{0cm}{} $\msp_5 = \frac{1}{5} (1,2,3)$ &  & $X'_9 \subset (1,2,3,3,1)$ & \\
\hline
\end{tabular}
\end{center}

\begin{center}
\begin{flushleft}
No. 28: $X_{8,10} \subset \mbP (1,1,2,3,5,7)$; 
$(A^3) = 8/21$, Case 5.
\end{flushleft} \nopagebreak
\begin{tabular}{|p{103pt}|p{112pt}|p{103pt}|>{\centering\arraybackslash}p{26pt}|}
\hline
\parbox[c][\myheight][c]{0cm}{} $\msp_3 = \frac{1}{3} (1,1,2)$ &  & $X'_{10} \subset (1,1,2,5,2)$ & $(\mathrm{C}_1)$ \\
\hline
\parbox[c][\myheight][c]{0cm}{} $\msp_5 = \frac{1}{7} (1,2,5)$ &  & $X'_{10} \subset (1,1,2,5,2)$ & \\
\hline
\end{tabular}
\end{center}

\begin{center}
\begin{flushleft}
No. 29: $X_{8,10} \subset \mbP (1,1,2,4,5,6)$; 
$(A^3) = 1/3$, Case 3.
\end{flushleft} \nopagebreak
\begin{tabular}{|p{103pt}|p{112pt}|p{103pt}|>{\centering\arraybackslash}p{26pt}|}
\hline
\parbox[c][\myheight][c]{0cm}{} $\msp_{2,3,5} \!=\! 3 \!\times \! \frac{1}{2} (1,1,1)$ & $T \in |B|$ & & \\
\hline
\parbox[c][\myheight][c]{0cm}{} $\msp_5 = \frac{1}{6} (1,1,5)$ &  & $X'_{10} \subset (1,1,2,5,2)$ & \\
\hline
\end{tabular}
\end{center}

\begin{center}
\begin{flushleft}
No. 30: $X_{8,10} \subset \mbP (1,1,3,4,5,5)$; 
$(A^3) = 4/15$, Case 3.
\end{flushleft} \nopagebreak
\begin{tabular}{|p{103pt}|p{112pt}|p{103pt}|>{\centering\arraybackslash}p{26pt}|}
\hline
\parbox[c][\myheight][c]{0cm}{} $\msp_2 = \frac{1}{3} (1,1,2)$ &  & Q.I. & $\dagger y^2 z$ \\
\hline
\parbox[c][\myheight][c]{0cm}{} $\msp_{4,5} = 2 \times \frac{1}{5} (1,1,4)$ &  & $X'_{10} \subset (1,1,3,4,2)$ & \\
\hline
\end{tabular}
\end{center}

\begin{center}
\begin{flushleft}
No. 31: $X_{8,10} \subset \mbP (1,2,3,4,4,5)$; 
$(A^3) = 1/6$, Case 1.
\end{flushleft} \nopagebreak
\begin{tabular}{|p{103pt}|p{112pt}|p{103pt}|>{\centering\arraybackslash}p{26pt}|}
\hline
\parbox[c][\myheight][c]{0cm}{} $\msp_{1,3,4} \!=\! 4 \!\times \! \frac{1}{2} (1,1,1)$ &  $3 B + E$, $\{x,z,s_0,s_1\}_{\sharp}$ & & \\
\hline
\parbox[c][\myheight][c]{0cm}{} $\msp_2 = \frac{1}{3} (1,1,2)$ & $T \in |2 B|$ & & $(\mathrm{C}_2)$ \\
\hline
\parbox[c][\myheight][c]{0cm}{} $\msp_{3,4} = 2 \times \frac{1}{4} (1,1,3)$ &  & Q.I. & \\
\hline
\end{tabular}
\end{center}

\begin{center}
\begin{flushleft}
No. 32: $X_{9,10} \subset \mbP (1,1,2,3,5,8)$; 
$(A^3) = 3/8$, Case 5.
\end{flushleft} \nopagebreak
\begin{tabular}{|p{103pt}|p{112pt}|p{103pt}|>{\centering\arraybackslash}p{26pt}|}
\hline
\parbox[c][\myheight][c]{0cm}{} $\msp_{2,5} = \frac{1}{2} (1,1,1)$ & $T \in |B|$ & & $(\mathrm{C}_1)$ \\
\hline
\parbox[c][\myheight][c]{0cm}{} $\msp_5 = \frac{1}{8} (1,3,5)$ &  & $X'_{10} \subset (1,1,3,5,1)$ & \\
\hline
\end{tabular}
\end{center}

\begin{center}
\begin{flushleft}
No. 33: $X_{9,10} \subset \mbP (1,1,3,4,5,6)$; 
$(A^3) = 1/4$, Case 5.
\end{flushleft} \nopagebreak
\begin{tabular}{|p{103pt}|p{112pt}|p{103pt}|>{\centering\arraybackslash}p{26pt}|}
\hline
\parbox[c][\myheight][c]{0cm}{} $\msp_{2,5} = \frac{1}{3} (1,1,2)$ & $B$, $\{x_0,x_1,z',t'\}_{\sharp}$ & & $(\mathrm{C}_2)$\\
\hline
\parbox[c][\myheight][c]{0cm}{} $\msp_3 = \frac{1}{4} (1,1,3)$ &  & $X'_{10} \subset (1,1,3,5,1)$ & $(\mathrm{C}_1)$ \\
\hline
\parbox[c][\myheight][c]{0cm}{} $\msp_5 = \frac{1}{6} (1,1,5)$ &  & $X'_{10} \subset (1,1,3,5,1)$ & \\
\hline
\end{tabular}
\end{center}

\begin{center}
\begin{flushleft}
No. 34: $X_{9,10} \subset \mbP (1,2,2,3,5,7)$; 
$(A^3) = 3/14$, Case 3.
\end{flushleft} \nopagebreak
\begin{tabular}{|p{103pt}|p{112pt}|p{103pt}|>{\centering\arraybackslash}p{26pt}|}
\hline
\parbox[c][\myheight][c]{0cm}{} $\msp_{1,2} = 5 \times \frac{1}{2} (1,1,1)$ & $T \in |2 B|$ & & $(\mathrm{C}_2)$ \\
\hline
\parbox[c][\myheight][c]{0cm}{} $\msp_5 = \frac{1}{7} (1,2,5)$ &  & $X'_{10} \subset (1,2,2,5,1)$ & \\
\hline
\end{tabular}
\end{center}

\begin{center}
\begin{flushleft}
No. 35: $X_{9,10}\subset \mbP (1,2,3,4,5,5)$; 
$(A^3) = 3/20$, Case 1.
\end{flushleft} \nopagebreak
\begin{tabular}{|p{103pt}|p{112pt}|p{103pt}|>{\centering\arraybackslash}p{26pt}|}
\hline
\parbox[c][\myheight][c]{0cm}{} $\msp_{1,3} = 2 \times \frac{1}{2} (1,1,1)$ &  $3 B + E$, $\{x,z,s',t_i'\}_{\sharp}$ & & $(\mathrm{C}_2)$ \\
\hline
\parbox[c][\myheight][c]{0cm}{} $\msp_3 = \frac{1}{4} (1,1,3)$ &  & Q.I. & \\
\hline
\parbox[c][\myheight][c]{0cm}{} $\msp_{4,5} = 2 \times \frac{1}{5} (1,2,3)$ &  & $X'_{10} \subset (1,2,3,4,1)$ & \\
\hline
\end{tabular}
\end{center}

\begin{center}
\begin{flushleft}
No. 36: $X_{8,12} \subset \mbP (1,1,3,4,5,7)$; 
$(A^3) = 8/35$, Case 5.
\end{flushleft} \nopagebreak
\begin{tabular}{|p{103pt}|p{112pt}|p{103pt}|>{\centering\arraybackslash}p{26pt}|}
\hline
\parbox[c][\myheight][c]{0cm}{} $\msp_4 = \frac{1}{5} (1,1,4)$ &  & $X'_{12} \subset (1,1,3,4,4)$ & \\
\hline
\parbox[c][\myheight][c]{0cm}{} $\msp_5 = \frac{1}{7} (1,3,4)$ &  & $X'_{12} \subset (1,1,3,4,4)$ & \\
\hline
\end{tabular}
\end{center}

\begin{center}
\begin{flushleft}
No. 37: $X_{8,12} \subset \mbP (1,2,3,4,5,6)$; 
$(A^3) = 2/15$, Case 1.
\end{flushleft} \nopagebreak
\begin{tabular}{|p{103pt}|p{112pt}|p{103pt}|>{\centering\arraybackslash}p{26pt}|}
\hline
\parbox[c][\myheight][c]{0cm}{} $\msp_{1,3,5} \!=\! 4 \!\times \! \frac{1}{2} (1,1,1)$ &  $3 B + E$, $\{x,z,s',u'\}_{\sharp}$ & & $(\mathrm{C}_1)$ \\
\hline
\parbox[c][\myheight][c]{0cm}{} $\msp_{2,5} = 2 \times \frac{1}{3} (1,1,2)$ & $T \in |2 B|$ & & $(\mathrm{C}_{1,2})$ \\
\hline
\parbox[c][\myheight][c]{0cm}{} $\msp_4 = \frac{1}{5} (1,1,4)$ &  & Q.I. & \\
\hline
\end{tabular}
\end{center}

\begin{center}
\begin{flushleft}
No. 38: $X_{9,12} \subset \mbP (1,2,3,4,5,7)$; 
$(A^3) = 9/70$, Case 2.
\end{flushleft} \nopagebreak
\begin{tabular}{|p{103pt}|p{112pt}|p{103pt}|>{\centering\arraybackslash}p{26pt}|}
\hline
\parbox[c][\myheight][c]{0cm}{} $\msp_{1,3} = 3 \times \frac{1}{2} (1,1,1)$ &  $3 B + E$, $\{x,z,s',u'\}_{\sharp}$ & & $(\mathrm{C}_2)$ \\
\hline
\parbox[c][\myheight][c]{0cm}{} $\msp_4 = \frac{1}{5} (1,2,3)$ &  & $X'_{12} \subset (1,2,3,4,3)$ & \\
\hline
\parbox[c][\myheight][c]{0cm}{} $\msp_5 = \frac{1}{7} (1,3,4)$ &  & $X'_{12} \subset (1,2,3,4,3)$ & \\
\hline
\end{tabular}
\end{center}

\begin{center}
\begin{flushleft}
No. 39: $X_{10,11} \subset \mbP (1,2,3,4,5,7)$; 
$(A^3) = 11/84$, Case 2.
\end{flushleft} \nopagebreak
\begin{tabular}{|p{103pt}|p{112pt}|p{103pt}|>{\centering\arraybackslash}p{26pt}|}
\hline
\parbox[c][\myheight][c]{0cm}{} $\msp_{1,3} = 2 \times \frac{1}{2} (1,1,1)$ &  $3 B + E$, $\{x,z,s,u'\}_{\sharp}$ & & $(\mathrm{C}_2)$ \\
\hline
\parbox[c][\myheight][c]{0cm}{} $\msp_2 = \frac{1}{3} (1,1,2)$ & $T \in |2 B|$ & & $(\mathrm{C}_1)$ \\
\hline
\parbox[c][\myheight][c]{0cm}{} $\msp_3 = \frac{1}{4} (1,1,3)$ &  & Q.I. & \\
\hline
\parbox[c][\myheight][c]{0cm}{} $\msp_5 = \frac{1}{7} (1,2,5)$ &  & $X'_{11} \subset (1,2,3,5,1)$ & \\
\hline
\end{tabular}
\end{center}

\begin{center}
\begin{flushleft}
No. 40: $X_{10,12} \subset \mbP (1,1,3,4,5,9)$; 
$(A^3) = 2/9$, Case 5.
\end{flushleft} \nopagebreak
\begin{tabular}{|p{103pt}|p{112pt}|p{103pt}|>{\centering\arraybackslash}p{26pt}|}
\hline
\parbox[c][\myheight][c]{0cm}{} $\msp_{2,5} = \frac{1}{3} (1,1,2)$ &  & Q.I. & $\dagger y^2 z$ \\
\hline
\parbox[c][\myheight][c]{0cm}{} $\msp_5 = \frac{1}{9} (1,4,5)$ &  & $X'_{12} \subset (1,1,4,5,2)$ & \\
\hline
\end{tabular}
\end{center}

\begin{center}
\begin{flushleft}
No. 41: $X_{10,12} \subset \mbP (1,1,3,5,6,7)$; 
$(A^3) = 4/21$, Case 3.
\end{flushleft} \nopagebreak
\begin{tabular}{|p{103pt}|p{112pt}|p{103pt}|>{\centering\arraybackslash}p{26pt}|}
\hline
\parbox[c][\myheight][c]{0cm}{} $\msp_{2,4} = 2 \times \frac{1}{3} (1,1,2)$ &  & E.I. & \\
\hline
\parbox[c][\myheight][c]{0cm}{} $\msp_5 = \frac{1}{7} (1,1,6)$ &  & $X'_{12} \subset (1,1,3,6,2)$ & \\
\hline 
\end{tabular}
\end{center}

\begin{center}
\begin{flushleft}
No. 42: $X_{10, 12} \subset \mbP (1,1,4,5,6,6)$; 
$(A^3) = 1/6$, Case 3.
\end{flushleft} \nopagebreak
\begin{tabular}{|p{103pt}|p{112pt}|p{103pt}|>{\centering\arraybackslash}p{26pt}|}
\hline
\parbox[c][\myheight][c]{0cm}{} $\msp_{2,4,5} = \frac{1}{2} (1,1,1)$ & $T \in |B|$ & & \\
\hline
\parbox[c][\myheight][c]{0cm}{} $\msp_{4,5} = 2 \times \frac{1}{6} (1,1,5)$ &  & $X'_{12} \subset (1,1,4,5,2)$ & \\
\hline
\end{tabular}
\end{center}

\begin{center}
\begin{flushleft}
No. 43: $X_{10,12} \subset \mbP (1,2,3,4,5,8)$; 
$(A^3) = 1/8$, Case 2.
\end{flushleft} \nopagebreak
\begin{tabular}{|p{103pt}|p{112pt}|p{103pt}|>{\centering\arraybackslash}p{26pt}|}
\hline
\parbox[c][\myheight][c]{0cm}{} $\msp_{1,2,5} \!=\! 3 \!\times \! \frac{1}{2} (1,1,1)$ & $8 B + 3 E$, $\{x,z,s',u'\}_{\sharp}$ & & \\
\hline
\parbox[c][\myheight][c]{0cm}{} $\msp_{2,5} = \frac{1}{4} (1,1,3)$ &  & Q.I.  &\\
\hline
\parbox[c][\myheight][c]{0cm}{} $\msp_5 = \frac{1}{8} (1,1,7)$ &  & $X'_{12} \subset (1,2,3,5,2)$ & \\
\hline
\end{tabular}
\end{center}

\begin{center}
\begin{flushleft}
No. 44: $X_{10,12} \subset \mbP (1,2,3,5,5,7)$; 
$(A^3) = 4/35$, Case 1.
\end{flushleft} \nopagebreak
\begin{tabular}{|p{103pt}|p{112pt}|p{103pt}|>{\centering\arraybackslash}p{26pt}|}
\hline
\parbox[c][\myheight][c]{0cm}{} $\msp_{3,4} = 2 \times \frac{1}{5} (1,2,3)$ &  & $X'_{12} \subset (1,2,3,5,2)$ & $(\mathrm{C}_2)$ \\
\hline
\parbox[c][\myheight][c]{0cm}{} $\msp_5 = \frac{1}{7} (1,2,5)$ &  & $X'_{12} \subset (1,2,3,5,2)$ & \\
\hline
\end{tabular}
\end{center}

\begin{center}
\begin{flushleft}
No. 45: $X_{10,12} \subset \mbP (1,2,4,5,5,6)$; 
$(A^3) = 1/10$, Case 1.
\end{flushleft} \nopagebreak
\begin{tabular}{|p{103pt}|p{112pt}|p{103pt}|>{\centering\arraybackslash}p{26pt}|}
\hline
\parbox[c][\myheight][c]{0cm}{} $\msp_{1,2,5} \!=\! 5 \!\times \! \frac{1}{2} (1,1,1)$ & $5 B + 2 E$, $\{x,s_0,s_1\}$ & & \\
\hline
\parbox[c][\myheight][c]{0cm}{} $\msp_{3,4} = 2 \times \frac{1}{5} (1,1,4)$ &  & Q.I. & \\
\hline
\end{tabular}
\end{center}

\begin{center}
\begin{flushleft}
No. 46: $X_{10,12} \subset \mbP (1,3,3,4,5,7)$; 
$(A^3) = 2/21$, Case 1.
\end{flushleft} \nopagebreak
\begin{tabular}{|p{103pt}|p{112pt}|p{103pt}|>{\centering\arraybackslash}p{26pt}|}
\hline
\parbox[c][\myheight][c]{0cm}{} $\msp_{1,2} = 4 \times \frac{1}{3} (1,1,2)$ &  $7 B + E$, $\{x,y'_1,t'\}_{\sharp}$ & & $(\mathrm{C}_2)$ \\
\hline
\parbox[c][\myheight][c]{0cm}{} $\msp_5 = \frac{1}{7} (1,3,4)$ &  & $X'_{12} \subset (1,3,3,4,2)$ & \\
\hline
\end{tabular}
\end{center}

\begin{center}
\begin{flushleft}
No. 47: $X_{10,12} \subset \mbP (1,3,4,4,5,6)$; 
$(A^3) = 1/12$, Case 1.
\end{flushleft} \nopagebreak
\begin{tabular}{|p{103pt}|p{112pt}|p{103pt}|>{\centering\arraybackslash}p{26pt}|}
\hline
\parbox[c][\myheight][c]{0cm}{} $\msp_{1,5} = 2 \times \frac{1}{3} (1,1,2)$ &  $6 B + E$, $\{x, z'_1, t'\}_{\sharp}$ & & $(\mathrm{C}_2)$ \\
\hline
\parbox[c][\myheight][c]{0cm}{} $\msp_{2,3} = 3 \times \frac{1}{4} (1,1,3)$ & $B$, $\{x,y,z'_1\}_{\sharp}$ & & \\
\hline
\parbox[c][\myheight][c]{0cm}{} $\msp_{2,3,5} = \frac{1}{2} (1,1,1)$ &  $5 B + 2 E$, $\{x,y,s\}$ & & \\
\hline
\end{tabular}
\end{center}

\begin{center}
\begin{flushleft}
No. 48: $X_{11,12} \subset \mbP (1,1,4,5,6,7)$; 
$(A^3) = 11/70$, Case 5.
\end{flushleft} \nopagebreak
\begin{tabular}{|p{103pt}|p{112pt}|p{103pt}|>{\centering\arraybackslash}p{26pt}|}
\hline
\parbox[c][\myheight][c]{0cm}{} $\msp_{2,4} = \frac{1}{2} (1,1,1)$ & $T \in |B|$ & & $(\mathrm{C}_1)$ \\
\hline
\parbox[c][\myheight][c]{0cm}{} $\msp_3 = \frac{1}{5} (1,1,4)$ &  & $X'_{12} \subset (1,1,4,6,1)$ & $(\mathrm{C}_1)$ \\
\hline
\parbox[c][\myheight][c]{0cm}{} $\msp_5 = \frac{1}{7} (1,1,6)$ &  & $X'_{12} \subset (1,1,4,6,1)$ & \\
\hline
\end{tabular}
\end{center}

\begin{center}
\begin{flushleft}
No. 49: $X_{10,14} \subset \mbP (1,1,2,5,7,9)$; 
$(A^3) = 2/9$, Case 3.
\end{flushleft} \nopagebreak
\begin{tabular}{|p{103pt}|p{112pt}|p{103pt}|>{\centering\arraybackslash}p{26pt}|}
\hline
\parbox[c][\myheight][c]{0cm}{} $\msp_5 = \frac{1}{9} (1,2,7)$ &  & $X'_{14} \subset (1,1,2,7,4)$ & \\
\hline
\end{tabular}
\end{center}

\begin{center}
\begin{flushleft}
No. 50: $X_{10,14} \subset \mbP (1,2,3,5,7,7)$; 
$(A^3) = 2/21$, Case 2.
\end{flushleft} \nopagebreak
\begin{tabular}{|p{103pt}|p{112pt}|p{103pt}|>{\centering\arraybackslash}p{26pt}|}
\hline
\parbox[c][\myheight][c]{0cm}{} $\msp_2 = \frac{1}{3} (1,1,2)$ &  $7 B + E$, $\{x,y,t'_i\}_{\sharp}$ & & \\
\hline
\parbox[c][\myheight][c]{0cm}{} $\msp_{4,5} = 2 \times \frac{1}{7} (1,2,5)$ &  & $X'_{14} \subset (1,2,3,5,4)$ & \\
\hline
\end{tabular}
\end{center}

\begin{center}
\begin{flushleft}
No. 51: $X_{10,14} \subset \mbP (1,2,4,5,6,7)$; 
$(A^3) = 1/12$, Case 1.
\end{flushleft} \nopagebreak
\begin{tabular}{|p{103pt}|p{112pt}|p{103pt}|>{\centering\arraybackslash}p{26pt}|}
\hline
\parbox[c][\myheight][c]{0cm}{} $\msp_{1,2,4} \!=\! 5 \!\times \! \frac{1}{2} (1,1,1)$ &  $5 B + 2 E$, $\{x,z',s,t'\}_{\sharp}$ & & \\
\hline
\parbox[c][\myheight][c]{0cm}{} $\msp_2 = \frac{1}{4} (1,1,3)$ & $T \in |2 B|$ & & $(\mathrm{C}_3)$ \\
\hline
\parbox[c][\myheight][c]{0cm}{} $\msp_4 = \frac{1}{6} (1,1,5)$ &  & Q.I. & \\
\hline
\end{tabular}
\end{center}

\begin{center}
\begin{flushleft}
No. 52: $X_{10,15} \subset \mbP (1,2,3,5,7,8)$; 
$(A^3) = 5/56$, Case 2.
\end{flushleft} \nopagebreak
\begin{tabular}{|p{103pt}|p{112pt}|p{103pt}|>{\centering\arraybackslash}p{26pt}|}
\hline
\parbox[c][\myheight][c]{0cm}{} $\msp_{1,5} = \frac{1}{2} (1,1,1)$ &  $5 B + 2 E$, $\{x,z,s\}$ & & $(\mathrm{C}_2)$ \\
\hline
\parbox[c][\myheight][c]{0cm}{} $\msp_4 = \frac{1}{7} (1,2,5)$ & & $X'_{15} \subset (1,2,3,5,5)$ & \\
\hline
\parbox[c][\myheight][c]{0cm}{} $\msp_5 = \frac{1}{8} (1,3,5)$ &  & $X'_{15} \subset (1,2,3,5,5)$ & \\
\hline
\end{tabular}
\end{center}

\begin{center}
\begin{flushleft}
No. 53: $X_{12,13} \subset \mbP (1,3,4,5,6,7)$; 
$(A^3) = 13/210$, Case 1.
\end{flushleft} \nopagebreak
\begin{tabular}{|p{103pt}|p{112pt}|p{103pt}|>{\centering\arraybackslash}p{26pt}|}
\hline
\parbox[c][\myheight][c]{0cm}{} $\msp_{1,4} = 2 \times \frac{1}{3} (1,1,2)$ &  $5 B + E$, $\{x,s,t,u'\}_{\sharp}$ & & $(\mathrm{C}_2)$ \\
\hline
\parbox[c][\myheight][c]{0cm}{} $\msp_{2,4} = \frac{1}{2} (1,1,1)$ &  $7 B + 3 B$, $\{x,y,s,u\}$ & & \\
\hline
\parbox[c][\myheight][c]{0cm}{} $\msp_3 = \frac{1}{5} (1,1,4)$ &  & E.I. & \\
\hline
\parbox[c][\myheight][c]{0cm}{} $\msp_5 = \frac{1}{7} (1,3,4)$ &  & $X'_{13} \subset (1,3,4,5,1)$ & \\
\hline
\end{tabular}
\end{center}

\begin{center}
\begin{flushleft}
No. 54: $X_{12,14} \subset \mbP (1,1,3,4,7,11)$; 
$(A^3) = 2/11$, Case 5.
\end{flushleft} \nopagebreak
\begin{tabular}{|p{103pt}|p{112pt}|p{103pt}|>{\centering\arraybackslash}p{26pt}|}
\hline
\parbox[c][\myheight][c]{0cm}{} $\msp_5 = \frac{1}{11} (1,4,7)$ &  & $X'_{14} \subset (1,1,4,7,2)$ & \\
\hline
\end{tabular}
\end{center}

\begin{center}
\begin{flushleft}
No. 55: $X_{12,14} \subset \mbP (1,1,4,6,7,8)$; 
$(A^3) = 1/8$, Case 3.
\end{flushleft} \nopagebreak
\begin{tabular}{|p{103pt}|p{112pt}|p{103pt}|>{\centering\arraybackslash}p{26pt}|}
\hline
\parbox[c][\myheight][c]{0cm}{} $\msp_{2,3,5} = \frac{1}{2} (1,1,1)$ & $T \in |B|$ & & \\
\hline
\parbox[c][\myheight][c]{0cm}{} $\msp_{2,5} = \frac{1}{4} (1,1,3)$ &  & Q.I.  & \\
\hline
\parbox[c][\myheight][c]{0cm}{} $\msp_5 = \frac{1}{8} (1,1,7)$ &  & $X'_{14} \subset (1,1,4,7,2)$ & \\
\hline
\end{tabular}
\end{center}

\begin{center}
\begin{flushleft}
No. 56: $X_{12,14} \subset \mbP (1,2,3,4,7,10)$; 
$(A^3) = 1/10$, Case 2.
\end{flushleft} \nopagebreak
\begin{tabular}{|p{103pt}|p{112pt}|p{103pt}|>{\centering\arraybackslash}p{26pt}|}
\hline
\parbox[c][\myheight][c]{0cm}{} $\msp_{1,3,5} \!=\! 4 \!\times \! \frac{1}{2} (1,1,1)$ &  $5 B + 2 E$, $\{x,z,s',u'\}_{\sharp}$ & & \\
\hline
\parbox[c][\myheight][c]{0cm}{} $\msp_5 = \frac{1}{10} (1,3,7)$ &  & $X'_{14} \subset (1,2,3,7,2)$ & \\
\hline
\end{tabular}
\end{center}

\begin{center}
\begin{flushleft}
No. 57: $X_{12,14} \subset \mbP (1,2,3,5,7,9)$; 
$(A^3) = 4/45$, Case 2.
\end{flushleft} \nopagebreak
\begin{tabular}{|p{103pt}|p{112pt}|p{103pt}|>{\centering\arraybackslash}p{26pt}|}
\hline
\parbox[c][\myheight][c]{0cm}{} $\msp_{2,5} = \frac{1}{3} (1,1,2)$ & $T \in |2 B|$ & & $(\mathrm{C}_{1,2})$ \\
\hline
\parbox[c][\myheight][c]{0cm}{} $\msp_3 = \frac{1}{5} (1,2,3)$ &  & $X'_{14} \subset (1,2,3,7,2)$ & $(\mathrm{C}_1)$ \\
\hline
\parbox[c][\myheight][c]{0cm}{} $\msp_5 = \frac{1}{9} (1,2,7)$ &  & $X'_{14} \subset (1,2,3,7,2)$ & \\
\hline
\end{tabular}
\end{center}

\begin{center}
\begin{flushleft}
No. 58: $X_{12,14} \subset \mbP (1,3,4,5,7,7)$; 
$(A^3) = 2/35$, Case 1.
\end{flushleft} \nopagebreak
\begin{tabular}{|p{103pt}|p{112pt}|p{103pt}|>{\centering\arraybackslash}p{26pt}|}
\hline
\parbox[c][\myheight][c]{0cm}{} $\msp_3 = \frac{1}{5} (1,2,3)$ &  & Q.I. & \\
\hline
\parbox[c][\myheight][c]{0cm}{} $\msp_{4,5} = 2 \times \frac{1}{7} (1,3,4)$ &  & $X'_{14} \subset (1,3,4,5,2)$ & \\
\hline
\end{tabular}
\end{center}

\begin{center}
\begin{flushleft}
No. 59: $X_{12,14} \subset \mbP (1,4,4,5,6,7)$; 
$(A^3) = 1/20$, Case 1.
\end{flushleft} \nopagebreak
\begin{tabular}{|p{103pt}|p{112pt}|p{103pt}|>{\centering\arraybackslash}p{26pt}|}
\hline
\parbox[c][\myheight][c]{0cm}{} $\msp_{1,2} = 3 \times \frac{1}{4} (1,1,3)$ & $T \in |4 B|$ & & $(\mathrm{C}_2)$ \\
\hline
\parbox[c][\myheight][c]{0cm}{} $\msp_{1,2,4} \!=\! 2 \!\times \! \frac{1}{2} (1,1,1)$ &  $7 B + 3 E$, $\{x,z,t\}$ & & \\
\hline
\parbox[c][\myheight][c]{0cm}{} $\msp_3 = \frac{1}{5} (1,1,4)$ &  $B$, $\{x,y_0,y_1\}$ & & \\
\hline
\end{tabular}
\end{center}

\begin{center}
\begin{flushleft}
No. 60: $X_{12,14} \subset \mbP (2,3,4,5,6,7)$; 
$(A^3) = 1/30$, Case 1.
\end{flushleft} \nopagebreak
\begin{tabular}{|p{103pt}|p{112pt}|p{103pt}|>{\centering\arraybackslash}p{26pt}|}
\hline
\parbox[c][\myheight][c]{0cm}{} $\msp_{0,2,4} \!=\! 7 \!\times \! \frac{1}{2} (1,1,1)$ &  $7 B + 3 E$, $\{y,s,u\}$ & & \\
\hline
\parbox[c][\myheight][c]{0cm}{} $\msp_{1,4} = 2 \times \frac{1}{3} (1,1,2)$ &  $4 B + E$, $\{x,z,s\}$ & & \\
\hline
\parbox[c][\myheight][c]{0cm}{} $\msp_3 = \frac{1}{5} (1,2,3)$ &  $2 B$, $\{x,y,z\}$ & & \\
\hline
\end{tabular}
\end{center}

\begin{center}
\begin{flushleft}
No. 61: $X_{12,15} \subset \mbP (1,1,4,5,6,11)$; 
$(A^3) = 3/22$, Case 5.
\end{flushleft} \nopagebreak
\begin{tabular}{|p{103pt}|p{112pt}|p{103pt}|>{\centering\arraybackslash}p{26pt}|}
\hline
\parbox[c][\myheight][c]{0cm}{} $\msp_{2,4} = \frac{1}{2} (1,1,1)$ & $T \in |B|$ & & \\
\hline
\parbox[c][\myheight][c]{0cm}{} $\msp_5 = \frac{1}{11} (1,5,6)$ &  & $X'_{15} \subset (1,1,5,6,3)$ & \\
\hline
\end{tabular}
\end{center}

\begin{center}
\begin{flushleft}
No. 62: $X_{12,15} \subset \mbP (1,3,4,5,6,9)$; 
$(A^3) = 1/18$, Case 1.
\end{flushleft} \nopagebreak
\begin{tabular}{|p{103pt}|p{112pt}|p{103pt}|>{\centering\arraybackslash}p{26pt}|}
\hline
\parbox[c][\myheight][c]{0cm}{} $\msp_{1,4,5} \!=\! 3 \!\times \! \frac{1}{3} (1,1,2)$ &  $9 B + 2 E$, $\{x,s,t',u'\}_{\sharp}$ & & \\
\hline
\parbox[c][\myheight][c]{0cm}{} $\msp_{2,4} = \frac{1}{2} (1,1,1)$ &  $9 B + 4 E$, $\{x,y,s,u\}$ & & \\
\hline
\parbox[c][\myheight][c]{0cm}{} $\msp_5 = \frac{1}{9} (1,4,5)$ &  & $X'_{15} \subset (1,3,4,5,3)$ & \\
\hline
\end{tabular}
\end{center}

\begin{center}
\begin{flushleft}
No. 63: $X_{12,15} \subset \mbP (1,3,4,5,7,8)$; 
$(A^3) = 3/56$, Case 1.
\end{flushleft} \nopagebreak
\begin{tabular}{|p{103pt}|p{112pt}|p{103pt}|>{\centering\arraybackslash}p{26pt}|}
\hline
\parbox[c][\myheight][c]{0cm}{} $\msp_{2,5} = \frac{1}{4} (1,1,3)$ & $T \in |3 B|$ & & $(\mathrm{C}_2)$ \\
\hline
\parbox[c][\myheight][c]{0cm}{} $\msp_4 = \frac{1}{7} (1,3,4)$ & & $X'_{15} \subset (1,3,4,5,3)$ & \\
\hline
\parbox[c][\myheight][c]{0cm}{} $\msp_5 = \frac{1}{8} (1,3,5)$ &  & $X'_{15} \subset (1,3,4,5,3)$ & \\
\hline
\end{tabular}
\end{center}

\begin{center}
\begin{flushleft}
No. 64: $X_{12,16} \subset \mbP (1,2,5,6,7,8)$; 
$(A^3) = 2/35$, Case 1.
\end{flushleft} \nopagebreak
\begin{tabular}{|p{103pt}|p{112pt}|p{103pt}|>{\centering\arraybackslash}p{26pt}|}
\hline
\parbox[c][\myheight][c]{0cm}{} $\msp_{1,3,5} \!=\! 4 \!\times \! \frac{1}{2} (1,1,1)$ &  $7 B + 3 E$, $\{x,z,t\}$ & & \\
\hline
\parbox[c][\myheight][c]{0cm}{} $\msp_2 = \frac{1}{5} (1,2,3)$ &  & Q.I. & $\dagger z^2 s$ \\
\hline
\parbox[c][\myheight][c]{0cm}{} $\msp_4 = \frac{1}{7} (1,1,6)$ &  & Q.I. & \\
\hline
\end{tabular}
\end{center}

\begin{center}
\begin{flushleft}
No. 65: $X_{14,15} \subset \mbP (1,2,3,5,7,12)$; 
$(A^3) = 1/12$, Case 3.
\end{flushleft} \nopagebreak
\begin{tabular}{|p{103pt}|p{112pt}|p{103pt}|>{\centering\arraybackslash}p{26pt}|}
\hline
\parbox[c][\myheight][c]{0cm}{} $\msp_{1,5} = \frac{1}{2} (1,1,1)$ &  $12 B + 5 E$, $\{x,z,s,u'\}_{\sharp}$ & & \\
\hline
\parbox[c][\myheight][c]{0cm}{} $\msp_{2,5} = \frac{1}{3} (1,1,2)$ & $T \in |2 B|$ & & $\dagger z^3 s$ \\
\hline
\parbox[c][\myheight][c]{0cm}{} $\msp_5 = \frac{1}{12} (1,5,7)$ &  & $X'_{15} \subset (1,2,5,7,1)$ & \\
\hline
\end{tabular}
\end{center}

\begin{center}
\begin{flushleft}
No. 66: $X_{14,15} \subset \mbP (1,2,5,6,7,9)$; 
$(A^3) = 1/18$, Case 1.
\end{flushleft} \nopagebreak
\begin{tabular}{|p{103pt}|p{112pt}|p{103pt}|>{\centering\arraybackslash}p{26pt}|}
\hline
\parbox[c][\myheight][c]{0cm}{} $\msp_{1,3} = 2 \times \frac{1}{2} (1,1,1)$ & $9 B + 4 E$, $\{x,z,t,u\}$ & & \\
\hline
\parbox[c][\myheight][c]{0cm}{} $\msp_3 = \frac{1}{6} (1,1,5)$ &  & Q.I. & \\
\hline
\parbox[c][\myheight][c]{0cm}{} $\msp_5 = \frac{1}{9} (1,2,7)$ &  & $X'_{15} \subset (1,2,5,7,1)$ & \\
\hline
\end{tabular}
\end{center}

\begin{center}
\begin{flushleft}
No. 67: $X_{14,15} \subset \mbP (1,3,4,5,7,10)$; 
$(A^3) = 1/20$, Case 1.
\end{flushleft} \nopagebreak
\begin{tabular}{|p{103pt}|p{112pt}|p{103pt}|>{\centering\arraybackslash}p{26pt}|}
\hline
\parbox[c][\myheight][c]{0cm}{} $\msp_2 = \frac{1}{4} (1,1,3)$ & $10B+E$, $\{x,y,u'\}_{\sharp}$ & & \\
\hline
\parbox[c][\myheight][c]{0cm}{} $\msp_{2,5} = \frac{1}{2} (1,1,1)$ &  $7B+3E$, $\{x,y,s,t\}$ & & \\
\hline
\parbox[c][\myheight][c]{0cm}{} $\msp_{3,5} = \frac{1}{5} (1,2,3)$ &  & Q.I. & \\
\hline
\parbox[c][\myheight][c]{0cm}{} $\msp_5 = \frac{1}{10} (1,3,7)$ &  & $X'_{15} \subset (1,3,4,7,1)$ & \\
\hline
\end{tabular}
\end{center}

\begin{center}
\begin{flushleft}
No. 68: $X_{14,15} \subset \mbP (1,3,5,6,7,8)$; 
$(A^3) = 1/24$, Case 1.
\end{flushleft} \nopagebreak
\begin{tabular}{|p{103pt}|p{112pt}|p{103pt}|>{\centering\arraybackslash}p{26pt}|}
\hline
\parbox[c][\myheight][c]{0cm}{} $\msp_{1,3} = 2 \times \frac{1}{3} (1,1,2)$ &  $4 B + E$, $\{x,z,s',u\}_{\sharp}$ & & \\
\hline
\parbox[c][\myheight][c]{0cm}{} $\msp_3 = \frac{1}{6} (1,1,5)$ &  & E.I. & \\
\hline
\parbox[c][\myheight][c]{0cm}{} $\msp_5 = \frac{1}{8} (1,3,5)$ &  & $X'_{15} \subset (1,3,5,6,1)$ & \\
\hline
\end{tabular}
\end{center}

\begin{center}
\begin{flushleft}
No. 69: $X_{14,16} \subset \mbP (1,1,5,7,8,9)$; 
$(A^3) = 4/45$, Case 3.
\end{flushleft} \nopagebreak
\begin{tabular}{|p{103pt}|p{112pt}|p{103pt}|>{\centering\arraybackslash}p{26pt}|}
\hline
\parbox[c][\myheight][c]{0cm}{} $\msp_2 = \frac{1}{5} (1,2,3)$ & $T = B - E$  & & \\
\hline
\parbox[c][\myheight][c]{0cm}{} $\msp_5 = \frac{1}{9} (1,1,8)$ &  & $X'_{16} \subset (1,1,5,8,2)$ & \\
\hline
\end{tabular}
\end{center}

\begin{center}
\begin{flushleft}
No. 70: $X_{14,16} \subset \mbP (1,3,4,5,7,11)$; 
$(A^3) = 8/165$, Case 1.
\end{flushleft} \nopagebreak
\begin{tabular}{|p{103pt}|p{112pt}|p{103pt}|>{\centering\arraybackslash}p{26pt}|}
\hline
\parbox[c][\myheight][c]{0cm}{} $\msp_1 = \frac{1}{3} (1,1,2)$ & $5 B + E$, $\{x,s,t',u'\}$ & & $(\mathrm{C}_1)$ \\
\hline
\parbox[c][\myheight][c]{0cm}{} $\msp_3 = \frac{1}{5} (1,2,3)$ &  & Q.I. & \\
\hline
\parbox[c][\myheight][c]{0cm}{} $\msp_5 = \frac{1}{11} (1,4,7)$ &  & $X'_{16} \subset (1,3,4,7,2)$ & \\
\hline
\end{tabular}
\end{center}

\begin{center}
\begin{flushleft}
No. 71: $X_{14,16} \subset \mbP (1,4,5,6,7,8)$; 
$(A^3) = 1/30$, Case 1.
\end{flushleft} \nopagebreak
\begin{tabular}{|p{103pt}|p{112pt}|p{103pt}|>{\centering\arraybackslash}p{26pt}|}
\hline
\parbox[c][\myheight][c]{0cm}{} $\msp_{1,3,5} = \frac{1}{2} (1,1,1)$ &  $7 B + 3 E$, $\{x,z,t\}$ & & \\
\hline
\parbox[c][\myheight][c]{0cm}{} $\msp_{1,5} = 2 \times \frac{1}{4} (1,1,3)$ &  $7 B + E$, $\{x,t,u'\}_{\sharp}$ & & $(\mathrm{C}_2)$ \\
\hline
\parbox[c][\myheight][c]{0cm}{} $\msp_2 = \frac{1}{5} (1,2,3)$ &  $B$, $\{x,y,s'\}$ & & $\dagger z^2 s$ \\
\hline
\parbox[c][\myheight][c]{0cm}{} $\msp_3 = \frac{1}{6} (1,1,5)$ &  $B$, $\{x,y,z\}$ & & \\
\hline
\end{tabular}
\end{center}

\begin{center}
\begin{flushleft}
No. 72: $X_{15,16} \subset \mbP (1,2,3,5,8,13)$; 
$(A^3) = 1/13$, Case 3.
\end{flushleft} \nopagebreak
\begin{tabular}{|p{103pt}|p{112pt}|p{103pt}|>{\centering\arraybackslash}p{26pt}|}
\hline
\parbox[c][\myheight][c]{0cm}{} $\msp_{1,4} = 2 \times \frac{1}{2} (1,1,1)$ &  $5B + 2 E$, $\{x,z,s,t'\}_{\sharp}$ & & \\
\hline
\parbox[c][\myheight][c]{0cm}{} $\msp_5 = \frac{1}{13} (1,5,8)$ &  & $X'_{16} \subset (1,2,5,8, 1)$ & \\
\hline
\end{tabular}
\end{center}

\begin{center}
\begin{flushleft}
No. 73: $X_{15,16} \subset \mbP (1,3,4,5,8,11)$; 
$(A^3) = 1/22$, Case 1.
\end{flushleft} \nopagebreak
\begin{tabular}{|p{103pt}|p{112pt}|p{103pt}|>{\centering\arraybackslash}p{26pt}|}
\hline
\parbox[c][\myheight][c]{0cm}{} $\msp_{2,4} = 2 \times \frac{1}{4} (1,1,3)$ & $T \in |3 B|$ & & \\
\hline
\parbox[c][\myheight][c]{0cm}{} $\msp_5 = \frac{1}{11} (1,3,8)$ &  & $X'_{16} \subset (1,3,4,8, 1)$ & \\
\hline
\end{tabular}
\end{center}

\begin{center}
\begin{flushleft}
No. 74: $X_{14,18} \subset \mbP (1,2,3,7,9,11)$; 
$(A^3) = 2/33$, Case 3.
\end{flushleft} \nopagebreak
\begin{tabular}{|p{103pt}|p{112pt}|p{103pt}|>{\centering\arraybackslash}p{26pt}|}
\hline
\parbox[c][\myheight][c]{0cm}{} $\msp_{2,4} = 2 \times \frac{1}{3} (1,1,2)$ & $T \in |2 B|$ & & \\
\hline
\parbox[c][\myheight][c]{0cm}{} $\msp_5 = \frac{1}{11} (1,2,9)$ &  & $X'_{18} \subset (1,2,3,9,4)$ & \\
\hline
\end{tabular}
\end{center}

\begin{center}
\begin{flushleft}
No. 75: $X_{14,18} \subset \mbP (1,2,6,7,8,9)$; 
$(A^3) = 1/24$, Case 1.
\end{flushleft} \nopagebreak
\begin{tabular}{|p{103pt}|p{112pt}|p{103pt}|>{\centering\arraybackslash}p{26pt}|}
\hline
\parbox[c][\myheight][c]{0cm}{} $\msp_{1,2,4} \!=\! 5 \!\times \! \frac{1}{2} (1,1,1)$ &  $9 B + 4 E$, $\{x,s,u\}$ & & \\
\hline
\parbox[c][\myheight][c]{0cm}{} $\msp_{2,5} = \frac{1}{3} (1,1,2)$ &  $4 B + E$, $\{x,y,t\}$ & & \\
\hline
\parbox[c][\myheight][c]{0cm}{} $\msp_4 = \frac{1}{8} (1,1,7)$ &  & Q.I. & \\
\hline
\end{tabular}
\end{center}

\begin{center}
\begin{flushleft}
No. 76: $X_{12,20} \subset \mbP (1,4,5,6,7,10)$; 
$(A^3) = 1/35$, Case 1.
\end{flushleft} \nopagebreak
\begin{tabular}{|p{103pt}|p{112pt}|p{103pt}|>{\centering\arraybackslash}p{26pt}|}
\hline
\parbox[c][\myheight][c]{0cm}{} $\msp_{1,3,5} \!=\! 2 \!\times \! \frac{1}{2} (1,1,1)$ &  $7 B + 3 E$, $\{x,z,t\}$ & & \\
\hline
\parbox[c][\myheight][c]{0cm}{} $\msp_{2,5} = 2 \times \frac{1}{5} (1,1,4)$ & $T \in |4 B|$ & & \\
\hline
\parbox[c][\myheight][c]{0cm}{} $\msp_4 = \frac{1}{7} (1,3,4)$ & & Q.I. & \\
\hline
\end{tabular}
\end{center}

\begin{center}
\begin{flushleft}
No. 77: $X_{16,18} \subset \mbP (1,1,6,8,9,10)$; 
$(A^3) = 1/15$, Case 3.
\end{flushleft} \nopagebreak
\begin{tabular}{|p{103pt}|p{112pt}|p{103pt}|>{\centering\arraybackslash}p{26pt}|}
\hline
\parbox[c][\myheight][c]{0cm}{} $\msp_{2,3,5} = \frac{1}{2} (1,1,1)$ & $T \in |B|$ & & \\
\hline
\parbox[c][\myheight][c]{0cm}{} $\msp_{2,4} = \frac{1}{3} (1,1,2)$ & $T \in |B|$ & & \\
\hline
\parbox[c][\myheight][c]{0cm}{} $\msp_5 = \frac{1}{10} (1,1,9)$ &  & $X'_{18} \subset (1,1,6,9,2)$ & \\
\hline
\end{tabular}
\end{center}

\begin{center}
\begin{flushleft}
No. 78: $X_{16,18} \subset \mbP (1,4,6,7,8,9)$; 
$(A^3) = 1/42$, Case 1.
\end{flushleft} \nopagebreak
\begin{tabular}{|p{103pt}|p{112pt}|p{103pt}|>{\centering\arraybackslash}p{26pt}|}
\hline
\parbox[c][\myheight][c]{0cm}{} $\msp_{1,2,4} \!=\! 2 \!\times \! \frac{1}{2} (1,1,1)$ &  $9 B + 4 E$, $\{x,s,u\}$ & & \\
\hline
\parbox[c][\myheight][c]{0cm}{} $\msp_{1,4} = 2 \times \frac{1}{4} (1,1,3)$ &  $6B + E$, $\{x,z,s,t'\}_{\sharp}$ & & \\
\hline
\parbox[c][\myheight][c]{0cm}{} $\msp_{2,5} = \frac{1}{3} (1,1,2)$ &  $7 B + 2 E$, $\{x,y,s,t\}$ & & \\
\hline
\parbox[c][\myheight][c]{0cm}{} $\msp_3 = \frac{1}{7} (1,1,6)$ &  $B$, $\{x,y,z\}$ & & \\
\hline
\end{tabular}
\end{center}

\begin{center}
\begin{flushleft}
No. 79: $X_{18,20} \subset \mbP (1,4,5,6,9,14)$; 
$(A^3) = 1/42$, Case 1.
\end{flushleft} \nopagebreak
\begin{tabular}{|p{103pt}|p{112pt}|p{103pt}|>{\centering\arraybackslash}p{26pt}|}
\hline
\parbox[c][\myheight][c]{0cm}{} $\msp_{1,3,5} \!=\! 2 \!\times \! \frac{1}{2} (1,1,1)$ &  $9B + 4 E$, $\{x,z,t\}$ & & \\
\hline
\parbox[c][\myheight][c]{0cm}{} $\msp_{3,4} = 2 \times \frac{1}{3} (1,1,2)$ & $14B + 4 E$, $\{x,y,z,u\}$ & & \\
\hline
\parbox[c][\myheight][c]{0cm}{} $\msp_5 = \frac{1}{14} (1,5,9)$ &  & $X'_{20} \subset (1,4,5,9,2)$ & \\
\hline
\end{tabular}
\end{center}

\begin{center}
\begin{flushleft}
No. 80: $X_{18,20} \subset \mbP (1,4,5,7,9,13)$; 
$(A^3) = 2/91$, Case 1.
\end{flushleft} \nopagebreak
\begin{tabular}{|p{103pt}|p{112pt}|p{103pt}|>{\centering\arraybackslash}p{26pt}|}
\hline
\parbox[c][\myheight][c]{0cm}{} $\msp_3 = \frac{1}{7} (1,2,5)$ &  & Q.I. & \\
\hline
\parbox[c][\myheight][c]{0cm}{} $\msp_5 = \frac{1}{13} (1,4,9)$ &  & $X'_{20} \subset (1,4,5,9,2)$ & \\
\hline
\end{tabular}
\end{center}

\begin{center}
\begin{flushleft}
No. 81: $X_{18,20} \subset \mbP (1,5,6,7,9,11)$; 
$(A^3) = 4/231$, Case 1.
\end{flushleft} \nopagebreak
\begin{tabular}{|p{103pt}|p{112pt}|p{103pt}|>{\centering\arraybackslash}p{26pt}|}
\hline
\parbox[c][\myheight][c]{0cm}{} $\msp_{2,4} = \frac{1}{3} (1,1,2)$ & $7B + 2 E$, $\{x,y,s,u\}$ & & \\
\hline
\parbox[c][\myheight][c]{0cm}{} $\msp_3 = \frac{1}{7} (1,2,5)$ &  & E.I. & \\
\hline
\parbox[c][\myheight][c]{0cm}{} $\msp_5 = \frac{1}{11} (1,5,6)$ &  & $X'_{20} \subset (1,5,6,7,2)$ & \\
\hline
\end{tabular}
\end{center}

\begin{center}
\begin{flushleft}
No. 82: $X_{18,22} \subset \mbP (1,2,5,9,11,13)$; 
$(A^3) = 2/65$, Case 2.
\end{flushleft} \nopagebreak
\begin{tabular}{|p{103pt}|p{112pt}|p{103pt}|>{\centering\arraybackslash}p{26pt}|}
\hline
\parbox[c][\myheight][c]{0cm}{} $\msp_2 = \frac{1}{5} (1,1,4)$ & $13 B + E$, $\{x,y,u'\}$ & & \\
\hline
\parbox[c][\myheight][c]{0cm}{} $\msp_5 = \frac{1}{13} (1,2,11)$ & & $X'_{22} \subset (1,2,5,11,4)$ & \\
\hline
\end{tabular}
\end{center}

\begin{center}
\begin{flushleft}
No. 83: $X_{20,21} \subset \mbP (1,3,4,7,10,17)$; 
$(A^3) = 1/34$, Case 2.
\end{flushleft} \nopagebreak
\begin{tabular}{|p{103pt}|p{112pt}|p{103pt}|>{\centering\arraybackslash}p{26pt}|}
\hline
\parbox[c][\myheight][c]{0cm}{} $\msp_{2,4} = \frac{1}{2} (1,1,1)$ & $17 B + 8 E$, $\{x,y,s,u\}$ & & \\
\hline
\parbox[c][\myheight][c]{0cm}{} $\msp_5 = \frac{1}{17} (1,7,8)$ &  & $X'_{21} \subset (1,3,7,10,1)$ & \\
\hline
\end{tabular}
\end{center}

\begin{center}
\begin{flushleft}
No. 84: $X_{18,30} \subset \mbP (1,6,8,9,10,15)$; 
$(A^3) = 1/120$, Case 1.
\end{flushleft} \nopagebreak
\begin{tabular}{|p{103pt}|p{112pt}|p{103pt}|>{\centering\arraybackslash}p{26pt}|}
\hline
\parbox[c][\myheight][c]{0cm}{} $\msp_{1,2,4} \!=\! 2 \!\times \! \frac{1}{2} (1,1,1)$ & $15B + 7 E$, $\{x,s,u\}$ & & \\
\hline
\parbox[c][\myheight][c]{0cm}{} $\msp_{1,3,5} \!=\! 2 \!\times \! \frac{1}{3} (1,1,2)$ & $10 B + 3 E$, $\{x,z,t\}$ & & \\
\hline
\parbox[c][\myheight][c]{0cm}{} $\msp_2 = \frac{1}{8} (1,1,7)$ & $15B + E$, $\{x,y,u\}$ & & \\
\hline
\parbox[c][\myheight][c]{0cm}{} $\msp_{4,5} = \frac{1}{5} (1,1,4)$ & $6 B + E$, $\{x,y,z,s\}$ & & \\
\hline
\end{tabular}
\end{center}

\begin{center}
\begin{flushleft}
No. 85: $X_{24,30} \subset \mbP (1,8,9,10,12,15)$; 
$(A^3) = 1/180$, Case 1.
\end{flushleft} \nopagebreak
\begin{tabular}{|p{103pt}|p{112pt}|p{103pt}|>{\centering\arraybackslash}p{26pt}|}
\hline
\parbox[c][\myheight][c]{0cm}{} $\msp_{1,4} = \frac{1}{4} (1,1,3)$ & $9 B + 2 E$, $\{x,z,s,u\}$ & & \\
\hline
\parbox[c][\myheight][c]{0cm}{} $\msp_{1,3,4} = \frac{1}{2} (1,1,1)$ & $15 B + 7 E$, $\{x,z,u\}$ & & \\
\hline
\parbox[c][\myheight][c]{0cm}{} $\msp_2 = \frac{1}{9} (1,1,8)$ & $15 B + E$, $\{x,y,u\}$ & & \\
\hline
\parbox[c][\myheight][c]{0cm}{} $\msp_{2,4,5} = \frac{1}{3} (1,1,2)$ & $10 B + 3 E$, $\{x,y,s\}$ & & \\
\hline
\parbox[c][\myheight][c]{0cm}{} $\msp_{3,5} = \frac{1}{5} (1,2,3)$ & $6 B + E$, $\{x,y,z,t\}$ & & \\
\hline
\end{tabular}
\end{center}

\end{document}